\documentclass[a4paper,11pt]{amsart}

\usepackage[top=30truemm,bottom=25truemm,left=35truemm,right=35truemm]{geometry}

\usepackage{amsmath}
\usepackage{amssymb}
\usepackage{amsthm}
\usepackage{graphics}
\usepackage[abbrev,alphabetic]{amsrefs}
\usepackage{amscd}
\usepackage[dvipdfmx]{graphicx}
\usepackage{ulem}

\usepackage[all,cmtip]{xy}
\usepackage{xypic}

\title[Inversion of adjunction for quotient singularities II]
{Inversion of adjunction for quotient singularities II: Non-linear actions}

\author{Yusuke Nakamura}
\address{Graduate School of Mathematics, Nagoya University, Furo-cho, Chikusa-ku, Nagoya, 464-8602, Japan.}
\email{y.nakamura@math.nagoya-u.ac.jp}
\urladdr{https://sites.google.com/site/ynakamuraagmath/}

\author{Kohsuke Shibata}
\address{Department of Mathematics, School of Engineering,
Tokyo Denki University, Adachi-ku, Tokyo 120-8551, Japan.}
\email{shibata.kohsuke@mail.dendai.ac.jp}



\newtheorem{thm}{Theorem}[section]
\newtheorem{lem}[thm]{Lemma}

\newtheorem{cor}[thm]{Corollary}
\newtheorem{prop}[thm]{Proposition}

\newtheorem{claim}[thm]{Claim}

\theoremstyle{definition}
\newtheorem{defi}[thm]{Definition}

\newtheorem{conj}[thm]{Conjecture}

\theoremstyle{remark}
\newtheorem{rmk}[thm]{Remark}
\newtheorem{quest}[thm]{Question}
\newtheorem*{ackn}{Acknowledgements}

\begin{document}
\begin{abstract}
We prove the precise inversion of adjunction formula for quotient singularities. 
As an application, we prove the semi-continuity of minimal log discrepancies for hyperquotient singularities. 
This paper is a continuation of \cite{NS}, and we generalize the previous results to non-linear group actions. 
\end{abstract}

\maketitle

\tableofcontents

\section{Introduction}
The minimal log discrepancy is an invariant of singularities defined in birational geometry. 
The importance of this invariant is that two conjectures on the invariant, the LSC (lower semi-continuity) conjecture and 
the ACC (ascending chain condition) conjecture, imply the conjecture of termination of flips \cite{Sho04}. 
This paper is a continuation of \cite{NS}, and we focus on the LSC conjecture. 
We always work over an algebraically closed field $k$ of characteristic zero.

The LSC conjecture is proposed by Ambro \cite{Amb99} and the conjecture predicts that the minimal log discrepancies satisfy the lower semi-continuity. 
\begin{conj}[{LSC conjecture, \cite{Amb99}}]\label{conj:LSC}
Let $(X, \mathfrak{a})$ be a log pair with an $\mathbb{R}$-ideal $\mathfrak{a}$, 
and let $|X|$ be the set of all closed points of $X$ with the Zariski topology. 
Then the function 
\[
|X| \to \mathbb{R}_{\ge 0} \cup \{ - \infty \}; \quad x \mapsto \operatorname{mld}_x (X,\mathfrak{a})
\]
is lower semi-continuous. 
\end{conj}
The LSC conjecture is proved when $\dim X \le 3$ by Ambro \cite{Amb99}. 
Ein, Musta{\c{t}}{\v{a}}, and Yasuda in \cite{EMY03} prove the conjecture when $X$ is smooth. 
Ein and Musta{\c{t}}{\v{a}} in \cite{EM04} generalize the argument to the case where $X$ is a locally complete intersection variety. 
In \cite{Nak16}, the first author proves the conjecture when $X$ has quotient singularities, 
more generally when $X$ has a crepant resolution in the category of the Deligne-Mumford stacks. 
In \cite{NS}, the authors study the conjecture when $X$ has hyperquotient singularities and prove the conjecture for the following $X$: 
\begin{itemize}
\item Suppose that a finite subgroup $G \subset {\rm GL}_N(k)$ acts linearly on $\mathbb{A}_k^N$ freely in codimension one. 
Let $Y := \mathbb{A}_k^N / G$ be the quotient variety. 
Let $X \subset Y$ be a klt subvariety of codimension $c$ that is locally defined by $c$ equations in $Y$. 
\end{itemize}
The main purpose of this paper is to generalize the result in \cite{NS} to non-linear group actions. 
See Definition \ref{defi:qs} for the definition of quotient singularity in this paper. 
\begin{thm}[{$=$\ Theorem \ref{thm:LSC_general}}]\label{thm:LSC}
Let $Y$ be a variety with only quotient singularities. 
Let $X$ be a klt subvariety of $Y$ of codimension $c$ that is locally defined by $c$ equations in $Y$. 
Let $\mathfrak{a}$ be an $\mathbb{R}$-ideal sheaf on $X$. 
Then the function 
\[
|X| \to \mathbb{R}_{\ge 0} \cup \{ - \infty \}; \quad x \mapsto \operatorname{mld}_x (X,\mathfrak{a})
\]
is lower semi-continuous. 
\end{thm}

In this paper, we also treat the PIA (precise inversion of adjunction) conjecture. 

\begin{conj}[PIA conjecture, {\cite[17.3.1]{92}}]\label{conj:PIA}
Let $(X, \mathfrak{a})$ be a log pair and let $D$ be a normal Cartier prime divisor. 
Let $x \in D$ be a closed point. 
Suppose that $D$ is not contained in the cosupport of the $\mathbb{R}$-ideal sheaf $\mathfrak{a}$. 
Then 
\[
\operatorname{mld}_x \bigl( X, \mathfrak{a} \mathcal{O}_X(-D) \bigr) = \operatorname{mld}_x (D, \mathfrak{a} \mathcal{O}_D)
\]
holds. 
\end{conj}

Ein, Musta{\c{t}}{\v{a}}, and Yasuda in \cite{EMY03} prove the PIA conjecture when $X$ is smooth. 
Ein and Musta{\c{t}}{\v{a}} in \cite{EM04} generalize the argument to the case where $X$ is a locally complete intersection variety.
The authors in \cite{NS} prove the conjecture for the following $X$ and $D$: 
\begin{itemize}
\item Suppose that a finite subgroup $G \subset {\rm GL}_N(k)$ acts linearly on $\mathbb{A}_k^N$ freely in codimension one. 
Let $Y := \mathbb{A}_k^N / G$ be the quotient variety, and let $x \in Y$ be the image of the origin. 
Let $X$ be a subvariety of $Y$ of codimension $c$ that has only klt singularities and is locally defined by $c$ equations in $Y$ at $x$. 
Let $D$ be a Cartier prime divisor on $X$ through $x$ with a klt singularity at $x \in D$. 
\end{itemize}
In this paper, this result in \cite{NS} is generalized to non-linear group actions. 

\begin{thm}[{$=$\ Corollary \ref{cor:PIA_general}}]\label{thm:PIA_intro}
Suppose that a variety $Y$ has a quotient singularity at a closed point $x \in Y$. 
Let $X$ be a subvariety of $Y$ of codimension $c$ that is locally defined by $c$ equations at $x$. 
Suppose that $X$ is klt at $x$. 
Let $\mathfrak{a}$ be an $\mathbb{R}$-ideal sheaf on $X$. 
Let $D$ be a prime divisor on $X$ through $x$ that is klt and Cartier at $x$. 
Suppose that $D$ is not contained in the cosupport of the $\mathbb{R}$-ideal sheaf $\mathfrak{a}$.
Then it follows that
\[
\operatorname{mld}_x \bigl( X, \mathfrak{a} \mathcal{O}_X (-D) \bigr) = 
\operatorname{mld}_x (D, \mathfrak{a} \mathcal{O}_D). 
\]
\end{thm}

Due to Theorem \ref{thm:PIA_intro}, 
Theorem \ref{thm:LSC} can be reduced to the known case where $X$ has quotient singularities. 
Hence, this paper is mainly devoted to proving Theorem \ref{thm:PIA_intro}. 
If $X$ has a quotient singularity at a closed point $x \in X$, then 
the completion $\widehat{\mathcal{O}}_{X,x}$ is isomorphic to $k[[x_1, \ldots , x_N]]^G$ for some 
linear group action $G \subset \operatorname{GL}_{N}(k)$. 
Therefore, Theorem \ref{thm:PIA_intro} can be proved by extending the proofs in \cite{NS} to the case of the formal power series ring. 
In what follows, we shall explain the main differences from the proofs in \cite{NS} and 
the difficulties that arise when dealing with formal power series rings. 

The key ingredient of the proofs in \cite{NS} is the theory of arc spaces of quotient singularities 
established by Denef and Loeser in \cite{DL02}. 
Suppose that a finite group $G \subset \operatorname{GL}_{N}(k)$ of order $d$ acts on $\overline{A} = \operatorname{Spec} k[x_1, \ldots , x_N]$. 
Let $\overline{X} \subset \overline{A}$ be a $G$-invariant closed subvariety and 
let $X := \overline{X}/G$ be its quotient. 
Let $I \subset k[x_1, \ldots , x_N]^G$ be the defining ideal of $X \subset \overline{A}/G$. 
For each $\gamma \in G$, $\gamma$ can be diagonalized with some new basis $x^{(\gamma)}_1, \ldots, x^{(\gamma)}_N$. 
Let $\operatorname{diag} (\xi^{e_1}, \ldots , \xi^{e_N})$ be the diagonal matrix with $0 \le e_i \le d-1$, 
where $\xi$ is a primitive $d$-th root of unity in $k$. 
Then we define the ring homomorphism $\overline{\lambda}_{\gamma}^*$ by 
\[
\overline{\lambda}_{\gamma}^*: k[x_1, \ldots , x_N]^G \to k[t][x_1, \ldots , x_N]; \quad x^{(\gamma)}_i \mapsto t^{\frac{e_i}{d}} x^{(\gamma)}_i, 
\] 
and define a $k[t]$-scheme $\overline{X}^{(\gamma)}$ by
\[
\overline{X}^{(\gamma)} = \operatorname{Spec} \bigl( k[t][x_1, \ldots , x_N]/ \overline{I}^{(\gamma)} \bigr), 
\]
where $\overline{I}^{(\gamma)}$ is the ideal generated by the elements of $\overline{\lambda}_{\gamma}^* (I)$. 
Then the theory of Denef and Loeser in \cite{DL02} 
allows us to compare the spaces $X_{\infty}$ and $\bigsqcup _{\gamma \in G} \overline{X}^{(\gamma)}_{\infty}$. 
In \cite{NS}, using this theory, $X_{\infty}$ is studied through each $\overline{X}^{(\gamma)}_{\infty}$. 

In this paper, we deal with the case of formal power series rings, 
i.e.\ when $\overline{A} = \operatorname{Spec} k[[x_1, \ldots , x_N]]$. 
Let $G$, $\gamma$, $\overline{X}$, and $X$ be as above. 
In this case, we take $e_i$'s above to satisfy $1 \le e_i \le d$. 
Then we can define the following two natural ring homomorphisms
\begin{alignat*}{2}
\overline{\lambda}_{\gamma}^* &: k[[x_1, \ldots , x_N]]^G \to k[x_1, \ldots , x_N][[t]];& 
\quad &x^{(\gamma)}_i \mapsto t^{\frac{e_i}{d}} x^{(\gamma)}_i, \\
\overline{\lambda}_{\gamma}^{\prime *} &: k[[x_1, \ldots , x_N]]^G \to k[t][[x_1, \ldots , x_N]];& 
\quad &x^{(\gamma)}_i \mapsto t^{\frac{e_i}{d}} x^{(\gamma)}_i, 
\end{alignat*}
and we define $k[t]$-schemes
\[
\overline{X}^{(\gamma)} = \operatorname{Spec} \bigl( k[x_1, \ldots , x_N][[t]]/ \overline{I}^{(\gamma)} \bigr), \quad 
\overline{X}^{\prime (\gamma)} = \operatorname{Spec} \bigl( k[t][[x_1, \ldots , x_N]]/ \overline{I}^{\prime (\gamma)} \bigr), 
\]
where $\overline{I}^{\prime (\gamma)}$ is the ideal generated by $\overline{\lambda}_{\gamma}^{\prime *} (I)$. 
In this paper, we will use both arc spaces 
$\overline{X}^{(\gamma)}_{\infty}$ and $\overline{X}^{\prime (\gamma)}_{\infty}$ to study $X_{\infty}$. 
We shall also explain below how to use $\overline{X}^{(\gamma)}_{\infty}$ and $\overline{X}^{\prime (\gamma)}_{\infty}$ differently.

First, the theory of Denef and Loeser in \cite{DL02} can be generalized to the formal power series rings, 
and $X_{\infty}$ can be compared with $\overline{X}^{(\gamma)}_{\infty}$. 
Indeed, we shall see in Proposition \ref{prop:DL_X} that 
$\overline{\lambda}_{\gamma}^*$ gives a map $\bigsqcup _{\gamma \in G} \overline{X}^{(\gamma)}_{\infty} \to X_{\infty}$ 
that is surjective outside a thin set. 
We note that it is not enough to consider $\overline{X}^{\prime (\gamma)}_{\infty}$ in this respect (see Remark \ref{rmk:RF} for the detail).

On the other hand, $\overline{X}^{\prime (\gamma)}_{\infty}$ will be used with the following motivation. 
In \cite{NS}, a $k$-arc $\beta \in \overline{X}^{(\gamma)}_{\infty}$ is called the ``trivial arc" when it corresponds 
to the $k[t]$-ring homomorphism $\beta ^* : k[t][x_1, \ldots , x_N] \to k[[t]]$ satisfying $\beta ^* (x_i) = 0$ for each $i$. 
Another key point of the argument in \cite{NS} is to show the fact that 
the trivial arc always has a lift on a resolution $W$ of $\overline{X}^{(\gamma)}$. 
The existence of such a lift is proved by combining the result by Graber, Harris, and Starr \cite{GHS03} and 
the rational chain connectedness of the fibers of the resolution proved by Hacon and Mckernan \cite{HM07}. 
In our formal power series ring setting, 
this argument does not work directly on $\overline{X}^{(\gamma)}$ because each closed fiber 
$\overline{X}^{(\gamma)} \to \operatorname{Spec} k[t]$ over $t = a \in k^{\times}$ is empty. 
Whereas, the same argument works on $\overline{X}^{\prime (\gamma)}$ and proves that 
the trivial arc has a lift on a desingularization $W'$ of $\overline{X}^{\prime (\gamma)}$ (see Claim \ref{claim:not_thin_R}). 
It should be noted that it is not clear whether the results \cite{GHS03} and \cite{HM07} can be applied to the formal power series ring setting. 
However, this difficulty can be avoided by using the functorial desingularization due to Temkin \cite{Tem12}. 

A large part of this paper (Sections \ref{section:Omega'}, \ref{section:jet} and \ref{section:ktjet}) is devoted to proving and listing
the basic facts for dealing with the arc spaces $\overline{X}^{(\gamma)}_{\infty}$ and $\overline{X}^{\prime (\gamma)}_{\infty}$. 
Firstly, $\overline{X}^{(\gamma)}_{\infty}$ can be seen as the arc space (Greenberg scheme) of a formal $k[[t]]$-scheme, and 
the theory developed by Sebag in \cite{Seb04} can be applied (cf.\ \cite{CLNS, Yas}). 
On the other hand, as far as we know, the arc spaces of the form $\overline{X}^{\prime (\gamma)}$ have not been fully discussed so far. 
In Subsection \ref{subsection:ktjet}, we discuss the basic facts on the arc spaces of $k[t][[x_1, \ldots , x_N]]$-schemes of finite type. 
Furthermore, we discuss in Section \ref{section:Omega'} the theory of sheaves of special differentials introduced 
by de Fernex, Ein, and Musta{\c{t}}{\v{a}} in \cite{dFEM11}, 
and the theory of derivations which are needed in Section \ref{section:ktjet}. 
The theory of the arc spaces of $k[t][[x_1, \ldots , x_N]]$-schemes has the following technical difficulties 
(see Remark \ref{rmk:quest} for the detail). 
For a $k$-variety $X$, it is almost trivial that $Z_{\infty}$ is a thin set of $X_{\infty}$ 
for the closed subscheme $Z \subset X$ defined by the Jacobian ideal $\operatorname{Jac}_{X}$. 
This fact is also valid for $k[t]$-schemes $X$ of finite type dealt with in \cite{NS} and 
for formal $k[[t]]$-schemes of finite type dealt with in \cite{Seb04}. 
However, it is not clear to us whether the same statement is valid for $k[t][[x_1, \ldots , x_N]]$-schemes of finite type. 
For avoiding this difficulty, many propositions in Subsection \ref{subsubsection:thin} 
are proved under stronger assumptions.

The paper is organized as follows. 
In Section \ref{section:Omega'}, 
we discuss the theory of sheaves of special differentials introduced by de Fernex, Ein, and Musta{\c{t}}{\v{a}} in \cite{dFEM11} 
and the theory of derivations. 
In Section \ref{section:lp}, we review some definitions on log pairs. 
In Section \ref{section:jet}, we discuss the theory of arc spaces of $k[[x_1, \ldots , x_N]]$-schemes and see that 
the formula in \cite{EMY03} and \cite{EM09} representing the minimal log discrepancies of $k$-varieties in terms of arc spaces 
can be generalized to the formal power series ring setting (Theorem \ref{thm:mld_R}). 
In Section \ref{section:ktjet}, we discuss the theory of arc spaces for $k[t][[x_1, \ldots , x_N]]$-schemes and affine formal $k[[t]]$-schemes. 
In Section \ref{section:DL}, we review the theory of arc spaces of quotient varieties established 
by Denef and Loeser in \cite{DL02} in the formal power series ring setting. 
In Section \ref{section:mld_R}, we discuss the minimal log discrepancies of quotient singularities and 
describe them by the codimension of cylinders in arc spaces of the $k[t]$-schemes using the theories in Sections \ref{section:ktjet} and \ref{section:DL}. 
In Section \ref{section:PIA}, we prove the PIA conjecture for quotient singularities, where the group action may be non-linear (Theorem \ref{thm:PIA_nonlin}). 
In Section \ref{section:maintheorem}, we prove the main theorems Corollary \ref{cor:PIA_general} and Theorem \ref{thm:LSC_general}.

\begin{ackn} 
The first author is partially supported by JSPS KAKENHI No.\ 18K13384, 16H02141, 17H02831, 18H01108, and JPJSBP120219935. 
The second author is partially supported by the Grant-in-Aid for Young Scientists (KAKENHI No.\ 19K14496).
\end{ackn}

\section*{Notation}

\begin{itemize}
\item 
We basically follow the notations and the terminologies in \cite{Har77} and \cite{Kol13}.

\item 
Throughout this paper, $k$ is an algebraically closed field  of characteristic zero. 
We say that $X$ is a \textit{variety over} $k$ or a \textit{$k$-variety} if 
$X$ is an integral scheme that is separated and of finite type over $k$. 
\end{itemize}

\section{Sheaves of special differentials}\label{section:Omega'}
Let $R_0$ be a ring. 
In this section, following \cite[Appendix A]{dFEM11}, 
we define the sheaf $\Omega' _{X/R_0}$ of special differentials for a scheme $X$ of finite type over 
$\operatorname{Spec} R_0[[x_1, \ldots , x_N]]$. 
In \cite[Appendix A]{dFEM11}, the sheaf $\Omega' _{X/R_0}$ of special differentials is defined for $R_0 = k$. 
This definition can be generalized to an arbitrary ring $R_0$. 
We are interested in the case where $R_0 = k$ or $R_0 = k[t]$ for our later use.

Let $R = R_0[[x_1, \ldots , x_N]]$.  

\begin{defi}[{cf.\ \cite[Appendix A]{dFEM11}}]
\begin{enumerate}
\item Let $M$ be an $R$-module. 
Then an $R_0$-derivation $D: R \to M$ is called a \textit{special $R_0$-derivation} if $D$ satisfies 
\[
D(f) = \sum _{i=1} ^N \frac{\partial f}{\partial x_i} D(x_i)
\]
for any $f \in R$. 

\item For an $R$-algebra $A$ and an $A$-module $M$, an $R_0$-derivation $D: A \to M$ is called a \textit{special $R_0$-derivation} if 
its restriction to $R$ is a special $R_0$-derivation. 
We denote by $\operatorname{Der}'_{R_0} (A,M)$ the set of all special $R_0$-derivations. 
Then $\operatorname{Der}'_{R_0} (A,M)$ is an $A$-submodule of $\operatorname{Der}_{R_0} (A,M)$. 
\end{enumerate}
\end{defi}

\begin{lem}\label{lem:D=D'}
Let $M$ be an $A$-module that is separated in the $(x_1, \ldots, x_N)$-adic topology, 
i.e.\ $M$ satisfies $\bigcap _{\ell \ge 1}(x_1, \ldots, x_N)^{\ell}M = 0$. 
Then 
\[
\operatorname{Der}'_{R_0} (A,M) = \operatorname{Der}_{R_0} (A,M)
\] 
holds. 
In particular, it holds in the following two cases. 
\begin{enumerate}
\item[(1)] When $M$ is a finite $R$-module. 
\item[(2)] When $M=A$ and $A$ is an integral domain such that $(x_1, \ldots, x_N)A \not = A$. 
\end{enumerate}
\end{lem}
\begin{proof}
By the definition of special derivations, it is sufficient to show that 
$\operatorname{Der}'_{R_0} (R,M) = \operatorname{Der}_{R_0} (R,M)$. 
Let $D \in \operatorname{Der}_{R_0} (R,M)$. If we set $D': R \to M$ by 
$D'(f):= \sum _{i=1} ^N \frac{\partial f}{\partial x_i} D(x_i)$, then $D' \in \operatorname{Der}'_{R_0} (R,M)$ holds. 
Hence it is sufficient to show that $D=0$ if $D(x_i) = 0$ holds for all $i$. 
Let $f \in R$. 
For any $\ell \ge 0$, we may write $f = f_1 + f_2$ with $f_1 \in R_0[x_1, \ldots, x_n]$ 
and $f_2 \in (x_1, \ldots, x_n)^{\ell + 1}R$. 
Then we have
\[
D(f) = D(f_1) + D(f_2) = D(f_2) \in (x_1, \ldots, x_n)^{\ell}M. 
\]
It shows that $D(f) \in \bigcap _{\ell \ge 1}(x_1, \ldots, x_N)^{\ell}M = 0$. 

In both cases (1) and (2), $M$ is separated by \cite[Theorem 8.9]{Mat89}. 
\end{proof}

\begin{prop}[{cf.\ \cite[Appendix A]{dFEM11}}] \label{prop:spOmega}
For any $R$-algebra $A$,
there exists an $A$-module $\Omega' _{A/R_0}$ with a special $R_0$-derivation $d'_{A/R_0}: A \to \Omega' _{A/R_0}$ such that 
the induced map
\[
\operatorname{Hom}_{A} (\Omega' _{A/R_0}, M) \to \operatorname{Der}'_{R_0} (A,M); \quad f \mapsto f \circ d'_{A/R_0}
\]
is an isomorphism for any $A$-module $M$. 
\end{prop}
\begin{proof}
The same proof as in \cite[A.1-A.4]{dFEM11} works. 
\end{proof}
An $A$-module $\Omega' _{A/R_0}$ satisfying Proposition \ref{prop:spOmega} 
is unique up to an isomorphism commuting with $d'_{A/R_0}$, 
and $\Omega' _{A/R_0}$ is called the \textit{module of special differentials}. 
We sometimes abbreviate $d'_{A/R_0}$ to $d'$ when no confusion can arise. 
We note that $\Omega' _{A/R_0}$ depends on the choice of $R$. 
We list some properties on $\Omega' _{A/R_0}$ from \cite[Appendix A]{dFEM11}. 

\begin{prop}[{cf.\ \cite[Appendix A]{dFEM11}}] \label{prop:spOmega2}
\begin{enumerate}
\item \label{item:basis} 
If $A = R[y_1, \ldots, y_m]$, then $\Omega' _{A/R_0}$ is a free $A$-module of rank $N+m$ with basis 
\[
d'_{A/R_0}(x_1), \ldots , d'_{A/R_0}(x_N), d'_{A/R_0}(y_1), \ldots , d'_{A/R_0}(y_m). 
\]

\item \label{item:ex1} 
Let $f: A \to B$ be a homomorphism of $R$-algebras. Then we have an exact sequence
\[
\Omega' _{A/R_0} \otimes _{A} B \xrightarrow{\alpha} \Omega' _{B/R_0} \xrightarrow{\beta} \Omega _{B/A} \to 0
\]
of $B$-modules, where the maps $\alpha$ and $\beta$ are defined by
$\alpha \bigl( d'_{A/R_0}(g) \otimes 1 \bigr) = d'_{B/R_0}(f(g))$ for $g \in A$ and 
$\beta \bigl( d'_{B/R_0}(g) \bigr) = d_{B/A}(g)$ for $g \in B$. 

\item \label{item:ex2} 
Let $f: A \to B$ be a surjective homomorphism of $R$-algebras with $I = \operatorname{Ker} f$. Then we have an exact sequence 
\[
I/I^2 \xrightarrow{\delta} \Omega' _{A/R_0} \otimes _{A} B \xrightarrow{\alpha} \Omega' _{B/R_0} \to 0
\]
of $B$-modules, where the map $\delta$ is defined by $\delta(\overline{g}) = d'_{A/R_0}(g) \otimes 1$ for $g \in I$. 

\item \label{item:localization} 
If $S \subset A$ is a multiplicative system of $A$, then we have a canonical isomorphism
\[
\Omega' _{S^{-1}A/R_0} \simeq \Omega' _{A/R_0} \otimes _A S^{-1}A. 
\]
\end{enumerate}
\end{prop}
\begin{proof}
The same proofs as in \cite[Lemmas A.1-3, A.6-7]{dFEM11} work. 
The exact sequences in (2) and (3) are derived from the following corresponding exact sequences
\begin{align*}
0 \to \operatorname{Der} _A (B,M) \to \operatorname{Der}' _{R_0} (B,M) \to \operatorname{Der}' _{R_0} (A,M), \hspace{3mm} \\
0 \to \operatorname{Der}' _{R_0} (B,M) \to \operatorname{Der}' _{R_0} (A,M) \to \operatorname{Hom} _B (I/I^2,M), 
\end{align*}
for any $B$-module $M$. 
The isomorphism in (4) is derived from the following isomorphisms
\begin{align*}
\operatorname{Hom} _{S^{-1}A} (\Omega' _{A/R_0} \otimes _A S^{-1}A,M) 
&\simeq \operatorname{Hom} _{A} (\Omega' _{A/R_0},M) \\
&\simeq \operatorname{Der}' _{R_0} (A,M) \simeq \operatorname{Der}' _{R_0} (S^{-1}A,M)
\end{align*}
for any $S^{-1}A$-module $M$. 
\end{proof}

\begin{rmk}\label{rmk:UFMD}
\begin{enumerate}
\item 
The usual module $\Omega_{A/R_0}$ of differentials is not a finite $A$-module in general 
when $A$ is an $R$-algebra of finite type. 
However, the module $\Omega' _{A/R_0}$ of special differentials becomes a finite $A$-module and has similar properties as 
the module of differentials defined for $R_0$-algebras of finite type.

\item The universal finite module of differentials is also a module defined with the same motivation (see \cite[Section 11]{Kun86}). 
For an $R_0$-algebra $A$, a finite $A$-module $\widetilde{\Omega}_{A/R_0}$ 
with an $R_0$-derivation $\widetilde{d}_{A/R_0}:A \to \widetilde{\Omega}_{A/R_0}$ is called 
the \textit{universal finite module of differentials} 
if it satisfies the following universal property. 
\begin{itemize}
\item For any $R_0$-derivation $D:A \to M$ to a finite $A$-module $M$, there exists a unique homomorphism 
$\alpha: \widetilde{\Omega}_{A/R_0} \to M$ of $A$-modules satisfying $D = \alpha \circ \widetilde{d}_{A/R_0}$. 
\end{itemize}
In other words, $\widetilde{\Omega}_{A/R_0}$ and $\widetilde{d}_{A/R_0}$ satisfy
\[
\operatorname{Hom}_{A} (\widetilde{\Omega}_{A/R_0},M) \xrightarrow{\simeq} \operatorname{Der}_{R_0} (A,M); 
\quad \alpha \mapsto \alpha \circ \widetilde{d}_{A/R_0}
\]
for any finite $A$-module $M$. 

In contrast to the module $\Omega' _{A/R_0}$ of special differentials, 
the universal finite module $\widetilde{\Omega} _{A/R_0}$ of differentials does not necessarily exist. 

\item 
$A$ is called an \textit{analytic $R_0$-algebra} 
if there exists $R = R_0[[x_1, \ldots , x_N]]$ for some $N \ge 0$ such that $A$ is a finite $R$-algebra. 
If $A$ is an analytic $R_0$-algebra, 
then the universal finite module $\widetilde{\Omega} _{A/R_0}$ of differentials exists. 
Furthermore, if $A$ is a finite $R_0[[x_1, \ldots , x_N]]$-algebra and 
if $\Omega' _{A/R_0}$ is the module of special differentials with respect to $R = R_0[[x_1, \ldots , x_N]]$, then we have 
$\widetilde{\Omega} _{A/R_0} \simeq \Omega' _{A/R_0}$. This is because we have
\[
\operatorname{Der}' _{R_0} (A,M) = \operatorname{Der}_{R_0} (A,M)
\]
for any finite $A$-module $M$ by Lemma \ref{lem:D=D'} in this case. 
Therefore, we can also see that $\Omega' _{A/R_0}$ does not depend on the choice of $R$ 
as long as $A$ is finite as an $R$-module. 

\item Even if $A$ is an algebra of finite type over $R = R_0[[x_1, \ldots , x_N]]$, 
$\widetilde{\Omega} _{A/R_0}$ does not necessarily exist despite the fact that $\Omega' _{A/R_0}$ is a finite $A$-module. 
We shall see that $\widetilde{\Omega} _{A/k}$ and $\widetilde{\Omega} _{B/k}$ do not exist for 
\[
A = k[[x]][y], \qquad B = k[[x]][y]/(1-xy) \simeq k((x)). 
\]
In fact, 
since 
\begin{align*}
\dim _{k((x))} \operatorname{Der}_{k} \bigl( k((x)) \bigr) &= 
\dim _{k((x))} \operatorname{Hom}_{k((x))} \bigl( \Omega_{k((x))/k}, k((x)) \bigr)\\
&= \operatorname{trdeg} _k k((x)) = \infty, 
\end{align*}
$\operatorname{Der}_{k} (B)$ is not a finite $B$-module and hence $\widetilde{\Omega} _{B/k}$ does not exist. 
Furthermore, since there is a natural injective map $\operatorname{Der}_{k} (B, B) \to \operatorname{Der}_{k} (A, B)$, 
$\operatorname{Der}_{k} (A, B)$ is not a finite $A$-module and hence $\widetilde{\Omega} _{A/k}$ does not exist 
(cf.\ \cite[Corollary 11.10]{Kun86}). 
\end{enumerate}
\end{rmk}

\begin{rmk}\label{rmk:JC}
\begin{enumerate}
\item Let $A$ be a ring. 
For a non-negative integer $\ell$, and for subsets $F \subset A$ and $\Delta \subset \operatorname{Der} (A)$, we denote by 
$\mathcal{J}_{\ell}(F; \Delta)$ the ideal of $A$ generated by the determinants $\det \bigl( D_i(f_j) \bigr)_{1 \le i,j \le \ell}$ of 
all the matrices $\bigl( D_i(f_j) \bigr)_{1 \le i,j \le \ell}$ of size $\ell$ with $D_i \in \Delta$ and $f_j \in F$. 

If $I$ is an ideal of $A$ generated by $f_1, \ldots , f_{t}$ and the $A$-submodule $A \Delta $ of $\operatorname{Der} (A)$ 
is generated by $D_1, \ldots , D_s$ as an $A$-module, then we have
\[
\mathcal{J}_{\ell}(I; \Delta) +I = \mathcal{J}_{\ell}( \{f_1, \ldots , f_t \}; \{D_1, \ldots , D_s \}) +I. 
\]

\item Let $A$ be a regular ring and let $P$ be a prime ideal. For an ideal $I$ of $A$ such that $I \subset P$, 
the following hold (cf.\ \cite[Theorem 30.4]{Mat89}). 
\begin{enumerate}
\item 
$\mathcal{J}_{\ell}(I; \operatorname{Der}(A)) \subset P$
holds for any $\ell > \operatorname{ht}(IA_P)$.

\item \label{item:JCI} $A_P/IA_P$ is regular if 
$\mathcal{J}_{\ell}(I; \operatorname{Der}(A)) \not \subset P$
holds for $\ell = \operatorname{ht}(IA_P)$. 
\end{enumerate}
Some regular rings satisfy the inverse implication of (b), and such rings are said to satisfy 
the \textit{weak Jacobian condition} (WJ) (cf.\ \cite[Section 30]{Mat89}). 
Rings of finite type over $k$ are classically known to satisfy (WJ), and this is known as the Jacobian criterion for regularity. 
Matsumura proved in \cite{Mat77} that $R$-algebras of finite type satisfy (WJ)$_k$ when $R = k[[x_1, \ldots , x_N]]$ 
(see \cite[Theorem 9]{Mat77} for more general result): 
\begin{enumerate}
\item[(c)] Let $A = k[[x_1, \ldots , x_N]][y_1, \ldots, y_m]$ and 
let $P$ and $Q$ be prime ideals of $A$ such that $Q \subset P$. 
Then $A/Q$ is regular at $P$ if and only if 
$\mathcal{J}_{\ell}(Q; \operatorname{Der}_k(A)) \not \subset P$ holds for $\ell = \operatorname{ht} Q$. 
\end{enumerate}
Note that $\operatorname{Der}_k(A) = \operatorname{Der}'_k(A)$ holds for $A = k[[x_1, \ldots , x_N]][y_1, \ldots, y_m]$ 
(cf.\ Lemma \ref{lem:D=D'}), and 
this is a free $A$-module generated by $\frac{\partial}{\partial x_i}$'s and $\frac{\partial}{\partial y_i}$'s. 
\end{enumerate}
\end{rmk}

In \cite[Proposition A.8]{dFEM11}, the local freeness of $\Omega '_{A/k}$ is proved for regular rings $A$ when $R_0=k$. 

\begin{prop}[{\cite[Proposition A.8]{dFEM11}}]\label{prop:spOmega3}
Suppose that $R_0 = k$. 
Let $A$ be an $R$-algebra of finite type, and let $\mathfrak{q}$ be a prime ideal of $A$. 
If $A_{\mathfrak{q}}$ is regular, then $\Omega '_{A_{\mathfrak{q}}/k}$ 
is a free $A_{\mathfrak{q}}$-module of rank $\dim A_{\mathfrak{q}} + \dim _{k(\mathfrak{q})} (\Omega '_{k(\mathfrak{q})/k})$. 
\end{prop}

In the proof of \cite[Proposition A.8]{dFEM11}, the following statement is proved. 

\begin{prop}[{\cite[Proposition A.8]{dFEM11}}]\label{prop:spOmega4}
Suppose that $R_0 = k$. 
Let $S := R[y_1, \ldots , y_m]$, and let $P$ and $Q$ be prime ideals of $S$ with $P \subset Q$. 
Let $A := S/P$ and let $\mathfrak{q}$ be the prime ideal of $A$ corresponding to $Q$. 
If $A_{\mathfrak{q}}$ is regular, then the sequence 
\[
0 \to PS_Q/P^2S_Q \to \Omega'_{S/k} \otimes _S A_{\mathfrak{q}} \to \Omega '_{A_{\mathfrak{q}}/k} \to 0
\]
obtained by Proposition \ref{prop:spOmega2}(\ref{item:ex2}) is exact and splits. 
\end{prop}

\begin{rmk}\label{rmk:coht}
\begin{enumerate}
\item 
By Propositions \ref{prop:spOmega3} and \ref{prop:spOmega4}, we have
\[
\dim S - \operatorname{ht} P =  \dim A_{\mathfrak{q}} + \dim _{k(\mathfrak{q})} (\Omega '_{k(\mathfrak{q})/k}). 
\]
It shows that $\operatorname{coht} P = \dim S - \operatorname{ht} P$ is independent of the choice of $S$. 
Lemma \ref{lem:dim'}(1) below also proves this independence. 
Lemma \ref{lem:dim'}(1) also gives a ring-theoretic interpretation of this value without using $\Omega'$. 

\item 
In Proposition \ref{prop:spOmega3}, 
we note that the rank of $\Omega '_{A_{\mathfrak{q}}/k}$ is not equal to $\dim A$ in general. 

If we set $S=k[[x]][y]$ and $P=Q=(xy-1)$ in Proposition \ref{prop:spOmega4}, 
then we have $A_{\mathfrak{q}} = A = k((x))$. 
(1) shows that the rank of $\Omega '_{A_{\mathfrak{q}}/k}$ is equal to one even though $\dim A = 0$. 
\end{enumerate}
\end{rmk}

\begin{defi}\label{defi:dim'}
Suppose that $R_0$ is a Noetherian domain. Let $X$ be an irreducible scheme of finite type over $R$ and 
let $X_{\rm red}$ be its underlying reduced subscheme. 
Let $X_{\rm red} \to \operatorname{Spec} R$ be the structure morphism, 
and let $\mathfrak{p} \in \operatorname{Spec} R$ be the image of the generic point of $X_{\rm red}$. 
Then we define 
\[
\dim' X := \operatorname{trdeg}_{k(\mathfrak{p})} K(X_{\rm red}) + \dim R - \operatorname{ht} \mathfrak{p}, 
\]
where $k(\mathfrak{p}) := R_{\mathfrak{p}}/\mathfrak{p}R_{\mathfrak{p}}$ and $K(X_{\rm red})$ is the function field 
of $X_{\rm red}$. 
When $X = \operatorname{Spec} A$ is an affine scheme, we also write $\dim ' A := \dim ' X$. 
\end{defi}

\begin{lem}\label{lem:dim'}
Suppose that $R_0$ is a Noetherian domain. Let $X = \operatorname{Spec} A$ be an irreducible affine scheme of finite type over $R$. 
Then the following hold. 
\begin{enumerate}
\item If $A$ is a domain, then $\dim S - \operatorname{ht} P = \dim ' A$ holds 
for any representation $A \simeq S/P$ with $S = R[y_1, \ldots, y_m]$ and a prime ideal $P$ of $S$. 

\item $\dim A \le \dim ' A$ holds. 

\item Suppose that $R_0 = k$ and $A$ is a domain. Then $\dim ' A = \dim _{K} (\Omega ' _{K/k})$ holds for $K = \operatorname{Frac} A$. 

\item Suppose that $R_0 = k$ or $R_0 = k[t]$. 
Then $\dim' A = \dim A$ holds if $A/\mathfrak{m} = k$ holds for some maximal ideal $\mathfrak{m}$ of $A$. 

\item Suppose that $R$ is a universally catenary ring. 
If $I = (f_1, \ldots, f_c)$ is an ideal of $A$ generated by $c$ elements, 
then $\dim ' (A/Q) \ge \dim' A - c$ holds for any minimal prime $Q$ of $I$. 

\item Suppse that $R$ is a universally catenary ring. 
Let $\mathfrak{p}$ and $\mathfrak{q}$ be prime ideals of $A$ such that $\mathfrak{p} \subset \mathfrak{q}$. 
Let $B := A / \mathfrak{p}$ and $\overline{\mathfrak{q}} := \mathfrak{q}/\mathfrak{p} \in \operatorname{Spec} B$. 
Then we have $\operatorname{ht} \mathfrak{p} 
= \dim A_{\mathfrak{q}} - \dim B_{\overline{\mathfrak{q}}} 
= \dim ' A - \dim ' B$. 
\end{enumerate}
\end{lem}
\begin{proof}
First, we prove (1). Let $\mathfrak{p}$ be the image of $P$ according to the map $\operatorname{Spec} S \to \operatorname{Spec} R$. 
Then by \cite[Theorem 15.5]{Mat89} (and the definition below in \cite{Mat89}), it follows that 
\[
\operatorname{trdeg}_{\operatorname{Frac} R} (\operatorname{Frac} S) - \operatorname{ht} P = 
\operatorname{trdeg}_{k(\mathfrak{p})} (\operatorname{Frac} (S/P)) - \operatorname{ht} \mathfrak{p}. 
\]
Therefore, the assertion follows from $\operatorname{trdeg}_{\operatorname{Frac} R} (\operatorname{Frac} S) = m = \dim S - \dim R$. 

(2) follows from (1) and the inequality $\dim (S/P) + \operatorname{ht} P \le \dim S$. 
(3) follows from (1) and Remark \ref{rmk:coht}(1). 

Next, we prove (4) for $R_0 = k[t]$ (the case when $R_0 = k$ is similar). We may assume that $A$ is a domain. Take $S$ and $P$ as in (1). 
Then it is sufficient to prove 
\[
\dim S - \operatorname{ht} P = \dim (S/P). 
\]
Let $M$ be the maximal ideal of $S$ corresponding to $\mathfrak{m}$. 
Since $S/M = k$, $M$ is of the form 
\[
M = (t-a, x_1, \ldots, x_N, y_1 - b_1, \ldots , y_m - b_m)
\]
with $a, b_i \in k$. 
Therefore we have $\dim S = N + m + 1 = \operatorname{ht} M$. Since $S$ is a catenary ring, we also have 
$\dim (S/P) = \operatorname{ht} M - \operatorname{ht} P$, which proves $\dim S - \operatorname{ht} P = \dim (S/P)$. 

(5) follows from (1) and Krull's height theorem. (6) also follows from (1). 
\end{proof}

\begin{prop}\label{prop:spOmega_exact}
Suppose that $R_0 = k$. 
Let $A$ be a domain of finite type over $R$, and 
let $\mathfrak{q}_1$ and $\mathfrak{q}_2$ be prime ideals of $A$ such that $\mathfrak{q}_1 \subset \mathfrak{q}_2$. 
Let $B := A / \mathfrak{q}_1$ and $\overline{\mathfrak{q}}_2 := \mathfrak{q}_2 / \mathfrak{q}_1 \in \operatorname{Spec} B$. 
If both $A_{\mathfrak{q}_2}$ and $B_{\overline{\mathfrak{q}}_2}$ are regular, then 
the sequence 
\[
0 \to \mathfrak{q}_1 A_{\mathfrak{q}_2}/\mathfrak{q}_1 ^2 A_{\mathfrak{q}_2} 
\to \Omega ' _{A_{\mathfrak{q}_2}/k} \otimes _{A_{\mathfrak{q}_2}} B_{\overline{\mathfrak{q}}_2} 
\to \Omega ' _{B_{\overline{\mathfrak{q}}_2}/k}
\to 0
\]
induced by Proposition \ref{prop:spOmega4}(\ref{item:ex2}) is exact and splits. 
\end{prop}

\begin{proof}
By Proposition \ref{prop:spOmega4}(\ref{item:ex2}), the sequence
\[
\mathfrak{q}_1 A_{\mathfrak{q}_2}/\mathfrak{q}_1 ^2 A_{\mathfrak{q}_2} 
\to \Omega ' _{A_{\mathfrak{q}_2}/k} \otimes _{A_{\mathfrak{q}_2}} B_{\overline{\mathfrak{q}}_2} 
\xrightarrow{\delta} \Omega ' _{B_{\overline{\mathfrak{q}}_2}/k}
\to 0
\]
is exact. Since $A_{\mathfrak{q}_2}$ and $B_{\overline{\mathfrak{q}}_2}$ are regular, 
it follows from Proposition \ref{prop:spOmega3} that
\[
\Omega ' _{A_{\mathfrak{q}_2}/k} \otimes _{A_{\mathfrak{q}_2}} B_{\overline{\mathfrak{q}}_2} \simeq B_{\overline{\mathfrak{q}}_2} ^{\oplus \dim' A}, \quad 
\Omega ' _{B_{\overline{\mathfrak{q}}_2}/k} \simeq B_{\overline{\mathfrak{q}}_2} ^{\oplus \dim' B}. 
\]
Therefore $\operatorname{Ker} (\delta)$ is a free $B_{\overline{\mathfrak{q}}_2}$-module of rank equal to $\dim ' A - \dim' B$. 
On the other hand, since $A_{\mathfrak{q}_2}$ and $B_{\overline{\mathfrak{q}}_2}$ are regular, 
$\mathfrak{q}_1 A_{\mathfrak{q}_2} / \mathfrak{q}_1 ^2 A_{\mathfrak{q}_2}$ 
is also a free $B_{\overline{\mathfrak{q}}_2}$-module of rank equal to $\operatorname{ht} ( \mathfrak{q}_1 )$, 
which is equal to $\dim ' A - \dim' B$ by Lemma \ref{lem:dim'}(6). 
Hence, the induced surjective map 
$\mathfrak{q}_1 A_{\mathfrak{q}_2}/\mathfrak{q}_1 ^2 A_{\mathfrak{q}_2} \to \operatorname{Ker} (\delta)$ 
should be an isomorphism. We complete the proof. 
\end{proof}

\begin{defi}\label{defi:omega'}
\begin{enumerate}
\item 
Let $X$ be a scheme over $\operatorname{Spec} R$. 
Then due to Proposition \ref{prop:spOmega2}(\ref{item:localization}), 
there exists a quasi-coherent sheaf $\Omega' _{X/R_0}$ satisfying 
$\Omega' _{X/R_0}(U) \simeq \Omega' _{\mathcal{O}_X(U)/R_0}$ for any affine open subset $U \subset X$. 
Note that $\Omega' _{X/R_0}$ is coherent by Proposition \ref{prop:spOmega2}(\ref{item:basis})(\ref{item:ex2}) when $X$ is of finite type over $\operatorname{Spec} R$. 
The sheaf $\Omega' _{X/R_0}$ is called the \textit{sheaf of special differentials}. 
We denote $\Omega'^n _{X/R_0} := \bigwedge ^n \Omega' _{X/R_0}$ for a non-negative integer $n$. 

\item 
Suppose that $R_0 = k$ and $X$ is a scheme of finite type over $\operatorname{Spec} R$. 
Let $n$ be a non-negative integer. 
Suppose that any irreducible component $X_i$ of $X$ satisfies $\dim ' X_i = n$, 
where $\dim ' X$ is defined in Definition \ref{defi:dim'}. 
Then we denote $\operatorname{Jac}'_{X/k} := \operatorname{Fitt} ^{n} (\Omega' _{X/k})$ and 
it is called the \textit{special Jacobian ideal} of $X$ (see \cite[Section 20.2]{Eis95} for the definition of the Fitting ideal). 
We note that $\dim ' X_i = \dim X_i$ holds if $X_i$ contains a $k$-point by Lemma \ref{lem:dim'}(4). 

\item 
Suppose that $R_0 = k$ and $X$ is an integral normal scheme of finite type over $\operatorname{Spec} R$. 
Let $n = \dim' X$ and 
let $i: X_{\rm reg} \to X$ be the inclusion map from the regular locus $X_{\rm reg}$ of $X$. 
Then the \textit{special canonical sheaf} $\omega ' _{X/k}$ is defined by 
$\omega ' _{X/k} = i_* (\Omega'^n _{X_{\rm reg}/k})$. 

\item
Under the same setting in (3), 
a Weil divisor $K_X$ on $X$ satisfying $\mathcal{O}_X(K_X)|_{X_{\mathrm{reg}}} \simeq \Omega'^n _{X_{\rm reg}/k}$ 
is called the \textit{canonical divisor} on $X$. 
The canonical divisor $K_X$ is defined up to linear equivalence. 
Note that $\omega ' _{X/k} \simeq \mathcal{O}_X(K_X)$ holds as usual. 
In fact, since we have $\operatorname{codim}_X (X \setminus X_{\rm reg}) \ge 2$ by the normality of $X$, it follows that 
\[
\Gamma \bigl( V, \mathcal{O}_X(K_X) \bigr) = \Gamma \bigl( V \cap X_{\rm reg}, \mathcal{O}_X(K_X) \bigr) = 
\Gamma \bigl( V , i_* i^* \mathcal{O}_X(K_X) \bigr)
\]
for any open subset $V \subset X$. 

\item 
Under the same setting in (3), we say that $X$ is \textit{$\mathbb{Q}$-Gorenstein} if 
$\omega _{X/k} ^{\prime [r]} := (\omega _{X/k} ^{\prime \otimes r})^{**} \simeq \mathcal{O}_X(r K_X)$ 
is an invertible sheaf for some $r \in \mathbb{Z}_{>0}$. 
In this case, we have a canonical map
\[
\eta_r \colon (\Omega_{X/k} ^{\prime n})^{\otimes r} \to \omega _{X/k} ^{\prime [r]}.
\] 
Since $\omega _{X/k} ^{\prime [r]}$ is an invertible sheaf, 
an ideal sheaf $\mathfrak{n}_{r,X} \subset \mathcal{O}_X$ is uniquely determined by 
${\rm Im}(\eta_r) = \mathfrak{n}_{r,X} \otimes _{\mathcal{O}_X} \omega _{X/k} ^{\prime [r]}$.
The ideal sheaf $\mathfrak n_{r,X}$ is called the $r$-th \textit{Nash ideal} of $X$.
\end{enumerate}
\end{defi}

\begin{rmk}\label{rmk:jacobianmatrix}
As with the usual Jacobian ideals for varieties over $k$, 
the special Jacobian ideal can be locally described by the Jacobian matrix. 
Let $S = R[y_1, \ldots , y_m]$ with $R=R_0[[x_1, \ldots , x_N]]$ and let $A = S/I$ for some ideal $I = (f_1, \ldots, f_r) \subset S$. 
Then by Proposition \ref{prop:spOmega2}(\ref{item:ex2}), we have an exact sequence
\[
I/I^2 \to \Omega' _{S/R_0} \otimes _{S} A \to \Omega' _{A/R_0} \to 0. 
\]
Here, we have $\Omega' _{S/R_0} \simeq S^{\oplus N+m}$ with basis $d'x_i$'s and $d'y_j$'s. 
Furthermore, for $f \in I$, we have 
\[
d' f = \sum _{i=1} ^N \frac{\partial f}{\partial x_i} d'x_i + \sum _{j=1} ^m \frac{\partial f}{\partial y_j} d'y_j. 
\]
Therefore, we have
\[
\operatorname{Fitt} ^{n} (\Omega' _{A/R_0})
= \bigl( \mathcal{J}_{N+m-n}(I; \operatorname{Der}_{R_0}(S)) + I \bigr)/I. 
\]
Note here that $\operatorname{Der}_{R_0}(S) = \operatorname{Der}' _{R_0}(S)$ is a free $S$-module generated by
$\partial / \partial x_i$'s and $\partial / \partial y_j$'s. 

This observation shows that if $R_0 = k$ and $X$ is an integral scheme of finite type over $R$, 
then $\operatorname{Jac}'_{X/k}$ defines the singular locus of $X$ by Remark \ref{rmk:JC}(2)(c). 
Note here that $\operatorname{ht} I = N + m - \dim' X$ holds by Lemma \ref{lem:dim'}(1). 
\end{rmk}

\section{Log pairs}\label{section:lp}
A \textit{log pair} $(X, \mathfrak{a})$ is a normal $\mathbb{Q}$-Gorenstein $k$-variety $X$ and 
an $\mathbb{R}$-ideal sheaf $\mathfrak{a}$ on $X$. 
Here, an $\mathbb{R}$-\textit{ideal sheaf} $\mathfrak{a}$ on $X$ is a formal product 
$\mathfrak{a} = \prod _{i = 1} ^s \mathfrak{a}_i ^{r_i}$, where $\mathfrak{a}_1, \ldots, \mathfrak{a}_s$ are 
non-zero coherent ideal sheaves on $X$ 
and $r_1, \ldots , r_s$ are positive real numbers. 
For a morphism $Y \to X$ and an $\mathbb{R}$-ideal sheaf $\mathfrak{a} = \prod _{i = 1} ^s \mathfrak{a}_i ^{r_i}$ on $X$, 
we denote by $\mathfrak{a} \mathcal{O}_Y$ the $\mathbb{R}$-ideal sheaf $\prod _{i = 1} ^s (\mathfrak{a}_i \mathcal{O}_Y)  ^{r_i}$ on $Y$. 

Let $\bigl( X, \mathfrak{a} = \prod _{i = 1} ^s \mathfrak{a}_i ^{r_i} \bigr)$ be a log pair. 
Let $f: X' \to X$ be a proper birational morphism from a normal variety $X'$ and let $E$ be a prime divisor on $X'$. 
We denote by $K_{X'/X} := K_{X'} - f^* K_X$ the relative canonical divisor. 
Then the \textit{log discrepancy} of $(X, \mathfrak{a})$ at $E$ is defined as 
\[
a_E(X, \mathfrak{a}) := 1 + \operatorname{ord}_E (K_{X'/X}) - \operatorname{ord}_E \mathfrak{a}, 
\]
where we define $\operatorname{ord}_E \mathfrak{a} := \sum _{i=1} ^s r_i \operatorname{ord}_E \mathfrak{a}_i$. 
The image $f(E)$ is called the \textit{center} of $E$ on $X$ and we denote it by $c_X(E)$. 
For a closed point $x \in X$, we define \textit{the minimal log discrepancy} at $x$ as 
\[
\operatorname{mld}_x (X, \mathfrak{a}) := \inf _{c_X(E) = \{ x \}} a_E (X, \mathfrak{a})
\]
if $\dim X \ge 2$, where the infimum is taken over all prime divisors $E$ over $X$ with center $c_X(E) = \{ x \}$. 
It is known that $\operatorname{mld}_x (X, \mathfrak{a}) \in \mathbb{R}_{\ge 0} \cup \{ - \infty \}$ in this case (cf.\ \cite[Corollary 2.31]{KM98}). 
When $\dim X = 1$, we define $\operatorname{mld}_x (X, \mathfrak{a}) := \inf _{c_X(E) = \{ x \}} a_E (X, \mathfrak{a})$ 
if the infimum is non-negative, and we define $\operatorname{mld}_x (X, \mathfrak{a}) := - \infty$ otherwise.

Let $R = k[[x_1, \ldots , x_N]]$. 
By \cite[Appendix A]{dFEM11}, we can extend the definition 
above to normal $R$-schemes of finite type. 
Let $X$ be an integral normal scheme of finite type over $R$. 
Then the canonical divisor $K_X$ is defined in Definition \ref{defi:omega'}(4). 
Suppose that $X$ is $\mathbb{Q}$-Gorenstein, that is, $r K_X$ is Cartier for some positive integer $r$. 
Let $f: Y \to X$ be a proper birational morphism over $R$ from a regular scheme $Y$.
Then the \textit{relative canonical divisor} $K_{Y/X}$ of $f$ is defined as the $\mathbb{Q}$-divisor supported 
on the exceptional locus of $f$ such that $rK_Y- f^*(rK_X)$ and $rK_{Y/X}$ are linearly equivalent. 
We note that $K_{Y/X}$ is uniquely defined as a $\mathbb{Q}$-divisor (cf.\ \cite[Lemma A.11(ii)]{dFEM11}). 
Therefore, the log discrepancies and the minimal log discrepancies for $k$-varieties defined above can be extended to 
$\mathbb{Q}$-Gorenstein normal schemes of finite type over $R$, and we use the same notation. 

\begin{rmk}\label{rmk:compareK}
Let $X$ be a normal $k$-variety and let $x \in X$ be a closed point. 
Let $\widehat{\mathcal{O}}_{X,x}$ be the completion of the local ring $\mathcal{O}_{X,x}$ at its maximal ideal. 
Let $\widehat{X} := \operatorname{Spec} \bigl( \widehat{\mathcal{O}}_{X,x} \bigr)$ and let $\widehat{x} \in \widehat{X}$ be the closed point. 
Then for the induced flat morphism $f: \widehat{X} \to X$, it follows that
\[
f^* (K_X) = K_{\widehat{X}}, \qquad f^* (\omega _{X}) = \omega ' _{\widehat{X}/k}
\]
by \cite[Proposition A.14]{dFEM11}. Furthermore, for an integer $r$, if $r K_X$ is Cartier, then so is $r K_{\widehat{X}}$. 

Suppose further that $X$ is $\mathbb{Q}$-Gorenstein. 
Let $\mathfrak{a}$ be an $\mathbb{R}$-ideal sheaf on $X$. 
Then it follows from \cite[Remark 2.6]{Kaw21} (cf.\ \cite[Proposition 2.11]{dFEM11}) that
\[
\operatorname{mld}_{\widehat{x}} (\widehat{X}, \widehat{\mathfrak{a}}) = 
\operatorname{mld}_x (X, \mathfrak{a}), 
\]
where $\widehat{\mathfrak{a}} := \mathfrak{a} \mathcal{O}_{\widehat{X}}$.
\end{rmk}

\section{Arc spaces of $k[[x_1, \ldots , x_N]]$-schemes}\label{section:jet}
In this section, we suppose $R_0 = k$ and $R = k[[x_1, \ldots , x_N]]$, 
and we discuss the jet schemes and the arc spaces of $R$-schemes of finite type. 
We refer the reader to \cite{EM09} and \cite{CLNS} for the theory of jet schemes and arc spaces of $k$-varieties. 
In this section, we see that the codimensions of cylinders of arc spaces can be defined 
in the same way as with $k$-varieties. 

Let $X$ be a scheme over $k$. Let $({\sf Sch}/k)$ be the category of $k$-schemes and $({\sf Sets})$ the category of sets. 
For a non-negative integer $m$, we define a contravariant functor  $F^X_{m}: ({\sf Sch}/k) \to ({\sf Sets})$ by 
\[
F^X_{m}(Y) = \operatorname{Hom} _{k}\left( Y\times_{\operatorname{Spec} k} \operatorname{Spec} k[t]/(t^{m+1}), X \right).
\]
It is known that the functor $F^X_{m}$ is always represented by a scheme $X_m$ over $k$ (cf.\ \cite[Ch.3.\ Proposition 2.1.3]{CLNS}). 

For $m \ge n \ge 0$, the canonical surjective ring homomorphism $k[t]/(t^{m+1}) \to k[t]/(t^{n+1})$ 
induces a morphism $\pi^X _{mn}:X_m \to X_n$, which is called the \textit{truncation morphism}. 
There exist the projective limit and the projections
\[
X_\infty := \mathop{\varprojlim}\limits_{m} X_m, \qquad \psi^X _{m}:X_{\infty} \to X_m, 
\]
and $X_{\infty}$ is called the {\it arc space} of $X$. 
Then there is a bijective map
\[
\operatorname{Hom} _{k}(\operatorname{Spec} K, X_{\infty}) \simeq \operatorname{Hom} _{k}(\operatorname{Spec} K[[t]], X)
\]
for any field $K$ with $k\subset K$. 
For $m \in \mathbb{Z}_{\ge 0} \cup \{ \infty \}$, we denote by $\pi^X _m: X_m \to X$ the canonical truncation morphism. 
For $m\in\mathbb Z_{\ge 0}\cup\{\infty\}$ and a morphism $f:Y\to X$ of schemes over $k$,
we denote by $f_m : Y_m \to X_m$ the morphism induced by $f$. 
We often abbreviate $\pi^X _{mn}$, $\pi^X _m$ and $\psi^X _{m}$ to 
$\pi _{mn}$, $\pi _m$ and $\psi _{m}$, respectively when no confusion can arise. 

If $X$ is a scheme of finite type over $k$, 
then so is $X_m$ (cf.\ \cite[Proposition 2.2]{EM09}). 
In this paper, we deal with a scheme of finite type over $R = k[[x_1, \ldots , x_N]]$. 

\begin{prop}[{cf.\ \cite[Corollary 4.2]{Ish09}}]\label{prop:fintype}
Let $X$ be a scheme of finite type over $R = k[[x_1, \ldots , x_N]]$. 
Then the following hold. 
\begin{enumerate}
\item $X_m$ is a scheme of finite type over $R$. 
\item For any $m \ge n \ge 0$, the truncation map $\pi_{mn}: X_m \to X_n$ is a morphism of finite type. 
\end{enumerate}
\end{prop}
\begin{proof}
We omit the proof because we will give a complete proof for Proposition \ref{prop:repl}, 
which deals with a more complicated case. 
See also Remark \ref{rmk:fintype} below. 
\end{proof}

\begin{rmk}\label{rmk:fintype}
The same arguments in Lemmas \ref{lem:repl} and \ref{lem:repl2} give a local description of $X_m$ as follows. 
\begin{enumerate}
\item 
Let $S := k[[x_1, \ldots , x_N]][y_1\ldots,y_M]$ and $A := \operatorname{Spec} S$.
Then we have $A_m \simeq \operatorname{Spec} S_m$, where 
\[
S_m := k \bigl[ \bigl[ x_1^{(0)}, \ldots, x_N^{(0)} \bigr] \bigr]
\bigl[ x_j^{(1)},\dots, x_j^{(m)},y_{j'}^{(0)},\dots,y_{j'}^{(m)} 
\ \big | \  1\le j \le N,1\le j' \le M \bigr].
\]
Furthermore, for $m \ge n \ge 0$, the truncation map $\pi_{mn}:A_m \to A_n$ is induced 
by the ring inclusion $S_n \hookrightarrow S_m$. 

\item 
Let $X = \operatorname{Spec} (S/I)$ be the closed subscheme of $A$ defined by an ideal $I = (f_1, \ldots , f_r) \subset S$. 
For $1 \le i \le r$ and $0 \le \ell \le m$, we define $F_i ^{(\ell)} \in S_m$ as follows: 
\[
f_{i} \biggl( \sum_{\ell = 0}^mx_{1}^{(\ell)}t^{\ell} , \dots, \sum_{\ell =0}^mx_{N}^{(\ell)}t^{\ell}, 
\sum_{\ell=0}^my_{1}^{(\ell)}t^{\ell},\dots,\sum_{\ell=0}^m y_{M}^{(\ell)}t^{\ell} \biggr) \equiv 
\sum_{\ell=0}^m F_{i}^{(\ell)}t^{\ell} \pmod{t^{m+1}}. 
\]
Let 
\[
I_m := \bigl( F_i ^{(s)} \ \big| \  1 \le i \le r, \ 0 \le s \le m \bigr) \subset S_m
\]
be the ideal of $S_m$ generated by $F_i ^{(s)}$'s. 
Then we have $X_m \simeq \operatorname{Spec} (S_m/ I_m)$. 
Furthermore, for $m \ge n \ge 0$, the truncation map $\pi_{mn}: X_m \to X_n$ is induced 
by the ring homomorphism $S_n/I_n \to S_m/I_m$. 
\end{enumerate}
\end{rmk}

A subset $C \subset X_{\infty}$ is called a \textit{cylinder} if $C = \psi_{m} ^{-1}(S)$ holds for some $m \ge 0$ and 
a constructible subset $S \subset X_m$. Typical examples of cylinders appearing in this paper are the \textit{contact loci} 
$\operatorname{Cont}^{m}(\mathfrak{a})$ and $\operatorname{Cont}^{\geq m}(\mathfrak{a})$ defined as follows. 
\begin{defi}
\begin{enumerate}
\item For an arc $\gamma\in X_\infty$ and an ideal sheaf $\mathfrak{a} \subset \mathcal O_X$, 
the \textit{order} of $\mathfrak{a}$  measured by $\gamma$
is defined as follows:
\[
\operatorname{ord}_\gamma(\mathfrak{a}) := \sup \{ r\in \mathbb Z_{\geq 0}\mid \gamma^*(\mathfrak a)\subset (t^r)\}, 
\]
where $\gamma^*: \mathcal{O}_{X} \to K[[t]]$ is the induced ring homomorphism by $\gamma$. 
Here $K$ is the field extension of $k$.

\item For $m \in \mathbb{Z}_{\ge 0}$, we define $\operatorname{Cont}^{m}(\mathfrak{a}), \operatorname{Cont}^{\geq m}(\mathfrak{a}) \subset X_{\infty}$ as follows:
\begin{align*}
\operatorname{Cont}^{m}(\mathfrak{a}) &:= \{\gamma \in X_\infty \mid \operatorname{ord}_\gamma(\mathfrak a)= m\}, \\
\operatorname{Cont}^{\geq m}(\mathfrak{a}) &:= \{\gamma \in X_\infty \mid \operatorname{ord}_\gamma(\mathfrak a)\geq m\}.
\end{align*}
\end{enumerate}
\end{defi}

\noindent
By definition, we have
\[
\operatorname{Cont}^{\geq m}(\mathfrak{a})=\psi_{m-1}^{-1}(Z(\mathfrak{a})_{m-1}),
\]
where $Z(\mathfrak{a})$ is the closed subscheme of $X$ defined by the ideal sheaf $\mathfrak{a}$. 
Therefore, $\operatorname{Cont}^{m}(\mathfrak{a})$ and $\operatorname{Cont}^{\geq m}(\mathfrak{a})$ are cylinders. 

For $m \le n+1$, we also define the subsets 
$\operatorname{Cont}^{m}(\mathfrak a)_n$ and $\operatorname{Cont}^{\geq m}(\mathfrak a)_n$ of $X_n$ in the same way. 

We denote by $\mathfrak{o}_{X} \subset \mathcal {O}_X$ the ideal sheaf
\[
\mathfrak{o}_{X} := (x_1, \ldots , x_N) \mathcal {O}_X \subset \mathcal {O}_X
\]
generated by $x_1, \ldots , x_N  \in R$. 
In this paper, we are interested in arcs contained in the contact locus $\operatorname{Cont}^{\ge 1}(\mathfrak{o}_{X})$. 
Due to the following lemma, the contact locus $\operatorname{Cont}^{\ge 1}(\mathfrak{o}_{X}) _m$ is a scheme of finite type over $k$.  
\begin{lem}\label{lem:contm_k}
Let $X$ be a scheme of finite type over $R=k[[x_1,\ldots,x_N]]$. 
Then for each $m \ge 0$, the contact locus 
$\operatorname{Cont}^{\ge 1}(\mathfrak{o}_{X})_m \subset X_m$ is a scheme of finite type over $k$. 
\end{lem}
\begin{proof}
The assertion follows from Proposition \ref{prop:fintype}(2). 
\end{proof}

For the proof of Lemma \ref{lem:EM4.1CI_R}, we state Hensel's lemma in several variables. 
\begin{lem}\label{lem:Hensel}
Let $K$ be a field. Let $N, s$ and $r$ be non-negative integers with $N + s \ge r$. 
Let $f_1, \ldots , f_{r} \in K[[t]][[x_1, \ldots , x_{N}]][x_{N+1}, \ldots , x_{N+s}]$ 
and let $a_1, \ldots , a_{N+s} \in K[[t]]$. 
Let $\mathcal{S} \subset \{ 1, \ldots , N+s \}$ be a subset with cardinality $\# \mathcal{S} = r$. 
Let $m$ and $e$ be non-negative integers with $m \ge e$. 
Suppose that
\begin{itemize}
\item $a_1, \ldots , a_{N} \in (t)$, 
\item $f_i(a_1, \ldots , a_{N+s}) \in (t^{m+e+1})$ for each $1 \le i \le r$, and 
\item $\det \left( \frac{\partial f_i}{\partial x_j} (a_1, \ldots , a_{N+s}) \right)_{1 \le i \le r,\ j \in \mathcal{S}} \not \in (t^{e+1})$. 
\end{itemize}
Then the following hold. 
\begin{enumerate}
\item 
There exist $b_1, \ldots , b_{N+s} \in K[[t]]$ such that 
\begin{itemize}
\item $f_i(b_1, \ldots , b_{N+s}) = 0$ for each $1 \le i \le r$, and 
\item $a_j - b_j \in (t^{m+1})$ for each $1 \le j \le N+s$. 
\end{itemize}
Furthermore, for $b _1, \ldots , b _{N+s} \in K[[t]]$ and $b' _1, \ldots , b' _{N+s} \in K[[t]]$ with the above two conditions, 
if 
\begin{itemize}
\item $b_j - b'_j \in (t^{m+2})$ holds for each $j \in \{ 1, \ldots , N+s \} \setminus \mathcal{S}$, 
\end{itemize}
then 
\begin{itemize}
\item $b_j - b'_j \in (t^{m+2})$ holds also for each $j \in \mathcal{S}$. 
\end{itemize}

\item Moreover, for any sequence $\bigl( a' _{j} \in K[[t]] \ \big| \ j \in \{ 1, \ldots , N+s \} \setminus \mathcal{S} \bigr )$ 
satisfying $a' _j - a_j \in (t^{m+1})$, 
there exist $b_1, \ldots , b_{N+s} \in K[[t]]$ satisfying the following conditions: 
\begin{itemize}
\item $f_i(b_1, \ldots , b_{N+s}) = 0$ for each $1 \le i \le r$, 
\item $a'_j - b_j \in (t^{m+2})$ for each $j \in \{ 1, \ldots , N+s \} \setminus \mathcal{S}$, and 
\item $a_j - b_j \in (t^{m+2})$ for each $j \in \mathcal{S}$.
\end{itemize}
\end{enumerate}
\end{lem}

\begin{proof}
When $N = 0$ and $f_1, \ldots , f_{r} \in K[x_1, \ldots, x_s]$, 
the assertions are proved in the proof of \cite[Lemma 4.1]{DL99} (cf.\ \cite[Proposition 4.1]{EM09}). 
The same proof works in our setting. 
\end{proof}

\begin{rmk}\label{rmk:Hensel}
The same statement as in Lemma \ref{lem:Hensel} holds even when we replace $K[[t]][[x_1, \ldots , x_{N}]][x_{N+1}, \ldots , x_{N+s}]$ with 
$K[x_{N+1}, \ldots , x_{N+s}][[t]]$. 
This version will be used in the proof of Proposition \ref{prop:EM4.1_k[[t]]}. 
\end{rmk}

\begin{lem}\label{lem:EM4.1CI_R}
Let $N, s$ and $r$ be non-negative integers with $N + s \ge r$. 
Let $R = k[[x_1, \ldots , x_N]]$ and let $S = R[y_1, \ldots , y_s]$. 
Let $I = (F_1, \ldots , F_r)$ be the ideal generated by elements $F_1, \ldots , F_r \in S$, and 
let $M = \operatorname{Spec} (S/I)$. 
Let $\mathfrak{o}_{M} \subset \mathcal {O}_M$ be the ideal sheaf generated by $x_1, \ldots , x_N \in R$.
Let $\overline{J} = \operatorname{Fitt}^{N+s-r} (\Omega' _{M/k})$. 
Then, for non-negative integers $m$ and $e$ with $m \ge e$, the following hold. 
\begin{enumerate}
\item It follows that 
\[
\psi _{m} \left( \operatorname{Cont}^{e}(\overline{J}) 
	\cap \operatorname{Cont}^{\ge 1}(\mathfrak{o}_{M}) \right) 
= \pi _{m+e, m} \left( \operatorname{Cont}^{e}(\overline{J})_{m+e} 
	\cap \operatorname{Cont}^{\ge 1}(\mathfrak{o}_{M})_{m+e} \right). 
\]

\item $\pi_{m+1, m}:M_{m+1} \to M_m$ induces a piecewise trivial fibration 
\[
\psi _{m+1} \left( \operatorname{Cont}^e(\overline{J}) \cap \operatorname{Cont}^{\ge 1}(\mathfrak{o}_{M}) \right) 
	\to \psi _{m} \left( \operatorname{Cont}^e(\overline{J}) \cap \operatorname{Cont}^{\ge 1}(\mathfrak{o}_{M}) \right)
\]
with fiber $\mathbb{A}^{N+s-r}$. 
\end{enumerate}
\end{lem}
\begin{proof}
Let $J := \mathcal{J}_r \bigl( I; \operatorname{Der}_{k}(S) \bigr) \subset S$. 
Then we have $\overline{J} = (J+I)/I$ by Remark \ref{rmk:jacobianmatrix}. 
Note here that $\operatorname{Der}_{k}(S) = \operatorname{Der}'_{k}(S)$ is generated by 
$\partial / \partial x_i$'s and $\partial / \partial y_j$'s. 
Therefore, (1) follows from the first assertion of Lemma \ref{lem:Hensel}(1). 
Furthermore, (2) follows from Lemma \ref{lem:Hensel}(2) and the second assertion of Lemma \ref{lem:Hensel}(1). 
\end{proof}

The following proposition is a formal power series ring version of \cite[Proposition 4.1]{EM09}. 

\begin{prop}\label{prop:EM4.1_R}
Let $X$ be an integral scheme of finite type over $R = k[[x_1, \ldots, x_N]]$ of $\dim X = n$. 
Then there exists a positive integer $c$ such that the following hold 
for non-negative integers $m$ and $e$ with $m \ge ce$. 
\begin{enumerate}
\item We have
\begin{align*}
&\psi _{m} \left( \operatorname{Cont}^{e}(\operatorname{Jac}'_{X/k}) 
\cap \operatorname{Cont}^{\ge 1}(\mathfrak{o}_{X}) \right) \\
&= \pi _{m+e, m} \left( \operatorname{Cont}^{e}(\operatorname{Jac}'_{X/k})_{m+e} 
\cap \operatorname{Cont}^{\ge 1}(\mathfrak{o}_{X})_{m+e} \right). 
\end{align*}

\item $\pi_{m+1, m}:X_{m+1} \to X_m$ induces a piecewise trivial fibration 
\[
\psi _{m+1} \left( \operatorname{Cont}^e(\operatorname{Jac}'_{X/k}) \cap 
\operatorname{Cont}^{\ge 1}(\mathfrak{o}_{X}) \right) \to 
\psi _{m} \left( \operatorname{Cont}^e(\operatorname{Jac}'_{X/k}) \cap \operatorname{Cont}^{\ge 1}(\mathfrak{o}_{X}) \right)
\]
with fiber $\mathbb{A}^n$. 
\end{enumerate}
\end{prop}
\begin{proof}
We omit the proof. See the proof of Proposition \ref{prop:EM4.1_Rt} to see 
how it can be reduced to the complete intersection case proved in Lemma \ref{lem:EM4.1CI_R}. 
Note that we may assume that $X$ has a $k$-point and hence we have $\dim ' X = \dim X = n$ by Lemma \ref{lem:dim'}(4). 
Otherwise, we have $\operatorname{Cont}^{\ge 1}(\mathfrak{o}_{X}) = \emptyset$ (cf.\ Lemma \ref{lem:contm_k}), 
and the assertions are clear. 
\end{proof}

\begin{rmk}
Proposition \ref{prop:EM4.1_R} is a formal power series ring version of \cite[Proposition 4.1]{EM09}. 
Note that in \cite[Proposition 4.1]{EM09}, they prove that $c = 1$ satisfies the statement. 
However, the weaker statement as in Proposition \ref{prop:EM4.1_R} using $c$ is enough for our later use. 
\end{rmk}

By Proposition \ref{prop:EM4.1_R}, 
the codimension of cylinder contained in $\operatorname{Cont}^{\ge 1}(\mathfrak{o}_{X})$ is well-defined as follows. 

\begin{defi}\label{defi:codim}
Let $X$ be an integral scheme of finite type over $R=k[[x_1, \ldots, x_N]]$ and 
let $C \subset X_{\infty}$ be a cylinder contained in $\operatorname{Cont}^{\ge 1}(\mathfrak{o}_{X})$. 
\begin{enumerate}
\item 
Assume that $C \subset \operatorname{Cont}^e(\operatorname{Jac}'_{X/k})$ for some $e \in \mathbb{Z}_{\ge 0}$.
Then we define the codimension of $C$ in $X_\infty$ as 
\[
\operatorname{codim}(C):=(m+1)\operatorname{dim}X-\operatorname{dim}(\psi_m(C))
\]
for any sufficiently large $m$. This definition is well-defined by Proposition \ref{prop:EM4.1_R}.

\item 
In general, we define the codimension of $C$ in $X_\infty$ as follows:
\[
\operatorname{codim}(C) := 
\min_{e\in\mathbb Z_{\ge 0}} \operatorname{codim} \bigl ( C\cap \operatorname{Cont}^e(\operatorname{Jac}'_{X/k}) \bigr). 
\]
By convention, $\operatorname{codim}(C) = \infty$ if 
$C \cap \operatorname{Cont}^e(\operatorname{Jac}'_{X/k}) = \emptyset$ for any $e \ge 0$. 
\end{enumerate}
\end{defi}

The following theorem is a formal power series ring version of {\cite[Theorem 7.4]{EM09}}.
\begin{thm}\label{thm:mld_R}
Let $X$ be a $\mathbb Q$-Gorenstein integral normal scheme of finite type over $R = k[[x_1, \ldots, x_N]]$. 
Let $x$ be a $k$-point of $X$ and 
let $\mathfrak{m}_x \subset \mathcal{O}_X$ be the corresponding maximal ideal sheaf. 
Let $r$ be a positive integer such that $r K_X$ is Cartier. 
Let $\mathfrak a$ be a non-zero ideal sheaf on $X$ and $\delta$ a positive real number. 
Then we have
\begin{align*}
\operatorname{mld}_x(X,\mathfrak{a}^\delta)
&= \inf _{w, m \in \mathbb{Z}_{\ge 0}} 
\Bigl\{ \operatorname{codim} (C_{w,m}) - \frac{m}{r}-\delta w \Bigr\} \\
&= \inf _{w, m \in \mathbb{Z}_{\ge 0}} 
\Bigl\{ \operatorname{codim} (C' _{w,m}) - \frac{m}{r}-\delta w \Bigr\}, 
\end{align*}
where 
\begin{align*}
C_{w,m} &:= \operatorname{Cont}^w(\mathfrak a) \cap \operatorname{Cont}^{m}(\mathfrak n_{r,X}) 
\cap \operatorname{Cont}^{\ge 1}(\mathfrak m_x), \\
C' _{w,m} &:= \operatorname{Cont}^{\ge w}(\mathfrak a) \cap \operatorname{Cont}^{m}(\mathfrak n_{r,X}) 
\cap \operatorname{Cont}^{\ge 1}(\mathfrak m_x). 
\end{align*}
\end{thm}

\begin{proof}
The assertions are formal power series ring versions of Theorem 7.4 and Remark 7.5 in \cite{EM09}, 
and their proofs also work in this setting by making the following modifications: 
\begin{itemize}
\item Replacing $\Omega_{\text{--}/k}$ with $\Omega' _{\text{--}/k}$, and 
$\operatorname{Jac} _{\text{--}/k}$ with $\operatorname{Jac}' _{\text{--}/k}$. 

\item Considered cylinders $C$ are contained in $\operatorname{Cont}^{\ge 1}(\mathfrak{o}_{X})$. 
\end{itemize}
Theorem 7.4 in \cite{EM09} is a consequence of Lemma 7.3 in \cite{EM09}. 
The key ingredients of the proof of Lemma 7.3 in \cite{EM09} are 
\begin{itemize}
\item Theorem 6.2 and Corollary 6.4 in \cite{EM09}, and
\item Proposition 5.11 in \cite{EM09}. 
\end{itemize}

Theorem 6.2 and Corollary 6.4 in \cite{EM09} are the codimension formula as in Proposition \ref{prop:EM6.2_k[[t]]}, 
and they are formal consequences of Proposition 4.4(i) in \cite{EM09}. 
Proposition 4.4(i) in \cite{EM09} is still valid in our setting by replacing $\Omega_{X}$ with $\Omega' _{X/k}$ due to Lemma \ref{lem:D=D'} 
(see Lemma \ref{lem:EM4.4_R} for the detailed argument). 

Proposition 5.11 in \cite{EM09} is a proposition on codimensions as in Proposition \ref{prop:negligible}, and 
it is a consequence of Lemma 5.1 and Corollary 5.2 in \cite{EM09} (Corollary 5.2 is a corollary of Lemma 5.1). 
The proof of Lemma 5.1 in \cite{EM09} still works in our setting by replacing $\operatorname{Jac} _{\text{--}/k}$ with $\operatorname{Jac}' _{\text{--}/k}$. 
The only important point is that the ideal $\operatorname{Jac}' _{X/k}$ defines the singular locus of $X$ 
even when $X$ is an integral scheme of finite type over $R$ (cf.\ Remark \ref{rmk:jacobianmatrix}). 

Besides, Lemma 6.1 in \cite{EM09} is used in the proof of Lemma 7.3 in \cite{EM09}, and Proposition 3.2 in \cite{EM09} is used in the proof of Corollary 5.2 in \cite{EM09}. 
Proposition 3.2 and Lemma 6.1 in \cite{EM09} are formal consequences of the valuative criterion of properness, and their proofs work in our setting. 
\end{proof}

\section{Arc spaces of $k[t]$-schemes}\label{section:ktjet}
In this section, we deal with the arc spaces of $k[t]$-schemes. 
Let $X$ be a scheme over $k[t]$. 
For a non-negative integer $m$, we define a contravariant functor $F^X_{m}: ({\sf Sch}/k) \to ({\sf Sets})$ by 
\[
F^X _{m}(Y) = \operatorname{Hom} _{k[t]} \left( Y \times_{\operatorname{Spec} k} \operatorname{Spec} k[t]/(t^{m+1}), X \right).
\]
By the same argument as in \cite[Ch.4.\ Theorem 3.2.3]{CLNS}, we can see that the functor $F^X_{m}$ 
is always represented by a scheme $X_m$ over $k$. 
We shall use the same symbols $X_{\infty}$, $\pi_{mn}$, $\psi_m$ and $\pi_m$ as in Section \ref{section:jet} also for this setting. 

In this section, we deal with the following two categories of $k[t]$-schemes: 
\begin{enumerate}
\item[(1)] $X$ is a scheme of finite type over $k[t][[x_1, \ldots , x_N]]$. 
\item[(2)] $X$ is an affine scheme of the form $X = \operatorname{Spec} \bigl( k[x_1, \ldots , x_N][[t]]/I \bigr)$.
\end{enumerate}

\noindent
We note that in \cite{DL02}, Denef and Loeser extend the theory of arc spaces of $k$-varieties 
to the case where 
\begin{enumerate}
\item[(3)] $X$ is a scheme of finite type over $k[t]$. 
\end{enumerate}
In \cite{NS}, we dealt with the arc spaces of $X$ in (3). 
However, in this paper, we need to work on the arc spaces of $X$ in (1) and (2). 
We also note that Sebag in \cite{Seb04} deals with formal $k[[t]]$-schemes of finite type, 
and this theory can be applied to (2) and (3) (see also \cite{CLNS} for this theory). 

In Subsection \ref{subsection:ktjet}, we discuss case (1). 
In Subsection \ref{subsection:formal}, we discuss case (2), 
where we will not deal with formal $k[[t]]$-schemes in general but deal with only affine schemes in a minimum way.

\subsection{Arc spaces of $k[t][[x_1, \ldots, x_N]]$-schemes}\label{subsection:ktjet}
In this subsection, we suppose $R_0 = k[t]$ and $R = k[t][[x_1, \ldots, x_N]]$, and 
we discuss the arc spaces of $R$-schemes of finite type. 

\subsubsection{Arc spaces}

First, we prove that if $X$ is a scheme of finite type over $R$, 
then so is $X_m$. 

\begin{lem}\label{lem:repl}
Let $S := k[t][[x_1, \ldots , x_N]][y_1, \ldots ,y_M]$ and let $A := \operatorname{Spec} S$.
Then we have $A_m \simeq \operatorname{Spec} S_m$, where 
\[
S_m := k \bigl[ \bigl[ x_1^{(0)}, \ldots, x_N^{(0)} \bigr] \bigr]
\bigl[ x_j^{(1)},\dots, x_j^{(m)},y_{j'}^{(0)},\dots,y_{j'}^{(m)} 
\ \big | \  1\le j \le N,1\le j' \le M \bigr].
\]
Furthermore, for $m \ge n \ge 0$, the truncation map $\pi_{mn}:A_m \to A_n$ is induced 
by the ring inclusion $S_n \hookrightarrow S_m$. 
\end{lem}

\begin{proof}
Let $Y := \operatorname{Spec} C$ be an affine scheme over $k$. 
We shall give a natural bijective map
\[
\Phi:\operatorname{Hom} _{k[t]} \bigl( S , C[t]/(t^{m+1}) \bigr) 
\to \operatorname{Hom} _{k} ( S_m, C ). 
\]

For each $0 \le i \le m$, we denote by $p_i$ the projection
\[
p_i: C[t]/(t^{m+1}) \to C; \quad c_0+c_1t+\cdots+c_mt^{m} \mapsto c_i. 
\]
For $\alpha \in \operatorname{Hom} _{k[t]} \left( S, C[t]/(t^{m+1}) \right)$, 
we define $\Phi (\alpha) \in \operatorname{Hom} _{k} ( S_m, C )$ as follows. 
First, we define the ring homomorphism $\alpha ' _0 : S_0 \to C$ as the composition
\[
S_0 \hookrightarrow S \xrightarrow{\alpha} C[t]/(t^{m+1}) \xrightarrow{p_0} C. 
\]
Then we define $\alpha ' : S_m \to C$ as the ring homomorphism uniquely determined by the following conditions: 
\begin{itemize}
\item $\alpha ' (f) = \alpha '_0 (f)$ holds for any $f \in S_0$. 
\item $\alpha ' \bigl( x_j ^{(s)} \bigr) = p_s \bigl( \alpha(x_j) \bigr)$ holds for each $1 \le j \le N$ and $1 \le s \le m$. 
\item $\alpha ' \bigl( y_j ^{(s)} \bigr) = p_s \bigl( \alpha(y_j) \bigr)$ holds for each $1 \le j \le M$ and $1 \le s \le m$.
\end{itemize}
Then we define $\Phi (\alpha) = \alpha '$. 

Next, we define the inverse map 
\[
\Psi: \operatorname{Hom} _{k} ( S_m , C ) \to \operatorname{Hom} _{k[t]} \bigl( S, C[t]/(t^{m+1}) \bigr). 
\]
We set 
\begin{align*}
S'_m &:= k \bigl[ x_j^{(1)},\dots, x_j^{(m)} \ \big| \ 1 \le j \le N \bigr] \bigl[ \bigl[ x^{(0)}_1, \ldots , x^{(0)}_N \bigr] \bigr], \\
S''_m &:= k \bigl[ \bigl[ x^{(0)}_1, \ldots , x^{(0)}_N \bigr] \bigr] \bigl[ x_j^{(1)},\dots, x_j^{(m)} \ \big| \ 1 \le j \le N \bigr].
\end{align*}
We define a ring homomorphim $\Lambda _1 : k[t][x_1, \ldots , x_N] \to S'_m[t]/(t^{m+1})$ by 
\[
\Lambda _1 (t) = t, \qquad \Lambda _1 (x_j) = x_j^{(0)} + x_j^{(1)}t + \cdots + x_j^{(m)}t^m. 
\]
Since $\Lambda _1 \bigl((x_1,\ldots, x_N) \bigr) \subset \bigl( x_1^{(0)}, \ldots, x_N^{(0)},t \bigr)$ holds, 
$\Lambda _1$ induces a ring homomorphism 
$\Lambda _2 : k[t][[x_1, \ldots , x_N]] \to S'_m[t]/(t^{m+1})$. 
Note here that its image is contained in $S'' _m[t]/(t^{m+1})$. 
Furthermore, $S'' _m[t]/(t^{m+1})$ is a subring of $S_m[t]/(t^{m+1})$. 
Therefore, we have a ring homomorphism $\Lambda _3: k[t][[x_1, \ldots , x_N]] \to S_m[t]/(t^{m+1})$. 
Then we define $\Lambda : S \to S_m[t]/(t^{m+1})$ as the ring homomorphism uniquely determined by the following conditions: 
\begin{itemize}
\item $\Lambda (f) = \Lambda _3 (f)$ holds for any $f \in k[t][[x_1, \ldots , x_N]]$. 
\item $\Lambda (y_j) = y_j^{(0)}+y_j^{(1)}t+\cdots +y_j^{(m)}t^m$ holds for each $1 \le j \le M$. 
\end{itemize}
Then $\Lambda$ is a $k[t]$-ring homomorphism. 

For $\beta \in \operatorname{Hom} _{k} ( S_m, C )$, 
we define $\Psi (\beta) \in \operatorname{Hom} _{k[t]} \bigl( S, C[t]/(t^{m+1}) \bigr)$ as the composition
\[
S \xrightarrow{\Lambda} S_m[t]/(t^{m+1}) \xrightarrow{\beta} C[t]/(t^{m+1}), 
\]
where $S_m[t]/(t^{m+1}) \to C[t]/(t^{m+1})$ is the $k[t]$-ring homomorphism induced by $\beta:S_m \to C$. 

By the construction of $\Phi$ and $\Psi$, if $\alpha ' = \Phi (\alpha)$ and $\beta ' = \Psi (\beta)$, then they satisfy the following: 
\begin{itemize}
\item $p_s \bigl( \alpha(x_j) \bigr) = \alpha ' \bigl( x_j ^{(s)} \bigr)$  holds for each $1 \le j \le N$ and $0 \le s \le m$. 
\item $p_s \bigl( \alpha(y_j) \bigr) = \alpha ' \bigl( y_j ^{(s)} \bigr)$ holds for each $1 \le j \le M$ and $0 \le s \le m$.
\item $p_s \bigl( \beta '(x_j) \bigr) = \beta \bigl( x_j ^{(s)} \bigr)$  holds for each $1 \le j \le N$ and $0 \le s \le m$. 
\item $p_s \bigl( \beta '(y_j) \bigr) = \beta \bigl( y_j ^{(s)} \bigr)$ holds for each $1 \le j \le M$ and $0 \le s \le m$.
\end{itemize}
Therefore, we have $\Psi\circ \Phi=\operatorname{id}$ and $\Phi\circ \Psi=\operatorname{id}$. 
Hence, $F_{m}^A$ is represented by $\operatorname{Spec} S_m$. 
The second assertion follows from the construction of $A_m$. 
\end{proof}

\begin{lem}\label{lem:repl2}
We take over the notation in Lemma \ref{lem:repl}. 
Let $X = \operatorname{Spec} (S/I)$ be the closed subscheme of $A$ defined by an ideal $I = (f_1, \ldots , f_r) \subset S$. 
For $1 \le i \le r$ and $0 \le \ell \le m$, we define $F_i ^{(\ell)} \in S_m$ as follows: 
\[
f_{i} \biggl( \sum_{\ell = 0}^mx_{1}^{(\ell)}t^{\ell} , \dots, \sum_{\ell =0}^mx_{N}^{(\ell)}t^{\ell}, 
\sum_{\ell=0}^my_{1}^{(\ell)}t^{\ell},\dots,\sum_{\ell=0}^m y_{M}^{(\ell)}t^{\ell} \biggr) \equiv 
\sum_{\ell=0}^m F_{i}^{(\ell)}t^{\ell} \pmod{t^{m+1}}. 
\]
Let 
\[
I_m := \bigl( F_i ^{(s)} \ \big| \  1 \le i \le r, \ 0 \le s \le m \bigr) \subset S_m
\]
be the ideal of $S_m$ generated by $F_i ^{(s)}$'s. 
Then we have $X_m \simeq \operatorname{Spec} (S_m/ I_m)$. 
Furthermore, for $m \ge n \ge 0$, the truncation map $\pi_{mn}: X_m \to X_n$ is induced 
by the ring homomorphism $S_n/I_n \to S_m/I_m$. 
\end{lem}

\begin{proof}
Let $Y := \operatorname{Spec} C$ be an affine scheme over $k$. 
We can see that the bijective map $\Phi$ in the proof of Lemma \ref{lem:repl} 
induces the bijective map
\[
\operatorname{Hom} _{k[t]} \bigl( S/I , C[t]/(t^{m+1}) \bigr) 
\to \operatorname{Hom} _{k} ( S_m/I_m , C ). 
\]
Therefore, the functor $F_{m}^X$ is represented by $\operatorname{Spec} (S_m/ I_m)$. 
The second assertion follows from the construction of $X_m$. 
\end{proof}

\begin{rmk}
More precisely, 
\[
f_{i} \biggl( \sum_{\ell = 0}^mx_{1}^{(\ell)}t^{\ell} , \dots, \sum_{\ell =0}^mx_{N}^{(\ell)}t^{\ell}, 
\sum_{\ell=0}^my_{1}^{(\ell)}t^{\ell},\dots,\sum_{\ell=0}^m y_{M}^{(\ell)}t^{\ell} \biggr) 
\]
in Lemma \ref{lem:repl2} is defined as $\Lambda (f_i) \in S_m[t]/(t^{m+1})$, 
where $\Lambda$ is defined within the proof of Lemma \ref{lem:repl}. 
\end{rmk}

\begin{prop}\label{prop:repl}
If $X$ is a scheme of finite type over $R = k[t][[x_1, \ldots , x_N]]$. 
Then the following hold. 
\begin{enumerate}
\item $X_m$ is a scheme of finite type over $R$.
\item For any $m \ge n \ge 0$, the truncation map $\pi _{mn}: X_m \to X_n$ is a morphism of finite type. 
\end{enumerate}
\end{prop}

\begin{proof}
Take an affine cover $X = U_1 \cup \cdots \cup U_s$. 
Then, $F_{m}^X$ is represented by the scheme obtained by gluing the schemes $(U_i)_m$ constructed in Lemma \ref{lem:repl2} 
(cf.\ \cite[Proposition 2.2]{EM09}). 
Therefore, the assertions follow from Lemma \ref{lem:repl2}. 
\end{proof}

Cylinders and the contact loci 
\[
\operatorname{Cont}^{m}(\mathfrak{a}),\ \operatorname{Cont}^{\ge m}(\mathfrak{a}) \subset X_{\infty}, \qquad 
\operatorname{Cont}^{m}(\mathfrak{a})_n,\ \operatorname{Cont}^{\ge m}(\mathfrak{a})_n \subset X_n
\]
are also defined in this setting in the same way.

We denote by $\mathfrak{o}_{X} \subset \mathcal{O}_X$ the ideal sheaf
\[
\mathfrak{o}_{X} := (x_1, \ldots , x_N)\mathcal{O}_X \subset \mathcal{O}_X
\] 
generated by $x_1, \ldots , x_N \in R$. 
From the next subsection, we basically work on arcs contained in the contact locus 
$\operatorname{Cont}^{\ge 1}(\mathfrak{o}_{X})$. 
Due to the following lemma, the contact locus 
$\operatorname{Cont}^{\ge 1}(\mathfrak{o}_{X})_m$ is a scheme of finite type over $k$. 

\begin{lem}\label{lem:contm}
Let $X$ be a scheme of finite type over $R=k[t][[x_1, \ldots , x_N]]$. 
Then the following hold. 
\begin{enumerate}
\item 
For each $m \ge 0$, the contact locus $\operatorname{Cont}^{\ge 1}(\mathfrak{o}_{X})_m \subset X_m$ 
is a scheme of finite type over $k$. 

\item 
Any $k$-arc of $X$ is contained in $\operatorname{Cont}^{\ge 1}(\mathfrak{o}_{X})$. 
\end{enumerate}
\end{lem}
\begin{proof}
(1) follows from Proposition \ref{prop:repl}. 

We shall prove (2). 
Let $\gamma \in X_{\infty}$ be a $k$-arc. We may assume that $X$ is affine, and 
we may write $X = \operatorname{Spec} A$ with $A = S/I$, where
\[
S:=k[t][[x_1, \ldots, x_N]][y_1, \ldots , y_m]
\]
and $I$ is an ideal of $S$. 
Let $\gamma ^* : A \to k[[t]]$ be the corresponding $k[t]$-ring homomorphism. 
Let $M$ be the kernel of the composite map $S \to A \xrightarrow{\gamma ^*} k[[t]] \to k$. 
Since $S/M = k$, $M$ is of the form 
\[
(t, x_1, \ldots, x_N, y_1-a_1, \ldots , y_m - a_m)
\]
for some $a_i \in k$. It shows that $\gamma ^* (\mathfrak{o}_{X}) \subset (t)$ and 
hence $\gamma \in \operatorname{Cont}^{\ge 1}(\mathfrak{o}_{X})$. 
\end{proof}

\begin{lem}\label{lem:finjac}
Let $n$ be a non-negative integer and 
let $X$ be a scheme of finite type over $R = k[t][[x_1, \ldots, x_N]]$. 
Suppose that each irreducible component $X_i$ of $X$ has $\dim ' X_i \ge n+1$ (see Definition \ref{defi:dim'}). 
Let $\gamma \in X_{\infty}$ be a $k$-arc with 
$\operatorname{ord}_{\gamma} \bigl( \operatorname{Fitt}^n (\Omega' _{X/k[t]}) \bigr) < \infty$. 
Then we have
\[
\gamma ^* \Omega' _{X / k[t]} /T \simeq k[[t]]^{\oplus n}, 
\]
where $T$ is the torsion part of $\gamma ^* \Omega' _{X / k[t]}$. 
\end{lem}
\begin{proof}
We may assume that $X$ is affine, and we may write $X = \operatorname{Spec} A$ with $A = S/I$, where
\[
S := k[t][[x_1, \ldots, x_N]][y_1, \ldots , y_m]
\]
and $I$ is an ideal of $S$. 
If $P$ is a minimal prime of $I$, then we have
\[
\operatorname{ht} P = \dim S - \dim' (S/P) \le (N+m+1) - (n+1) = N+m-n
\]
by Lemma \ref{lem:dim'}(1). Therefore we have $\operatorname{ht} I \le N + m -n$. 

Let $\gamma ^* : A \to k[[t]]$ be the corresponding $k[t]$-ring homomorphism, 
and let $\overline{\gamma} ^*: A \to k((t))$ be its composition with $k[[t]] \to k((t))$. 
Let $\mathfrak{q} \subset A$ be the kernel of $\gamma ^*$ and $Q \subset S$ the corresponding prime ideal. 
Since $\overline{\gamma} ^*$ factors through $A_{\mathfrak{q}}$, 
it is sufficient to show that $\Omega' _{A_{\mathfrak{q}} / k[t]} \otimes _{A_{\mathfrak{q}}} k((t))$ has dimension $n$ as a $k((t))$-vector space. 

Let $w_1, \ldots , w_{\ell} \in I$ be generators of $I$. 
Let $M \in M_{N+m, \ell}(A_{\mathfrak{q}})$ be the Jacobian matrix with respect to 
$w_1, \ldots , w_{\ell} \in I$ and derivations $\frac{\partial}{\partial x_i}$'s and $\frac{\partial}{\partial y_i}$'s. 
Then $M$ defines a map $M: A_{\mathfrak{q}}^{\ell} \to A_{\mathfrak{q}}^{N+m}$ and 
its cokernel is isomorphic to $\Omega' _{A_{\mathfrak{q}} / k[t]}$ by Proposition \ref{prop:spOmega2}(\ref{item:ex2})(\ref{item:localization}). 
Since $\operatorname{ord}_{\gamma} \bigl( \operatorname{Fitt}^n (\Omega' _{X/k[t]}) \bigr) < \infty$, 
$M$ has an $(N+m-n)$-minor which is not contained in $\mathfrak{q}A_{\mathfrak{q}}$ (cf.\ Remark \ref{rmk:jacobianmatrix}). 
Furthermore, since we have
\[
\operatorname{ht}(IS_Q) \le \operatorname{ht} I \le N+m-n, 
\]
any $(N+m-n+1)$-minor of $M$ is contained in $\mathfrak{q}A_{\mathfrak{q}}$ by Remark \ref{rmk:JC}(2)(a) (cf.\ \cite[Theorem 30.4]{Mat89}). 
Therefore, the image of $M$ in $M_{N+m, \ell}(k((t)))$ has rank $N+m-n$, and it follows that 
$\Omega' _{A_{\mathfrak{q}} / k[t]} \otimes _{A_{\mathfrak{q}}} k((t))$ has dimension $n$ as a $k((t))$-vector space. 
\end{proof}

\begin{lem}\label{lem:order_R}
Let $n$ and $e$ be non-negative integers and 
let $X$ be a scheme of finite type over $R = k[t][[x_1, \ldots, x_N]]$. 
Suppose that each irreducible component $X_i$ of $X$ has $\dim ' X_i \ge n+1$.
Let $\gamma \in \operatorname{Cont}^e \bigl( \operatorname{Fitt}^n (\Omega' _{X/k[t]}) \bigr)$ be a $k$-arc. 
Then we have
\[
\gamma ^* \Omega' _{X / k[t]} \simeq 
k[[t]]^{\oplus n} \oplus \bigoplus _i k[t]/(t^{e_i})
\]
as $k[[t]]$-modules with $\sum_i e_i = e$. 
\end{lem}
\begin{proof}
The same proof as in \cite[Lemma 2.13(1)]{NS} works due to Lemma \ref{lem:finjac}. 
\end{proof}

\subsubsection{Cylinders and Codimension}\label{subsubsection:codim_R}
In this subsection, we define and discuss the codimensions of cylinders of the arc space of an $R$-scheme $X$ of finite type. 
We define the codimension only for cylinders contained in the contact locus $\operatorname{Cont}^{\ge 1}(\mathfrak{o}_{X})$, 
where $\mathfrak{o}_{X} \subset \mathcal {O}_X$ is the ideal sheaf generated by $x_1, \ldots , x_N \in R$. 
Due to Lemma \ref{lem:contm}, the contact locus $\operatorname{Cont}^{\ge 1}(\mathfrak{o}_{X})_m \subset X_m$ is a scheme of finite type over $k$, 
and hence cylinders contained in $\operatorname{Cont}^{\ge 1}(\mathfrak{o}_{X})$ are easier to handle than the general cylinders. 

First, we prove Proposition \ref{prop:EM4.1_Rt}, which is necessary for defining the codimension of cylinders. 

\begin{lem}\label{lem:EM4.1CI_Rt}
Let $N, s$ and $r$ be non-negative integers with $N + s \ge r$. 
Let $R = k[t][[x_1, \ldots , x_N]]$ and let $S = R[y_1, \ldots , y_s]$. 
Let $I = (F_1, \ldots , F_r)$ be the ideal generated by elements $F_1, \ldots , F_r \in S$, and 
let $M = \operatorname{Spec} (S/I)$. 
Let $\mathfrak{o}_{M} \subset \mathcal {O}_M$ be the ideal sheaf generated by $x_1, \ldots , x_N \in R$.
Let $\overline{J} = \operatorname{Fitt}^{N+s-r} (\Omega' _{M/k[t]})$. 
For non-negative integers $m$ and $e$ with $m \ge e$, the following hold. 
\begin{enumerate}
\item It follows that 
\[
\psi _{m} \left( \operatorname{Cont}^{e}(\overline{J}) 
	\cap \operatorname{Cont}^{\ge 1}(\mathfrak{o}_{M}) \right) 
= \pi _{m+e, m} \left( \operatorname{Cont}^{e}(\overline{J})_{m+e} 
	\cap \operatorname{Cont}^{\ge 1}(\mathfrak{o}_{M})_{m+e} \right). 
\]

\item $\pi_{m+1, m}:M_{m+1} \to M_m$ induces a piecewise trivial fibration 
\[
\psi _{m+1} \left( \operatorname{Cont}^e(\overline{J}) \cap \operatorname{Cont}^{\ge 1}(\mathfrak{o}_{M}) \right) 
	\to \psi _{m} \left( \operatorname{Cont}^e(\overline{J}) \cap \operatorname{Cont}^{\ge 1}(\mathfrak{o}_{M}) \right)
\]
with fiber $\mathbb{A}^{N+s-r}$. 
\end{enumerate}
\end{lem}
\begin{proof}
Let $J := \mathcal{J}_r \bigl( I; \operatorname{Der}_{k[t]}(S) \bigr) \subset S$. 
Then we have $\overline{J} = (J+I)/I$ by Remark \ref{rmk:jacobianmatrix}. 
Note here that $\operatorname{Der}_{k[t]}(S) = \operatorname{Der}'_{k[t]}(S)$ is generated by 
$\partial / \partial x_i$'s and $\partial / \partial y_j$'s. 
Therefore, (1) follows from the first assertion of Lemma \ref{lem:Hensel}(1). 
Furthermore, (2) follows from Lemma \ref{lem:Hensel}(2) and the second assertion of Lemma \ref{lem:Hensel}(1). 
\end{proof}

\begin{prop}\label{prop:EM4.1_Rt}
Let $n$ be a non-negative integer and 
let $X$ be a scheme of finite type over $R = k[t][[x_1, \ldots, x_N]]$. 
Suppose that each irreducible component $X_i$ of $X$ has $\dim ' X_i \ge n+1$. 
Then there exists a positive integer $c$ such that the following hold 
for non-negative integers $m$ and $e$ with $m \ge ce$. 
\begin{enumerate}
\item It follows that 
\begin{align*}
&\psi _{m} \left( \operatorname{Cont}^{e} \bigl( \operatorname{Fitt}^n (\Omega' _{X/k[t]}) \bigr) 
	\cap \operatorname{Cont}^{\ge 1}(\mathfrak{o}_{X}) \right) \\
&= \pi _{m+e, m} \left( \operatorname{Cont}^{e}\bigl( \operatorname{Fitt}^n (\Omega' _{X/k[t]}) \bigr)_{m+e} 
	\cap \operatorname{Cont}^{\ge 1}(\mathfrak{o}_{X})_{m+e} \right). 
\end{align*}

\item $\pi_{m+1, m}:X_{m+1} \to X_m$ induces a piecewise trivial fibration 
\begin{align*}
\psi _{m+1} & \left( \operatorname{Cont}^e  \bigl( \operatorname{Fitt}^n (\Omega' _{X/k[t]}) \bigr) 
	\cap \operatorname{Cont}^{\ge 1}(\mathfrak{o}_{X}) \right) \\
	& \to 
\psi _{m} \left( \operatorname{Cont}^e \bigl( \operatorname{Fitt}^n (\Omega' _{X/k[t]}) \bigr) 
	\cap \operatorname{Cont}^{\ge 1}(\mathfrak{o}_{X}) \right)
\end{align*}
with fiber $\mathbb{A}^n$. 
\end{enumerate}
\end{prop}
\begin{proof}
The same proof as in \cite[Proposition 2.17]{NS} works. We shall give a sketch of the proof. 

We may assume that $X$ is affine, and we may write $X = \operatorname{Spec} (S/I_X)$, where 
\[
S := k[t][[x_1, \ldots , x_N]][y_1, \ldots , y_m]
\]
and $I_X$ is an ideal of $S$. 
By the assumption and Lemma \ref{lem:dim'}(1), we have
\[
\operatorname{ht} P = \dim S - \dim' (S/P) \le (N+m+1) - (n+1) = N+m-n
\]
for any minimal prime $P$ of $I_X$. We set $r := N+m-n$. 

Let $f_1, \ldots , f_d$ be generators of $I_X$. 
For $1 \le i \le r$, we set $F_i := \sum _{j=1}^d a_{ij}f_j$ for general $a_{ij} \in k$. 
Let $M \subset \operatorname{Spec} S$ be the closed subscheme defined by the ideal $I_M := (F_1, \ldots, F_r)$. 
We denote 
\[
I_{X'} := (I_M : I_X) \subset S, \qquad J := \mathcal{J}_{r} \bigl(I_M; \operatorname{Der}_{k[t]}(S) \bigr) \subset S. 
\]
Here, we claim that 
\begin{itemize}
\item[($\spadesuit$)]
$J \subset \sqrt{I_X + I_{X'}}$ holds. 
\end{itemize}
We note that if ($\spadesuit$) is true, then 
the assertions for $X$ can be reduced to those for $M$ by the same argument as in \cite[Proposition 2.17]{NS}. 
Therefore, the assertions follow from Lemma \ref{lem:EM4.1CI_Rt}. 

Let $\mathfrak{p}$ be a prime ideal satisfying $I_X + I_{X'} \subset \mathfrak{p}$. 
To prove ($\spadesuit$), it is sufficient to show that $S/I_M$ is not regular at $\mathfrak{p}$. 
Indeed, if $S/I_M$ is not regular at $\mathfrak{p}$, then we have 
\[
J = \mathcal{J}_{r} \bigl(I_M; \operatorname{Der}_{k[t]}(S) \bigr) \subset \mathcal{J}_{r} \bigl(I_M; \operatorname{Der}_{k}(S) \bigr) \subset \mathfrak{p}
\]
by $\operatorname{ht}(I_M S_{\mathfrak{p}}) \le r$ and the Jacobian criterion of regularity (Remark \ref{rmk:JC}(2)(a)(b)).

Suppose the contrary that $M$ is regular at $\mathfrak{p}$. 
Since any minimal prime $P$ of $I_X$ satisfies $\operatorname{ht} P \le r$ and $a_{ij} \in k$ are general, 
for any irreducible component $X_0$ of $X$, 
there exists an irreducible component $M_0$ of $M$ such that $X_0 \subset M_0$ and $(X_0)_{\rm red} = (M_0)_{\rm red}$. 
Therefore, since $I_M \subset I_X \subset \mathfrak{p}$ and $M$ is regular at $\mathfrak{p}$, 
we have $(I_M)_{\mathfrak{p}} = (I_X)_{\mathfrak{p}}$. 
It shows that 
\[
(I_{X'})_{\mathfrak{p}} = (I_M:I_X)_{\mathfrak{p}} = \bigl( (I_M)_{\mathfrak{p}}:(I_X)_{\mathfrak{p}} \bigr) = S_{\mathfrak{p}}, 
\]
which contradicts $I_{X'} \subset \mathfrak{p}$. We complete the proof of ($\spadesuit$). 
\end{proof}

For an $R$-scheme $X$, a subset $C \subset X_{\infty}$ is called a \textit{cylinder} if $C = \psi_{m} ^{-1}(S)$ holds for some $m \ge 0$ and 
a constructible subset $S \subset X_m$. 

\begin{prop}\label{prop:const_R}
Let $n$ be a non-negative integer and 
let $X$ be a scheme of finite type over $R = k[t][[x_1, \ldots, x_N]]$. 
Suppose that each irreducible component $X_i$ of $X$ has $\dim ' X_i \ge n+1$. 
Let $C$ be a cylinder in $X_{\infty}$ which is contained in 
$\operatorname{Cont}^{\ge 1}(\mathfrak{o}_{X}) \cap \operatorname{Cont}^{e} \bigl( \operatorname{Fitt}^n (\Omega' _{X/k[t]}) \bigr)$
for some $e \ge 0$. 
Then its image $\psi _m (C) \subset X_m$ is a constructible subset for any $m \ge 0$. 
\end{prop}
\begin{proof}
Let $S \subset X_{\ell}$ be a constructible subset such that $\psi^{-1}_{\ell} (S) = C$. 
For $m \ge \ell$, we have 
\[
\pi_{m, \ell}^{-1} (S) \cap \psi_{m} (C) = \psi _m (C) = \pi_{m, \ell}^{-1} (S) \cap \psi_{m} (X_{\infty}). 
\]
By the assumption 
$C \subset \operatorname{Cont}^{\ge 1}(\mathfrak{o}_{X}) 
	\cap \operatorname{Cont}^{e} \bigl( \operatorname{Fitt}^n (\Omega' _{X/k[t]}) \bigr)$, 
we also have 
\[
\psi_{m} (C) = \pi_{m, \ell}^{-1} (S) \cap \psi_{m} \bigl( \operatorname{Cont}^{\ge 1}(\mathfrak{o}_{X}) 
	\cap \operatorname{Cont}^{e} \bigl( \operatorname{Fitt}^n (\Omega' _{X/k[t]}) \bigr) \bigr). 
\]
Let $c$ be the positive integer appearing in Proposition \ref{prop:EM4.1_Rt}. 
Then the constructibility of $\psi_{m} (C)$ follows from Proposition \ref{prop:EM4.1_Rt}(1) if $m \ge \max \{ce, \ell \}$. 
When $m < \max \{ce, \ell \}$, the constructibility follows from that for $m = \max \{ce, \ell \}$ and Chevalley's theorem. 
\end{proof}

We define the codimensions of cylinders $C$ when they satisfy $C \subset \operatorname{Cont}^{\ge 1}(\mathfrak{o}_{X})$. 
\begin{defi}\label{defi:codim_R}
Let $n$ be a non-negative integer and 
let $X$ be a scheme of finite type over $R = k[t][[x_1, \ldots, x_N]]$. 
Suppose that each irreducible component $X_i$ of $X$ has $\dim ' X_i \ge n+1$. 
Let $C \subset X_{\infty}$ be a cylinder contained in $\operatorname{Cont}^{\ge 1}(\mathfrak{o}_{X})$. 
\begin{enumerate}
\item 
Assume that $C \subset \operatorname{Cont}^e \bigl( \operatorname{Fitt}^n (\Omega' _{X/k[t]}) \bigr)$ 
for some $e \in \mathbb{Z}_{\ge 0}$.
Then we define the codimension of $C$ in $X_\infty$ as
\[
\operatorname{codim}(C) := (m+1) n - \operatorname{dim}(\psi_m (C))
\]
for any sufficiently large $m$. This definition is well-defined by Proposition \ref{prop:EM4.1_Rt}(2).

\item 
In general, we define the codimension of $C$ in $X_\infty$ as follows:
\[
\operatorname{codim}(C) := \min_{e \in \mathbb{Z}_{\ge 0}} {\operatorname{codim} \bigl(C \cap 
	\operatorname{Cont}^e \bigl( \operatorname{Fitt}^n (\Omega' _{X/k[t]}) \bigr) \bigr)}. 
\]
By convention, $\operatorname{codim}(C) = \infty$ if 
$C \cap \operatorname{Cont}^e \bigl( \operatorname{Fitt}^n (\Omega' _{X/k[t]}) \bigr) = \emptyset$
for any $e \ge 0$. 
\end{enumerate}
\end{defi}

\begin{rmk}\label{rmk:codim}
The definition of the codimension above depends on the choice of $n$. 
In Subsection \ref{subsubsection:thin}, we fix a non-negative integer $n$, 
and we use the codimension defined for this $n$. 
\end{rmk}

\begin{lem}\label{lem:EM4.4_R}
Let $X$ be a scheme of finite type over $R = k[t][[x_1, \ldots , x_N]]$. 
Let $p$ and $m$ be non-negative integers with $2p+1 \ge m \ge p$. 
Let $\gamma \in X_p (k)$ be a jet. 
If $\pi _{m,p}^{-1} (\gamma) \not = \emptyset$, it follows that 
\[
\pi _{m,p}^{-1} (\gamma) \simeq \operatorname{Hom}_{k[t]/(t^{p+1})} 
\bigl(\gamma ^* \Omega' _{X/k[t]}, (t^{p+1})/(t^{m+1}) \bigr). 
\]
\end{lem}
\begin{proof}
We may assume that $X$ is affine, and we may write $X = \operatorname{Spec} A$ with an $R$-algebra $A$ of finite type. 
Let $\gamma ^* : A \to k[t]/(t^{p+1})$ be the corresponding $k[t]$-ring homomorphism to $\gamma$. 
Take any $\alpha \in \pi _{m,p}^{-1} (\gamma)$. 
Let $\alpha ^* : A \to k[t]/(t^{m+1})$ be the corresponding $k[t]$-ring homomorphism. 
Then for the same reason as in the case of $k$-schemes (cf.\ \cite[Proposition 4.4]{EM09}), we have an isomorphism
\[
\pi _{m,p}^{-1} (\gamma) \simeq \operatorname{Der}_{k[t]} \left( A, (t^{p+1})/(t^{m+1}) \right); \quad  \beta \mapsto \beta^* - \alpha^*.
\]
Here, $(t^{p+1})/(t^{m+1})$ in the right-hand side has an $A$-module structure via $\gamma ^*$. 
Then the assertion follows from the isomorphisms
\begin{align*}
\operatorname{Der}_{k[t]} \left( A, (t^{p+1})/(t^{m+1}) \right)
&= \operatorname{Der}' _{k[t]} \left( A, (t^{p+1})/(t^{m+1}) \right) \\
&\simeq \operatorname{Hom}_{A} \bigl( \Omega' _{A/k[t]}, (t^{p+1})/(t^{m+1}) \bigr). 
\end{align*}
Here, the first equality follows from Lemma \ref{lem:D=D'}. 
\end{proof}

\subsubsection{Thin and very thin cylinders}\label{subsubsection:thin}

We fix a non-negative integer $n$ throughout this subsection. 

\begin{defi}\label{defi:vt}
Let $X$ be a scheme of finite type over $R = k[t][[x_1, \ldots , x_N]]$. 
Suppose that each irreducible component $X_i$ of $X$ has $\dim ' X_i \ge n+1$. 
A subset $A \subset X_{\infty}$ is called \textit{thin} 
if $A \subset Z_{\infty}$ holds for some closed subscheme $Z$ of $X$ with $\dim Z \le n$. 
$A$ is called \textit{very thin} if $A \subset Z_{\infty}$ holds for some closed subscheme $Z$ of $X$ with $\dim Z \le n-1$. 
\end{defi}

The term ``very thin" is used only in this paper. 
In Question \ref{quest}(1) and Remark \ref{rmk:quest}(1), we shall explain the motivation to introduce this terminology. 

\begin{quest}\label{quest}
Let $R$ and $X$ be as in Definition \ref{defi:vt}. 
\begin{enumerate}
\item Suppose that $C$ is a thin cylinder of $X_{\infty}$. 
Then, does $C \cap \operatorname{Cont}^e \bigl( \operatorname{Fitt}^n (\Omega' _{X/k[t]}) \bigr) = \emptyset$ hold for any $e \ge 0$?  

\item Suppose that $X$ is an integral scheme and $Y \subset X$ is the closed subscheme defined by the ideal $\operatorname{Fitt}^n (\Omega' _{X/k[t]})$. 
Then, is $Y_{\infty}$ a thin set of $X_{\infty}$? 

\item Let $S = k[t][[x_1, \ldots, x_N]][y_1, \ldots , y_m]$, and 
let $P$ be a prime ideal of $S$ of height $r$. 
Suppose that $P \cap k[t] = (0)$. 
Then, does $\mathcal{J}_{r}\bigl( P; \operatorname{Der}_{k[t]}(S) \bigr) \not \subset P$ hold? 
\end{enumerate}
\end{quest}

\begin{rmk}\label{rmk:quest}
\begin{enumerate}
\item 
Note that Question \ref{quest}(1) is true for the arc spaces of $k$-varieties $X$: 
\begin{itemize}
\item 
If $C$ is a thin cylinder of $X_{\infty}$, 
then $C \cap \operatorname{Cont}^{e} (\operatorname{Jac}_{X/k}) = \emptyset$ holds for any $e \ge 0$ 
(cf.\ \cite[Lemma 5.1]{EM09}). 
\end{itemize}
The same statement is true for schemes $X$ of finite type over $k[[x_1, \ldots , x_N]]$ 
by replacing $\operatorname{Jac}_{X/k}$ with $\operatorname{Jac}'_{X/k}$. 
Furthermore, Question \ref{quest}(1) is true also for schemes $X$ of finite type over $k[t]$ (cf.\ \cite[Lemma 2.23]{NS}). 
However, the same proofs do not work for schemes $X$ of finite type over $R = k[t][[x_1, \ldots , x_N]]$, 
and hence it is not clear to us whether Question \ref{quest}(1) is true for this setting (see also the discussion in (\ref{item:N=0}) below). 
This is why we introduce the term ``very thin" and we will prove weaker statements instead in Lemma \ref{lem:thin_R} for very thin cylinders 
and Proposition \ref{prop:resol} for $X$ with an additional assumption. 

\item 
Due to the proof of \cite[Lemma 5.1]{EM09}, 
Question \ref{quest}(1) can be reduced to Question \ref{quest}(2) by the Noetherian induction on dimension. 
Furthermore, Question \ref{quest}(3) implies Question \ref{quest}(2). 

\item \label{item:N=0}
Question \ref{quest}(3) is related to the weak Jacobian condition (WJ) explained in Remark \ref{rmk:JC}(2). 
Indeed, if $N = 0$, then Question \ref{quest}(3) can be proved using Remark \ref{rmk:JC}(2)(c) as follows. 
We denote
\[
S' := (k[t] \setminus \{ 0 \})^{-1} S = k(t)[y_1, \ldots , y_m]
\]
the localization. Then by the assumption $P \cap k[t] = (0)$, 
we have $PS' \not = S'$ and hence $PS'$ is a prime ideal of height $r$. 
Since $S'$ satisfies (WJ)$_{k(t)}$, we have 
\[
\mathcal{J}_{r}\bigl( P ; \operatorname{Der}_{k[t]}(S) \bigr) S' + PS' = 
\mathcal{J}_{r}\bigl( PS' ; \operatorname{Der}_{k(t)}(S') \bigr) + PS' \not \subset PS',  
\]
which proves $\mathcal{J}_{r}\bigl( P; \operatorname{Der}_{k[t]}(S) \bigr) \not \subset P$. 
Note that the same proof does not work when $N > 0$ because we have 
\[
S' := (k[t] \setminus \{ 0 \})^{-1} S \not = k(t)[[x_1, \ldots , x_N]][y_1, \ldots , y_m], 
\] 
and it is not clear whether 
$\mathcal{J}_{r}\bigl( PS' ; \Delta \bigr) \not \subset PS'$
holds for $\Delta = \bigl\{ \partial / \partial x_i, \partial / \partial y_j \ \big|\ i,j \bigr\}$. 

\item 
Question \ref{quest}(1) is also true for the arc spaces (Greenberg schemes) of formal $k[[t]]$-schemes of finite type, 
which will be dealt with in Subsection \ref{subsection:formal} (see \cite[Ch.6.\ Proposition 2.4.3]{CLNS}). 
Actually, Question \ref{quest}(3) is true for this setting: 
\begin{itemize}
\item Let $S = k[x_1, \ldots, x_N][[t]]$, and let $P$ be a prime ideal of $S$ of height $r$. 
Suppose that $P \cap k[[t]] = (0)$. Then we have $\mathcal{J}_{r}\bigl( P; \operatorname{Der}_{k[[t]]}(S) \bigr) \not \subset P$. 
\end{itemize}
We denote 
\[
S' := S_t = k[x_1, \ldots , x_N]((t))
\]
the localization. 
We note that the assumption $P \cap k[[t]] = (0)$ is equivalent to $t \not \in P$, and hence $PS'$ is a prime ideal of height $r$. 
Since $S'$ satisfies (WJ)$_{k((t))}$ by \cite[Theorem 46.3]{Nag62}, the same proof as in (\ref{item:N=0}) above works and 
we have $\mathcal{J}_{r}\bigl( P; \operatorname{Der}_{k[[t]]}(S) \bigr) \not \subset P$. 
Note here that both $\operatorname{Der}_{k[[t]]}(S)$ and $\operatorname{Der}_{k((t))}(S')$ are generated by 
$\partial / \partial x_i$'s. 
\end{enumerate}
\end{rmk}

In Lemmas \ref{lem:forget} and \ref{lem:thin_R} below, for a scheme $X$ over $R = k[t][[x_1, \ldots , x_N]]$, 
we also consider the jet schemes and the arc space in the sense of Section \ref{section:jet}. 
To avoid confusion, we denote them by $\mathcal{L}_m (X)$ and $\mathcal{L}_{\infty} (X)$, that is, 
$\mathcal{L}_m(X)$ is the scheme representing the functor 
\[
F_{m}: ({\sf Sch}/k) \to ({\sf Set}); \quad 
Y \mapsto \operatorname{Hom} _{k} \left( Y \times_{\operatorname{Spec} k} \operatorname{Spec} k[t]/(t^{m+1}), X \right)
\]
and $\mathcal{L}_{\infty}(X) = \mathop{\varprojlim} \limits_{m} \mathcal{L}_m (X)$ is the projective limit. 

\begin{lem}\label{lem:forget}
Let $X$ be a scheme of finite type over $R = k[t][[x_1, \ldots , x_N]]$. 
Then the following hold. 
\begin{enumerate}
\item There exist natural closed immersions $X_{m} \to \mathcal{L}_m (X)$ for $m \ge 0$ which commute 
with the truncation morphisms. 

\item Let $x \in X$ be a $k$-point 
and let $\widehat{\mathcal{O}}_{X,x}$ be the completion of the local ring $\mathcal{O}_{X,x}$ at its maximal ideal.
Let $X' := \operatorname{Spec} (\widehat{\mathcal{O}}_{X,x})$ and let $x' \in X'$ be the corresponding $k$-point. 
Then for the truncation morphisms $\pi ^{X} _m : \mathcal{L}_m (X) \to X$ and $\pi ^{X'}_m : \mathcal{L}_m(X') \to X'$, 
we have $\bigl( \pi ^{X} _m \bigr)^{-1} (x) \simeq \bigl( \pi ^{X'} _m \bigr)^{-1} (x')$. 
\end{enumerate}
\end{lem}

\begin{proof}
We may assume that $X$ is affine, and we may write $X = \operatorname{Spec} (S/I)$ with 
\[
S := k[t][[x_1, \ldots , x_N]][y_1, \ldots, y_M]
\]
and $I := (f_1, \ldots, f_r)$ an ideal of $S$. We set $A := \operatorname{Spec} S$. 

Then, by the same argument as in Lemma \ref{lem:repl}, 
we have $\mathcal{L}_m(A) \simeq \operatorname{Spec} T_m$, where
\[
T_m := 
k \bigl[ u^{(0)} \bigr] \bigl[ \bigl[ x_1^{(0)}, \ldots, x_N^{(0)} \bigr] \bigr]
\biggl[ u^{(s)}, x_j^{(s)}, y_{j'}^{(s')} 
\ \bigg | \ 
\begin{array}{l}
 1 \le j \le N,\ 1\le j' \le M, \\
 1 \le s \le m,\ 0 \le s' \le m 
\end{array}
  \biggr].
\]
In the same way as in Lemma \ref{lem:repl}, we can define a ring homomorphism 
$\Lambda' : S \to T_m[t]/(t^{m+1})$ that satisfies
\begin{itemize}
\item $\Lambda'(t) = u^{(0)} + u^{(1)}t + \cdots + u^{(m)} t^m$, 
\item $\Lambda'(x_j) = x_j^{(0)} + x_j^{(1)}t + \cdots + x_j^{(m)}t^m$ for each $0 \le j \le N$, and
\item $\Lambda'(y_j) = y_j^{(0)} + y_j^{(1)}t + \cdots + y_j^{(m)}t^m$ for each $0 \le j \le M$.
\end{itemize}
For $1 \le i \le r$ and $0 \le \ell \le m$, 
we define $G_i ^{(\ell)} \in T_m$ as 
\[
\Lambda ' (f_i) = \sum _{\ell = 0} ^{m} G_i ^{(\ell)} t^{\ell} \pmod{t^{m+1}}. 
\]
Let \[
J_m := \bigl( G_i ^{(s)} \ \big| \  1 \le i \le r, \ 0 \le s \le m \bigr) \subset T_m
\]
be the ideal generated by $G^{(s)}_i$'s. 
Then by the same argument as in Lemma \ref{lem:repl2}, 
we have $\mathcal{L}_m (X) \simeq \operatorname{Spec} (T_m / J_m)$. 

Let $S_m$ be the ring defined in Lemma \ref{lem:repl}. 
Let $\Xi: T_m \to S_m$ be a surjective ring homomorphism defined by 
\begin{itemize}
\item $\Xi(u^{(1)}) = 1$, and $\Xi(u^{(s)}) = 0$ for each $s=0$ and $2 \le s \le m$. 
\item $\Xi(x^{(s)}_j) = x^{(s)}_j$ for each $1 \le j \le N$ and $0 \le s \le m$. 
\item $\Xi(y^{(s)}_j) = y^{(s)}_j$ for each $1 \le j \le M$ and $0 \le s \le m$. 
\end{itemize}
We note that $\Lambda : S \to S_m[t]/(t^{m+1})$ defined in the proof of Lemma \ref{lem:repl} coincides with the composition
\[
S \xrightarrow{\Lambda '} T_m[t]/(t^{m+1}) \xrightarrow{\Xi} S_m[t]/(t^{m+1}), 
\]
where $T_m[t]/(t^{m+1}) \to S_m[t]/(t^{m+1})$ is the $k[t]$-ring homomorphism induced by $\Xi$. 
Therefore, $F^{(s)}_i$ in Lemma \ref{lem:repl2} coincides with $\Xi (G^{(s)}_i)$ for each $i$ and $s$. 

Let $I_m \subset S_m$ be the ideal defined in Lemma \ref{lem:repl2}. 
Then $\Xi$ induces a surjective ring homomorphism $T_m / J_m \to S_m /I_m$. 
It gives a closed immersion
\[
X_{m} \simeq \operatorname{Spec} (S_m / I_m) \hookrightarrow  \operatorname{Spec} (T_m / J_m) \simeq \mathcal{L}_m (X), 
\]
which completes the proof of (1). 

Since $x \in X$ is a $k$-point, the corresponding maximal ideal of $S$ is of the form
\[
\bigl( t-a, x_1, \ldots , x_N, y_1 - b_1, \ldots, y_M - b_M \bigr)
\]
with $a, b_1, \ldots , b_M \in k$. 
For simplicity, we assume that $a = b_1 = \cdots = b_M = 0$. 
Then $\bigl( \pi ^{X} _m \bigr)^{-1} (x) \subset \mathcal{L}_m (X)$ 
is isomorphic to the closed subscheme of $\operatorname{Spec} T_m$ 
defined by 
\[
J_m + \bigl( u^{(0)}, x_1^{(0)}, \ldots, x_N^{(0)}, y_1^{(0)}, \ldots, y_M^{(0)} \bigr). 
\]

On the other hand, we have 
\[
\widehat{\mathcal{O}}_{A,x} \simeq k[[t, x_1, \ldots , x_N, y_1, \ldots, y_M ]]. 
\]
We set $A' := \operatorname{Spec} (\widehat{\mathcal{O}}_{A,x})$ and 
\[
T' _m := 
k \bigl[ \bigl[ u^{(0)}, x_1^{(0)}, \ldots, x_N^{(0)}, y_1^{(0)}, \ldots, y_M^{(0)} \bigr] \bigr]
\left[ u^{(s)}, x_j^{(s)}, y_{j'}^{(s)} 
\ \middle | \ 
\begin{array}{l}
 1 \le j \le N, \\
 1\le j' \le M, \\
 1 \le s \le m
\end{array}
  \right].
\]
Then by \cite[Proposition 4.1]{Ish09}, we have $\mathcal{L}_m (A') \simeq \operatorname{Spec} T'_m$. 
Furthermore, by \cite[Corollary 4.2]{Ish09}, we have $\mathcal{L}_m (X') \simeq \operatorname{Spec} (T'_m/J_mT'_m)$. 
Therefoere, $\bigl( \pi ^{X'} _m \bigr)^{-1} (x') \subset \mathcal{L}_m (X')$ 
is isomorphic to the closed subscheme of $\operatorname{Spec} T'_m$ 
defined by 
\[
J_mT'_m + \bigl( u^{(0)}, x_1^{(0)}, \ldots, x_N^{(0)}, y_1^{(0)}, \ldots, y_M^{(0)} \bigr). 
\]
Therefore we have $\bigl( \pi ^{X} _m \bigr)^{-1} (x) \simeq \bigl( \pi ^{X'} _m \bigr)^{-1} (x')$, which completes the proof of (2). 
\end{proof}

\begin{lem}\label{lem:thin_R}
Let $X$ be a scheme of finite type over $R = k[t][[x_1, \ldots , x_N]]$. 
Suppose that each irreducible component $X_i$ of $X$ has $\dim ' X_i \ge n+1$. 
Let $C \subset X_{\infty}$ be a cylinder contained in $\operatorname{Cont}^{\ge 1}(\mathfrak{o}_{X})$. 
If $C$ is very thin, then 
$C \cap \operatorname{Cont}^{e} \bigl( \operatorname{Fitt}^n (\Omega' _{X/k[t]}) \bigr) = \emptyset$ holds for any $e \ge 0$. 
\end{lem}
\begin{proof}
Suppose the contrary that 
$C \cap \operatorname{Cont}^{e} \bigl( \operatorname{Fitt}^n (\Omega' _{X/k[t]}) \bigr) \not = \emptyset$ for some $e \ge 0$. 
By replacing $C$ with $C \cap \operatorname{Cont}^{e} \bigl( \operatorname{Fitt}^n (\Omega' _{X/k[t]}) \bigr)$, 
we may assume that $\emptyset \not = C \subset \operatorname{Cont}^{e} \bigl( \operatorname{Fitt}^n (\Omega' _{X/k[t]}) \bigr)$. 
Pick a $k$-arc $\gamma \in C$. Let $x := \pi^X _{\infty} (\gamma) \in X$ be the $k$-point of $X$. 
Then, by replacing $C$ with $C \cap (\pi^X_{\infty}) ^{-1}(x)$, we may assume that $C \subset (\pi ^X_{\infty}) ^{-1}(x)$. 

Since $C$ is a very thin set, there exists a closed subscheme $Z \subset X$ such that $C \subset Z_{\infty}$ and $\dim Z \le n-1$. 
Since $\gamma \in Z_{\infty}$, it follows that $x \in Z$. 
Let $\widehat{\mathcal{O}}_{Z,x}$ be the completion of the local ring $\mathcal{O}_{Z,x}$ at its maximal ideal.
Let $Z' := \operatorname{Spec} (\widehat{\mathcal{O}}_{Z,x})$, and let $x' \in Z'$ be the corresponding $k$-point. 
Then, since $C \subset (\pi_{\infty}^Z) ^{-1}(x)$, we may identify $C$ 
with a subset of $\mathcal{L}_{\infty} (Z')$ by Lemma \ref{lem:forget}(1)(2). 
Note that $C \subset \mathcal{L}_{\infty} (Z')$ is not necessarily a cylinder of $\mathcal{L}_{\infty} (Z')$ 
under this identification. 

Let $\mathcal{S}$ be the set of the closed subschemes $Y'$ of $Z'$ with the following condition: 
\begin{itemize}
\item There exists a cylinder $C'$ of $X_{\infty}$ such that 
$\emptyset \not = C' \subset C$ and $C' \subset \mathcal{L}_{\infty} (Y')$. 
\end{itemize}
Here, the inclusion $C' \subset \mathcal{L}_{\infty} (Y')$ is considered by 
the identifications $\mathcal{L}_{\infty} (Y') \subset \mathcal{L}_{\infty} (Z')$ and $C' \subset C \subset \mathcal{L}_{\infty} (Z')$. 
Let $Y'$ be a minimal element of $\mathcal{S}$, and let $C'$ be a corresponding cylinder of $X_{\infty}$. 
Then $Y'$ is reduced by the minimality. 

We shall prove that $Y'$ is irreducible. 
Suppose the contrary that $Y' = Y' _1 \cup  \cdots \cup Y' _{\ell}$ is the irreducible decomposition with $\ell \ge 2$. 
By the minimality of $Y'$, it follows that $C' \not \subset \mathcal{L}_{\infty} (Y' _1)$ and hence we have
\[
C'' := C' \cap \bigl( (\psi^{Z'} _{q})^{-1}( \mathcal{L}_q (Y' _1)) \setminus (\psi^{Z'} _{q+1})^{-1}( \mathcal{L}_{q+1} (Y' _1))  \bigr) \not = \emptyset
\]
for some $q \ge -1$,  
where we set $(\psi^{Z'} _{q})^{-1}( \mathcal{L}_q (Y' _1)) = \mathcal{L}_{\infty} (Z')$ for $q = -1$ by abuse of notation. 
Here, we have taken the intersection in the space $\mathcal{L}_{\infty}(Z')$. 
Since $C'' \cap \mathcal{L}_{\infty} (Y' _1) = \emptyset$, we have 
\[
C'' \subset C \setminus \mathcal{L}_{\infty} (Y' _1) 
\subset \mathcal{L}_{\infty} (Y') \setminus \mathcal{L}_{\infty} (Y' _1) 
\subset \mathcal{L}_{\infty} (Y'_{2} \cup  \cdots \cup Y' _{\ell}). 
\]
To get a contradiction by the minimality of $Y'$, it is sufficient to show that $C''$ is a cylinder of $X_{\infty}$. 
For this purpose, we shall see that 
\[
C'_q := C' \cap (\psi^{Z'} _{q})^{-1}( \mathcal{L}_q (Y' _1))
\]
is a cylinder of $X_{\infty}$. Under the following identifications
\[
\xymatrix{
& X_{\infty} \cap (\pi^{X}_{\infty})^{-1}(x) \ar[r]^{\psi^X_q} & X_q \cap (\pi ^X _q)^{-1}(x) \\
& Z_{\infty} \cap (\pi^{Z}_{\infty})^{-1}(x) \ar[r] \ar@{}[u]|{\text{\large \rotatebox{90}{$\subset$}}} & Z_q \cap (\pi ^Z _q)^{-1}(x) \ar@{}[u]|{\text{\large \rotatebox{90}{$\subset$}}} \\
& \mathcal{L}_{\infty}(Z) \cap (\pi^{Z}_{\infty})^{-1}(x) \ar[r] \ar@{}[u]|{\text{\large \rotatebox{-90}{$\subset$}}} & \mathcal{L}_{q}(Z) \cap (\pi ^{Z} _q)^{-1}(x) \ar@{}[u]|{\text{\large \rotatebox{-90}{$\subset$}}}  \\
C' \ \ \subset \hspace{-7mm} & \mathcal{L}_{\infty}(Z') \cap (\pi^{Z'}_{\infty})^{-1}(x') \ar[r]^{\psi^{Z'}_q} \ar[u]^{\simeq} & \mathcal{L}_{q}(Z') \cap (\pi ^{Z'} _q)^{-1}(x') \ar[u]^{\simeq} \\
& \mathcal{L}_{\infty}(Y'_1) \cap (\pi^{Y'_1}_{\infty})^{-1}(x') \ar[r] \ar@{}[u]|{\text{\large \rotatebox{90}{$\subset$}}} & \mathcal{L}_{q}(Y'_1) \cap (\pi ^{Y'_1} _q)^{-1}(x'), \ar@{}[u]|{\text{\large \rotatebox{90}{$\subset$}}}
}
\]
we can consider the intersection $F = \mathcal{L}_q (Y'_1) \cap Z_q \cap (\pi ^{Z} _q)^{-1}(x)$ and 
it can be identified with a closed subset of $X_q \cap (\pi ^X _q)^{-1}(x)$. 
Then we have $C'_q = C' \cap (\psi_{q}^X)^{-1}(F)$ since $C' \subset Z_{\infty}$. 
Therefore, $C'_q$ turns out to be a cylinder of $X_{\infty}$, and hence so is $C'' = C'_q \setminus C'_{q+1}$. 
We have proved that $Y'$ is integral. 

Let $Y'' \subset Y'$ be the subscheme defined by $\operatorname{Jac}'_{Y'/k}$. 
Since $Y'$ is reduced, we have $Y'' \subsetneq Y'$ by the Jacobian criterion of regularity (cf.\ Remark \ref{rmk:JC}(2)(c)). 
By the minimality of $Y'$, we have $C' \not \subset Y''_{\infty}$ and hence 
$C' \cap \operatorname{Cont}^{e'} \bigl( \operatorname{Jac}'_{Y'/k} \bigr) \not = \emptyset$ holds for some $e' \ge 0$. 
Take a $k$-arc $\beta \in C' \cap \operatorname{Cont}^{e'} \bigl( \operatorname{Jac}'_{Y'/k} \bigr)$. 
For $m \ge 0$, we set 
\[
D_{m,\beta} := (\psi^{Y'}_m) ^{-1} (\psi^{Y'} _m (\beta)) \subset \mathcal{L}_{\infty}(Y'), \quad 
E_{m,\beta} := (\psi^{X}_m) ^{-1} (\psi^{X} _m (\beta)) \subset X_{\infty}. 
\]
Then by applying Proposition \ref{prop:EM4.1_R}(2) to the map
\[
\psi^{Y'}_{m+1} (D_{m,\beta}) \to 
\psi^{Y'} _{m} (D_{m,\beta}) = \bigl\{ \psi^{Y'} _m (\beta) \bigr\}, 
\]
we have 
\[
\dim \bigl( \psi^{Y'}_{m+1} (D_{m,\beta}) \bigr) = \dim Y'
\]
for sufficiently large $m$. 
On the other hand, by applying Proposition \ref{prop:EM4.1_Rt}(2) to the map
\[
\psi^{X}_{m+1} (E_{m,\beta}) \to 
\psi^{X} _{m} (E_{m,\beta}) = \bigl\{ \psi^{X} _m (\beta) \bigr\}, 
\]
we have 
\[
\dim \bigl( \psi^{X}_{m+1} (E_{m,\beta})) \bigr) = n
\]
for sufficiently large $m$. 

Since $\dim Y' \le \dim Z \le n -1$, 
to get a contradiction, it is enough to show $E_{m,\beta} \subset D_{m,\beta}$ for sufficiently large $m$. 
Since $C'$ is a cylinder of $X_{\infty}$, there exists a constructible subset $V \subset X_p$ for some $p \ge 0$ 
such that $C' = (\psi _p ^X) ^{-1} (V)$. 
We shall prove the inclusion $E_{m,\beta} \subset D_{m,\beta}$ for any $m \ge p$. 

Let $X' := \operatorname{Spec} (\widehat{\mathcal{O}}_{X,x})$. Then by Lemma \ref{lem:forget}(1)(2), we have the following diagram:
\[
\xymatrix{
X_{\infty} \cap (\pi^{X}_{\infty})^{-1}(x) \ar[d]_{\psi ^X _m} \ar@{}[r]|{\hspace{-3mm}\text{\large $\subset$}}  &  \mathcal{L}_{\infty}(X') \cap (\pi^{X'}_{\infty})^{-1}(x') \ar[d]^{\psi ^{X'} _m} & \mathcal{L}_{\infty}(Y') \cap (\pi^{Y'}_{\infty})^{-1}(x') \ar[d]^{\psi ^{Y'} _m} \ar@{}[l]|{\text{\large $\supset$}}\\
X_{m} \cap (\pi^{X}_{m})^{-1}(x)  \ar@{}[r]|{\hspace{-3mm}\text{\large $\subset$}} & \mathcal{L}_{m}(X') \cap (\pi^{X'}_{m})^{-1}(x') & \mathcal{L}_{m}(Y') \cap (\pi^{Y'}_{m})^{-1}(x') \ar@{}[l]|{\text{\large $\supset$}}
}
\]
Let $\beta_m := \psi^{X} _m (\beta)$. For $m \ge p$, we have 
\[
E_{m,\beta} = (\psi^{X}_m) ^{-1} (\beta _m) = (\psi^{X'}_m) ^{-1} (\beta _m) \cap X_{\infty} \cap (\pi^{X}_{\infty})^{-1}(x) = (\psi^{X'}_m) ^{-1} (\beta _m) \cap C'. 
\]
On the other hand, we have 
\[
D_{m,\beta} = (\psi^{Y'}_m) ^{-1} (\beta _m) = (\psi^{X'}_m) ^{-1} (\beta _m) \cap \mathcal{L}_{\infty}(Y') \cap (\pi^{Y'}_{m})^{-1}(x').
\]
Since $C' \subset \mathcal{L}_{\infty}(Y') \cap (\pi^{Y'}_{m})^{-1}(x')$, 
we have $E_{m,\beta} \subset D_{m,\beta}$ for $m \ge p$. We complete the proof. 
\end{proof}

We prove a much weaker version of \cite[Lemma 2.26(1)]{NS}. 

\begin{lem}\label{lem:thin2_R}
Let $X$ be a scheme of finite type over $R = k[t][[x_1, \ldots , x_N]]$. 
For each $a \in k$, we denote by $X_a$ the closed subscheme of $X$ defined by $(t-a)\mathcal{O}_X$. 
Suppose for any $a \in k^{\times}$ that 
$X_a$ is an integral regular scheme and has $\dim ' X_a = n$. 
Then there exists a positive integer $\ell$ such that 
$\operatorname{ord}_{\gamma} \bigl( \mathfrak{o}_{X} + \operatorname{Fitt}^n (\Omega' _{X/k[t]}) \bigr) \le \ell$ 
holds for any $k$-arc $\gamma \in X_{\infty}$. 
In particular, if $\gamma$ satisfies $\operatorname{ord}_{\gamma} (\mathfrak{o}_{X}) = \infty$, 
then $\operatorname{ord}_{\gamma} \bigr( \operatorname{Fitt}^n (\Omega' _{X/k[t]}) \bigr) \le \ell$. 
\end{lem}
\begin{proof}
We may assume that $X$ is affine, and we may write $X = \operatorname{Spec} (S/I)$, where 
\[
S := k[t][[x_1, \ldots, x_N]][y_1, \ldots , y_m]
\]
and $I$ is an ideal of $S$. 
We set $I_a := (I + (t-a))/(t-a)$, which is an ideal of the ring $S/(t-a) \simeq k[[x_1, \ldots, x_N]][y_1, \ldots , y_m]$. 
We have $\operatorname{ht} (I_a) = N + m -n$ for any $a \in k^{\times}$ since $\dim ' X_a = n$. 

We set 
\[
J := \mathcal{J}_{N+m-n} \bigl( I;\operatorname{Der}_{k[t]}(S) \bigr)  \subset S.   
\]
Note that $\operatorname{Der}_{k[t]}(S) = \operatorname{Der}'_{k[t]}(S)$ 
is generated by $\frac{\partial}{\partial x_i}$'s and $\frac{\partial}{\partial y_i}$'s. 
Then we have $\operatorname{Fitt}^n (\Omega' _{X/k[t]}) = (J + I) /I$ by Remark \ref{rmk:jacobianmatrix}. 
Let $a \in k^{\times}$. 
Since $\operatorname{ht} (I_a) = N + m -n$, we have
\[
\operatorname{Jac}'_{X_a/k} = (J + I + (t-a))/ (I + (t-a)). 
\]
Since $X_a$ is regular, we have 
\[
J + I + (t-a) = S
\]
by the Jacobian criterion of regularity (cf.\ Remark \ref{rmk:JC}(2)(c)). 
Therefore, for any $a \in k^{\times}$,  we have 
\[
(J + I + (t-a)) S' = S', 
\]
where $S' := S/(x_1, \ldots , x_N) \simeq k[t][y_1, \ldots , y_m]$. 
Then by Hilbert's nullstellensatz, we have $t ^{\ell} \in J + I + (x_1, \ldots , x_N)$ for some $\ell \ge 0$, 
which proves the assertions. 
\end{proof}

\begin{lem}\label{lem:ht1'}
Let $P$ be a prime ideal of $S = k[[x_1, \ldots, x_N]]$ of height $r$, and 
let $I$ be an ideal of $S$ satisfying $P \subsetneq I$. 
If $S/P$ is regular, then $\mathcal{J}_{r+1}(I;\operatorname{Der}_k(S)) \not \subset I$. 
\end{lem}
\begin{proof}
Note that $\operatorname{Der}_k(S) = \operatorname{Der}' _k(S)$ is generated by $\partial / \partial x_i$'s. 
First, we prove the assertion when $r=0$. 
Let $f \in I \setminus \{ 0 \}$ be an element with the minimum order $a$. 
Suppose that $x_i$ appears in the lowest order term of $f$. 
Then it follows from the minimality of $a$ that $\frac{\partial f}{\partial x_i} \not \in I$, 
which proves the assetion when $r=0$. 

Suppose $r > 0$. Since $S/P$ is regular, by the Jacobian criterion of regularity (cf.\ Remark \ref{rmk:JC}(2)(c)), 
there exist $D_1, \ldots, D_r \in \operatorname{Der}_k (S)$ and $f_1, \ldots , f_r \in P$ such that 
\[
u := \det (D_i(f_j))_{1\le i,j \le r} \not \in (x_1, \ldots , x_N). 
\]
Since $S/P$ is a complete regular local ring with the coefficient field $k$, 
$S/P$ is isomorphic to $k[[y_1, \ldots, y_{N-r}]]$. 
Therefore, by what we have already proved, 
there exist $D' \in \operatorname{Der}_k (S/P)$ and $f' \in I/P$ such that $D'(f') \not \in I/P$. 
Let $f_{r+1} \in I$ be a lift of $f'$. 
By \cite[Theorem 30.8]{Mat89}, there exists a lift $D_{r+1} \in \operatorname{Der}_k (S)$ of $D'$ too. 
Since $D_{r+1}(P) = 0$, we have
\[
\det (D_i(f_j))_{1\le i,j \le r+1} = u D_{r+1}(f_{r+1}) \not \in I, 
\]
which shows that $\mathcal{J}_{r+1}(I;\operatorname{Der}_k(S)) \not \subset I$. 
\end{proof}

\begin{rmk}
We are interested in the case where $\operatorname{ht} I = r+1$. 
If $I$ is a prime ideal, then it is true more generally that 
$\mathcal{J}_{\ell}(I; \operatorname{Der}_k(S)) \not \subset I$ for $\ell = \operatorname{ht} I$. 
It is true because $S$ satisfies the weak Jacobian condition (WJ)$_k$ (cf.\ Remark \ref{rmk:JC}(2)). 
If $I$ is not a prime ideal, 
then $\mathcal{J}_{\ell}(I; \operatorname{Der}_k(S)) \not \subset I$ 
does not hold in general. 
\end{rmk}

\begin{lem}\label{lem:ht1}
Let $P$ be a prime ideal of $S = k[[x_1, \ldots, x_N]][y_1, \ldots , y_m]$ of height $r$, and 
let $I$ be an ideal of $S$ satisfying $P \subsetneq I$. 
Suppose that $S/P$ is regular and $I + (x_1, \ldots, x_N) \not = S$. 
Then it follows that $\mathcal{J}_{r+1}(I;\operatorname{Der}_{k}(S)) \not \subset I$. 
\end{lem}
\begin{proof}
Note that $\operatorname{Der}_k(S) = \operatorname{Der}'_k(S)$ is generated by $\partial / \partial x_i$'s and $\partial / \partial y_i$'s. 
Since $I + (x_1, \ldots, x_N) \not = S$, there exists a maximal ideal $\mathfrak{m}$ containing $I$ of the form 
\[
\mathfrak{m} = (x_1, \ldots , x_N, y_1 - a _1, \ldots , y_m - a_m), 
\] 
where $a_i \in k$. Let $\widehat{S}$ be the completion of $S$ at $\mathfrak{m}$. 
Let $Y_i \in \widehat{S}$ be the image of $y_i - a_i$. 
Then we have $\widehat{S} \simeq k[[x_1, \ldots , x_N, Y_1, \ldots , Y_m]]$, and $\operatorname{Der}_k (\widehat{S})$ is generated by
$\partial / \partial x_i$'s and $\partial / \partial Y_i$'s. 
Therefore we have 
\[
\mathcal{J}_{r+1} \bigl( I\widehat{S};\operatorname{Der}_{k}(\widehat{S}) \bigr) +I \widehat{S} = 
\mathcal{J}_{r+1}(I;\operatorname{Der}_{k}(S)) \widehat{S} + I \widehat{S}. 
\]
We also note that $\widehat{S}/P\widehat{S}$ is regular and $P\widehat{S} \subsetneq I\widehat{S}$. 
Therefore by Lemma \ref{lem:ht1'}, we have 
$\mathcal{J}_{r+1} \bigl( I\widehat{S};\operatorname{Der}_{k}(\widehat{S}) \bigr) \not \subset I\widehat{S}$, 
which shows the assertion $\mathcal{J}_{r+1}(I;\operatorname{Der}_{k}(S)) \not \subset I$. 
\end{proof}

\begin{prop}\label{prop:resol}
Let $X$ be a scheme of finite type over $R = k[t][[x_1, \ldots , x_N]]$. 
Suppose that each irreducible component $X_i$ of $X$ has $\dim ' X_i \ge n+1$. 
For each $a \in k$, we denote by $X_a$ the closed subscheme of $X$ defined by $(t-a)\mathcal{O}_X$. 
Suppose for any $a \in k^{\times}$ that 
$X_a$ is an integral regular scheme and has $\dim ' X_a = n$. 
Then, there is no thin cylinder $C$ of $X_{\infty}$ containing a $k$-arc $\gamma$ 
with $\operatorname{ord}_{\gamma} (\mathfrak{o}_{X}) = \infty$. 
\end{prop}

\begin{proof}
We may assume that $X$ is affine, and we may write $X = \operatorname{Spec} A$ with $A = S/I$, where 
\[
S := k[t][[x_1, \ldots, x_N]][y_1, \ldots , y_m]
\]
and $I$ is an ideal of $S$. 

Suppose the contrary that there exists a thin cylinder $C$ containing a 
$k$-arc $\gamma$ with $\operatorname{ord}_{\gamma} (\mathfrak{o}_{X}) = \infty$. 
Replacing $C$ with $C \cap \operatorname{Cont}^{\ge 1} (\mathfrak{o}_{X})$, 
we may assume that $C \subset \operatorname{Cont}^{\ge 1} (\mathfrak{o}_{X})$. 
By Lemma \ref{lem:thin2_R}, it follows that
$e := \operatorname{ord}_{\gamma} \bigr( \operatorname{Fitt}^n (\Omega' _{X/k[t]}) \bigr) < \infty$. 
By replacing $C$ with $C \cap \operatorname{Cont}^{e} \bigr( \operatorname{Fitt}^n (\Omega' _{X/k[t]}) \bigr)$, 
we may assume that
$\emptyset \not =  C \subset \operatorname{Cont}^{e} \bigr( \operatorname{Fitt}^n (\Omega' _{X/k[t]}) \bigr)$. 

Let $\mathcal{S}$ be the set of the closed subschemes $W$ of $X$ with the following condition: 
\begin{itemize}
\item There exists a cylinder $C'$ of $X_{\infty}$ such that 
$\gamma \in C' \subset C$ and $C' \subset W_{\infty}$. 
\end{itemize}
Let $W$ be a minimal element of $\mathcal{S}$, and let $C'$ be a corresponding cylinder of $X$. 
Then $W$ is reduced by the minimality. 
Let $W = W_1 \cup \cdots \cup W_{\ell}$ be its irreducible decomposition. 
Since $C$ is thin, we may assume $\dim W \le n$. 
Here, we claim as follows: 
\begin{claim}\label{claim:resol}
\begin{enumerate}
\item[(1)] $\gamma \in (W_i)_{\infty}$ holds for each $1 \le i \le \ell$. 
\item[(2)] $\dim W_i = n$ holds for each $1 \le i \le \ell$. 
\item[(3)] We denote by $Z_i \subset W_i$ the closed subscheme defined by $\operatorname{Fitt}^{n-1} \bigl( \Omega' _{W_i/k[t]} \bigr)$. 
Then $Z_i \subsetneq W_i$ holds for each $1 \le i \le \ell$.
\end{enumerate}
\end{claim}

First, we assume this claim and finish the proof. 
By Claim \ref{claim:resol}(3) and the minimality of $W$, we have 
\[
C' \not \subset (Z_1 \cup \cdots \cup Z_{\ell})_{\infty} = (Z_1)_{\infty} \cup \cdots \cup (Z_{\ell})_{\infty}. 
\]
Take a $k$-arc $\beta \in C' \setminus \bigl( (Z_1)_{\infty} \cup \cdots \cup (Z_{\ell})_{\infty} \bigr)$. 
For each $i$, we denote by $I_{Z_i} \subset A$ the ideal corresponding to $Z_i$, 
and we set $q_i := \operatorname{ord}_{\beta} (I_{Z_i}) < \infty$. 
Then 
\[
C'' := C' \cap \bigcap _i \operatorname{Cont}^{q_i} (I_{Z_i})
\]
is a non-empty cylinder of $X_{\infty}$. 
By applying Proposition \ref{prop:EM4.1_Rt}(2) to $W_i$ and its cylinder 
\[
C'' \cap (W_i)_{\infty} \subset \operatorname{Cont}^{q_i} \bigl( \operatorname{Fitt}^{n-1} \bigl( \Omega' _{W_i/k[t]} \bigr) \bigr), 
\]
it follows that the truncation map
\[
\psi_{m+1} \bigl( C'' \cap (W_i)_{\infty} \bigr) \to \psi_{m} \bigl( C'' \cap (W_i)_{\infty} \bigr)
\]
has $(n-1)$-dimensional fibers for sufficiently large $m$. 
Therefore, 
\[
\psi_{m+1} ( C'') = \bigcup _i \psi_{m+1} \bigl( C'' \cap (W_i)_{\infty} \bigr)  \to \psi_{m} ( C'') = \bigcup _i \psi_{m} \bigl( C'' \cap (W_i)_{\infty} \bigr)
\]
also has $(n-1)$-dimensional fibers for sufficiently large $m$. 
However, by Proposition \ref{prop:EM4.1_Rt}(2), it should have $n$-dimensional fibers 
because $\emptyset \not = C'' \subset \operatorname{Cont}^{e} \bigr( \operatorname{Fitt}^n (\Omega' _{X/k[t]}) \bigr)$. 
We get a contradiction. Therefore, it is sufficient to prove Claim \ref{claim:resol}(3). 

\begin{proof}[Proof of Claim \ref{claim:resol}]
We shall prove (1). 
Suppose the contrary that $\gamma \not \in (W_1)_{\infty}$. 
Let $I_{W_1} \subset A$ be the ideal corresponding to $W_1$ and 
let $q := \operatorname{ord}_{\gamma} (I_{W_1}) < \infty$. 
Then the cylinder
\[
C'' := C' \cap \operatorname{Cont}^{q}(I_{W_1})
\]
contains $\gamma$ and satisfies 
\[
C'' \subset W_{\infty} \setminus (W_1)_{\infty} \subset (W_2 \cup \cdots \cup W_{\ell})_{\infty}, 
\]
which contradicts the minimality of $W$. 

We shall prove (2). Suppose the contrary that $\dim W_1 \le n-1$. 
Let $W' := W_2 \cup \cdots \cup W_{\ell}$ and let $I_{W'} \subset A$ be the ideal corresponding to $W'$. 
By the minimality of $W$, it follows that $C' \not \subset W'_{\infty}$. 
Therefore, we have 
\[
C'' := C' \cap \operatorname{Cont}^{q}(I_{W'}) \not = \emptyset
\] for some $q \ge 0$. 
Since $C'' \cap W'_{\infty} = \emptyset$, we have $C'' \subset (W_1)_{\infty}$ and $C''$ 
turns out to be a very thin cylinder of $X_{\infty}$. 
It contradicts $\emptyset \not = C'' \subset \operatorname{Cont}^{e} \bigr( \operatorname{Fitt}^n (\Omega' _{X/k[t]}) \bigr)$
by Lemma \ref{lem:thin_R}.

We shall prove (3). Let $H$ be one of $W_i$'s. 
Let $Q$ be the prime ideal of $S$ corresponding to $H$. 
Since $H_{\infty}$ contains a $k$-arc, $H$ contains a $k$-point. 
Therefore, by Lemma \ref{lem:dim'}(4), we have 
\[
\operatorname{ht} Q = \dim S - \dim' H = \dim S - \dim H = N+m - n + 1. 
\]

First, we prove that 
\[
Q + (x_1, \ldots , x_N) + (t-a) \not = S \tag{$\clubsuit$}
\]
for some $a \in k^{\times}$. 
Suppose the contrary that $Q + (x_1, \ldots , x_N) + (t-a) = S$ holds for any $a \in k^{\times}$. 
Then by Hilbert's nullstellensatz, it follows that 
$t^{\ell} \in Q + (x_1, \ldots , x_N)$ for some $\ell \ge 0$. 
This contradicts $\gamma \in H_{\infty}$ and $\operatorname{ord}_{\gamma} (\mathfrak{o}_{X}) = \infty$, 
and we get ($\clubsuit$) for some $a \in k^{\times}$. 

($\clubsuit$) implies $Q + (t-a) \not = S$. 
Furthermore, we have $t-a \not \in Q$ because $H_{\infty}$ contains a $k$-arc. 
Therefore, we have $\operatorname{ht} (Q + (t-a)) = \operatorname{ht} Q + 1 = N + m - n + 2$. 

We set 
\begin{align*}
S_a &:= S/(t-a) \simeq k[[x_1, \ldots, x_N]][y_1, \ldots , y_m], \\
I_a &:= (I + (t-a))/(t-a), \quad Q_a := (Q+(t-a))/(t-a). 
\end{align*}
Then we have $\operatorname{ht} (Q_a) = N + m - n + 1$. 
Furthermore, we have $\operatorname{ht} (I_a) = N+m-n$ by the assumption $\dim ' X_a = n$. 
Therefore it follows that $Q_a \supsetneq I_a$. 
Let $J := \mathcal{J}_{N+m-n+1} \bigl( Q ; \operatorname{Der}_{k[t]}(S) \bigr)$. 
Then the ideal $(J+(t-a))/(t-a)$ of $S_a$ 
coincides with $\mathcal{J}_{N+m-n+1} \bigl( Q_a ; \operatorname{Der}_{k}(S_a) \bigr)$. 
Note that $S_a/I_a$ is regular, $\operatorname{ht}(I_a) = N+m-n$, and $Q_a \supsetneq I_a$. 
Therefore by Lemma \ref{lem:ht1}, we have 
\[
(J + (t-a))/(t-a) = \mathcal{J}_{N+m-n+1} \bigl( Q_a ; \operatorname{Der}_{k}(S_a) \bigr) \not \subset Q_a = (Q + (t-a))/(t-a). 
\]
In particular, we have $J \not \subset Q$. 

Since $\operatorname{Fitt}^{n-1} (\Omega' _{H/k[t]}) = (J+Q)/Q$, we complete the proof of (3). 
\end{proof}
\end{proof}

\begin{lem}\label{lem:action}
Let $R = k[t][[x_1, \ldots , x_N]]$. Let $e_1, \ldots, e_N$ and $d$ be integers satisfying 
$0 < e_i \le d$ for each $i$. 
For each $c \in k^{\times}$, let $T_c: R \to R$ be the ring isomorphism defined by 
$T_c(t) = c^{-d}t$ and $T_c(x_i) = c^{e_i}x_i$. 
Let $I$ be an ideal of $R$ that is $T_c$-invariant (i.e. $T_c(I)=I$ holds) for any $c \in k^{\times}$. 
Let $P$ be a minimal prime of $I$. Then $P$ satisfies one of the following conditions. 
\begin{enumerate}
\item $P \cap k[t] \not = (0)$ and $t \in P$. 
\item $P \cap k[t] = (0)$, and $P + (t-a) \not = R$ holds for any $a \in k^{\times}$. 
\item $P \cap k[t] = (0)$, and there exists $f \in P$ such that
\[
f - t^{\ell} \in (t^{\ell +1}x_1, \ldots, t^{\ell +1}x_N) + (x_1, \ldots , x_N)^{\ell +1}
\]
holds for some $\ell \ge 0$. 
\end{enumerate}
\end{lem}

\begin{proof}
First, we prove that $P$ is also $T_c$-invariant for any $c \in k^{\times}$. 
Let $P_1, \ldots, P_m$ be the minimal primes of $I$.
Since  $T_c$ is an isomorphism, $T_c$ induces a permutation on $P_1, \ldots, P_m$. 
Let $p: k^{\times} \to \mathfrak{S}_m$ be the induced group homomorphism, 
where $\mathfrak{S}_m$ is the symmetric group of degree $m$. 
For any $c \in k^{\times}$, we can take $b \in k^{\times}$ such that $c = b^{m!}$. 
Therefore, we have $p(c) = p(b^{m!}) = (p(b))^{m!} = 1$. 
It shows that $T_c(P_i) = P_i$ for any $c \in k^{\times}$ and $1 \le i \le m$. 

Suppose $P \cap k[t] \not = (0)$. Then $t-a \in P$ for some $a \in k$. 
Since $P$ is $T_c$-invariant for any $c \in k^{\times}$, it follows that $a = 0$. 
Therefore, $P$ satisfies (1). 

Suppose that $P \cap k[t] = (0)$ and $P + (t-a) = R$ holds for some $a \in k^{\times}$. 
We shall prove that $P$ satisfies (3). 
Since $P$ is $T_c$-invariant for any $c \in k^{\times}$, 
it follows that $P + (t-a) = R$ holds for any $a \in k^{\times}$. 
Then by Hilbert's nullstellensatz, it follows that 
\begin{itemize}
\item 
$t^{\ell} \in P + (x_1, \ldots , x_N)$ for some $\ell \ge 0$. 
\end{itemize}
Therefore there exists $g \in P$ such that $g - t^{\ell} \in (x_1, \ldots , x_N)$. 

We denote $M := (x_1, \ldots , x_N) \subset k[t][[x_1, \ldots, x_N]]$. 
Since $M$ is $T_c$-invariant, $T_c$ induces an automorphism on $k[t][[x_1, \ldots , x_N]]/M^{\ell + 1}$. 
Hence, $k[t][[x_1, \ldots , x_N]]/M^{\ell + 1}$ has a graded ring structure satisfying 
$\deg t = -d$ and $\deg x_i = e_i$. 
Then $(P + M^{\ell + 1})/M^{\ell + 1}$ is a homogeneous ideal. 
Therefore, the term $g_{-d \ell}$ of $g$ with degree $-d \ell$ is contained in $P + M^{\ell + 1}$. 
We may write $g_{-d \ell} = f - h$ with $f \in P$ and $h \in M^{\ell + 1}$. 
On the other hand, since $g_{- d \ell} - t^{\ell} \in M$, 
we have $g_{-d \ell} - t^{\ell} \in (t^{\ell +1}x_1, \ldots, t^{\ell +1}x_N)$ by looking at the degrees of its terms. 
Therefore, the condition (3) holds for this $f$. 
\end{proof}

\begin{rmk}\label{rmk:kpt}
Let $I$ and $P$ be as in Lemma \ref{lem:action}. 
Then the following hold for $Y := \operatorname{Spec} (R/P)$: 
\begin{itemize}
\item $Y_{\infty} = \emptyset$ if $P$ is of the form (1). 
\item $Y_{\infty} \cap \operatorname{Cont}^{\ge 1} (\mathfrak{o}_{Y}) = \emptyset$ holds
if $P$ is of the form (3). 
\end{itemize}
\end{rmk}

\subsection{Arc spaces of affine formal $k[[t]]$-schemes}\label{subsection:formal}

In this subsection, we discuss the arc space of $X$ of the form $X = \operatorname{Spec} \bigl( k[x_1, \ldots , x_N][[t]] / I \bigr)$. 
As we will mention in Remark \ref{rmk:Sebag}, the arc space of $X$ can be seen as the Greenberg scheme of the corresponding affine formal scheme. 
In this subsection, we do not deal with general formal $k[[t]]$-schemes. 

\begin{rmk}\label{rmk:Sebag}
Sebag in \cite{Seb04} investigates the theory of arc spaces of formal $k[[t]]$-schemes with $k$ a perfect field, 
and the theory can be applied to $X = \operatorname{Spec} \bigl( k[x_1, \ldots , x_N][[t]] / I \bigr)$ dealt with in this subsection. 
The reader is also referred to \cite{CLNS} to this theory. 

For a scheme $X = \operatorname{Spec} \bigl( k[x_1, \ldots , x_N][[t]] / I \bigr)$, 
we can associate the formal affine scheme $\mathcal{X} = \operatorname{Spf} \bigl( k[x_1, \ldots , x_N][[t]] / I \bigr)$. 
Then the Greenberg schemes $\operatorname{Gr}_{m}(\mathcal{X})$ and $\operatorname{Gr}(\mathcal{X})$ 
defined in \cite{Seb04} are isomorphic to $X_m$ and $X_{\infty}$, respectively. 
Therefore, the theory of Greenberg schemes developed in \cite{Seb04} and \cite{CLNS} 
can be applied to the arc space $X_{\infty}$ of $X$. 
\end{rmk}

\begin{defi}[{cf.\ \cite[Appendix 3.3]{CLNS}}]
Let $I$ be an ideal of $S = k[x_1, \ldots , x_N][[t]]$ and let $A := S/I$. 
Then we denote by $\widehat{\Omega}_{A/k[[t]]}$ the completion of the $A$-module $\Omega_{A/k[[t]]}$ 
with respect to the $(t)$-adic topology, i.e.
\[
\widehat{\Omega}_{A/k[[t]]} := \varprojlim_n \left(\Omega_{A/k[[t]]}/(t^n)\Omega_{A/k[[t]]} \right). 
\]
The canonical derivation $d_{A/k[[t]]}: A \to \Omega_{A/k[[t]]}$ induces a derivation
\[
\widehat{d}_{A/k[[t]]}: A \to \widehat{\Omega}_{A/k[[t]]}. 
\]
We sometimes abbreviate  $\widehat{d}_{A/k[[t]]}$ to $\widehat{d}$. 

When $X=\operatorname{Spec} A$, 
we denote by $\widehat{\Omega}_{X/k[[t]]}$ the sheaf on $X$ associated to the $A$-module $\widehat{\Omega}_{A/k[[t]]}$. 
\end{defi}

\begin{rmk}[{cf.\ \cite[Example 3.3.5 in Appendix]{CLNS}}]\label{rmk:hat}
\begin{enumerate}
\item 
$\widehat{\Omega}_{S/k[[t]]}$ is a free $S$-module of rank $N$ with basis 
\[
\widehat{d}_{S/k[[t]]}(x_1), \ldots , \widehat{d}_{S/k[[t]]}(x_N). 
\]
Furthermore, we have an exact sequence
\[
I/I^2 \xrightarrow{\delta} \widehat{\Omega} _{S/k[[t]]} \otimes _{S} A \xrightarrow{\alpha} \widehat{\Omega} _{A/k[[t]]} \to 0
\]
of $A$-modules, where 
$\alpha$ is the map satisfying $\alpha \bigl( \widehat{d}_{S/k[[t]]}(g) \otimes 1 \bigr) = \widehat{d}_{A/k[[t]]} ( \overline{g} )$ for $g \in S$, and 
$\delta$ is the map satisfying $\delta(\overline{g}) = \widehat{d}_{S/k[[t]]}(g) \otimes 1$ for $g \in I$. 
In particular, $\widehat{\Omega} _{A/k[[t]]}$ is a finite $A$-module.

\item \label{item:univ_hat}
The canonical derivation $\widehat{d}: A \to \widehat{\Omega} _{A/k[[t]]}$ has the following universal property: 
\begin{itemize}
\item[(\ref{item:univ_hat}-1)] The induced map 
\[
\operatorname{Hom}_A \bigl( \widehat{\Omega} _{A/k[[t]]}, M \bigr) \to 
\operatorname{Der}_{k[[t]]} \left( A, M \right); \quad f \mapsto f \circ \widehat{d}
\] 
is an isomorphism for any $A$-module $M$ that is complete with respect to the $(t)$-adic topology. 
In particular, this map is an isomorphism for any finite $A$-module (cf.\ \cite[Theorem 8.7]{Mat89}). 
\end{itemize}
This follows from the following general fact from \cite[20.4.8.2]{EGAIV1}: 
\begin{itemize}
\item[(\ref{item:univ_hat}-2)] 
Let $B$ be a topological ring and $C$ a topological $B$-algebra. 
Let $N$ be a topological $C$-module. Then we have an isomorphism
\[
\operatorname{Hom}^{\rm c}_C (\Omega _{C/B}, N) \xrightarrow{\simeq}  \operatorname{Der}^{\rm c}_B (C, N); \quad f \mapsto f \circ d_{C/B}. 
\]
Here, $\operatorname{Hom}^{\rm c}_C (\Omega _{C/B}, N)$ denotes the set of the continuous homomorphisms 
$\Omega _{C/B} \to N$ of $C$-modules, and 
$\operatorname{Der}^{\rm c}_B (C, N)$ denotes the set of the continuous $B$-derivations $C \to N$. 
The topology on $\Omega _{C/B} = \mathcal{I}/\mathcal{I}^2$ is defined as the quotient topology, 
where $\mathcal{I}$ is the kernel of the augmentation map $C \otimes _B C \to C;\ a \otimes b \mapsto ab$. 
\end{itemize}
In our case, the topology on $\Omega _{A/k[[t]]}$ coincides with the $(t)$-adic topology (cf.\ \cite[20.4.5]{EGAIV1}). 
Therefore, we have 
\begin{align*}
\operatorname{Hom}^{\rm c} _A \bigl( \Omega _{A/k[[t]]}, M \bigr) &= \operatorname{Hom}_A \bigl( \Omega _{A/k[[t]]}, M \bigr), \\
\operatorname{Der}^{\rm c} _{k[[t]]} ( A, M ) &= \operatorname{Der}_{k[[t]]} ( A, M )
\end{align*}
for any $A$-module $M$ with the $(t)$-adic topology. 
Hence, by (\ref{item:univ_hat}-2), we have an isomorphism 
\[
\operatorname{Hom}_A \bigl( \Omega _{A/k[[t]]}, M \bigr) \xrightarrow{\simeq}
\operatorname{Der}_{k[[t]]} ( A, M )
\]
for any $A$-module $M$. 
Moreover, if $M$ is complete with respect to the $(t)$-adic topology, then we have an isomorphism
\[
\operatorname{Hom}_A \bigl( \widehat{\Omega} _{A/k[[t]]}, M \bigr) \xrightarrow{\simeq}
\operatorname{Hom}_A \bigl( \Omega _{A/k[[t]]}, M \bigr), 
\]
which proves (\ref{item:univ_hat}-1). 

\item \label{item:Fitt_hat}
By (\ref{item:univ_hat}), it follows that 
$\operatorname{Der}_{k[[t]]}(S)$ is a free $S$-module of rank $N$ generated by $\partial/\partial x_i$'s. 
Therefore, by the exact sequence in (1), we have 
\[
\operatorname{Fitt}^n \bigl( \widehat{\Omega} _{A/k[[t]]} \bigr) = 
\bigl( \mathcal{J}_{N-n} \bigl( I; \operatorname{Der}_{k[[t]]}(S) \bigr) + I  \bigr)/I. 
\]
\end{enumerate}
\end{rmk}

\begin{lem}\label{lem:order_k[[t]]}
Let $n$ and $e$ be non-negative integers, and 
let $X$ be an affine scheme of the form $X = \operatorname{Spec} \bigl( k[x_1, \ldots , x_N][[t]]/I \bigr)$. 
Suppose that each irreducible component $X_i$ of $X$ has $\dim X_i \ge n+1$. 
Let $\gamma \in \operatorname{Cont}^e \bigl( \operatorname{Fitt}^n ( \widehat{\Omega} _{X/k[[t]]} ) \bigr)$ be a $k$-arc. 
Then we have
\[
\gamma ^* \widehat{\Omega} _{X / k[[t]]} \simeq 
k[[t]]^{\oplus n} \oplus \bigoplus _i k[t]/(t^{e_i})
\]
as $k[[t]]$-modules with $\sum_i e_i = e$. 
\end{lem}
\begin{proof}
The same proofs as in Lemmas \ref{lem:finjac} and \ref{lem:order_R} work. 
Note that any minimal prime $P$ of $I$ satisfies $\operatorname{ht} P \le N - n$. 
This is because 
\[
\operatorname{ht} P = \dim S - \dim (S/P) \le (N + 1) - (n+1) = N - n, 
\]
where we set $S := k[x_1, \ldots , x_N][[t]]$. 
The first equality follows from the facts 
that any maximal ideal $M$ of $S$ has $\operatorname{ht} M = N+1$ and $S$ is a catenary ring. 
\end{proof}

\begin{prop}\label{prop:EM4.1_k[[t]]}
Let $n$ be a non-negative integer, and 
let $X$ be an affine scheme of the form $X = \operatorname{Spec} \bigl( k[x_1, \ldots , x_N][[t]]/I \bigr)$. 
Suppose that each irreducible component $X_i$ of $X$ has $\dim X_i \ge n+1$. 
Then, there exists a positive integer $c$ such that the following hold 
for non-negative integers $m$ and $e$ with $m \ge ce$. 
\begin{enumerate}
\item We have
\begin{align*}
\psi _{m} \Bigl( \operatorname{Cont}^{e} \bigl( \operatorname{Fitt}^n \bigl( \widehat{\Omega} _{X/k[[t]]} \bigr) \bigr) \Bigr) 
= \pi _{m+e, m} \Bigl( \operatorname{Cont}^{e} \bigl( \operatorname{Fitt}^n \bigl( \widehat{\Omega} _{X/k[[t]]} \bigr) \bigr)_{m+e} \Bigr). 
\end{align*}

\item $\pi_{m+1, m}:X_{m+1} \to X_m$ induces a piecewise trivial fibration 
\[
\psi _{m+1} \Bigl( \operatorname{Cont}^{e} \bigl( \operatorname{Fitt}^n \bigl( \widehat{\Omega} _{X/k[[t]]} \bigr) \bigr) \Bigr) \to 
\psi _{m} \Bigl( \operatorname{Cont}^{e} \bigl( \operatorname{Fitt}^n \bigl( \widehat{\Omega} _{X/k[[t]]} \bigr) \bigr) \Bigr)
\]
with fiber $\mathbb{A}^n$. 
\end{enumerate}
\end{prop}

\begin{proof}
The same proof as in Proposition \ref{prop:EM4.1_Rt} works (cf.\ Remark \ref{rmk:Hensel}). 
\end{proof}

\begin{rmk}
When $X$ is flat over $k[[t]]$, 
Proposition \ref{prop:EM4.1_k[[t]]}(1) is proved in \cite[Ch.5.\ Proposition 2.3.4]{CLNS}, and 
Proposition \ref{prop:EM4.1_k[[t]]}(2) is proved in \cite[Lemme 4.5.4]{Seb04} (cf.\ \cite[Ch.5.\ Theorem 2.3.11]{CLNS}). 
We also note that Proposition \ref{prop:EM4.1_k[[t]]}(2) can be reduced to the flat case by the argument in \cite[Remark 2.14(3)]{NS}. 
\end{rmk}

We define cylinders and their codimensions. 

\begin{defi}\label{defi:codim_k[[t]]}
Let $n$ be a non-negative integer, and 
let $X$ be an affine scheme of the form $X = \operatorname{Spec} \bigl( k[x_1, \ldots , x_N][[t]]/I \bigr)$. 
Suppose that each irreducible component $X_i$ of $X$ has $\dim X_i \ge n+1$. 
A subset $C \subset X_{\infty}$ is called a \textit{cylinder} if $C = \psi_{m} ^{-1}(S)$ holds for some $m \ge 0$ 
and a constructible subset $S \subset X_m$. We define the codimension of $C$ as follows: 
\begin{enumerate}
\item 
Assume that $C \subset \operatorname{Cont}^{e} \bigl( \operatorname{Fitt}^n \bigl( \widehat{\Omega} _{X/k[[t]]} \bigr) \bigr)$ 
for some $e \in \mathbb{Z}_{\ge 0}$.
Then we define the codimension of $C$ in $X_\infty$ as
\[
\operatorname{codim}(C) := (m+1) n - \operatorname{dim}(\psi_m (C))
\]
for any sufficiently large $m$. This definition is well-defined by Proposition \ref{prop:EM4.1_k[[t]]}.

\item 
In general, we define the codimension of $C$ in $X_\infty$ as follows:
\[
\operatorname{codim}(C) := \min_{e \in \mathbb{Z}_{\ge 0}} \operatorname{codim}
\Bigl( C \cap \operatorname{Cont}^{e} \bigl( \operatorname{Fitt}^n \bigl( \widehat{\Omega} _{X/k[[t]]} \bigr) \bigr) \Bigr) . 
\]
By convention, $\operatorname{codim}(C) = \infty$ if 
$C \cap \operatorname{Cont}^e \bigl( \operatorname{Fitt}^n \bigl( \widehat{\Omega} _{X/k[[t]]} \bigr) \bigr) = \emptyset$
for any $e \ge 0$. 
\end{enumerate}
\end{defi}

\begin{rmk}\label{rmk:codim_hat}
As in Remark \ref{rmk:codim}, the definition of the codimension above depends on the choice of $n$. 
\end{rmk}

\begin{defi}
Let $n$ be a non-negative integer, and 
let $X$ be an affine scheme of the form $X = \operatorname{Spec} \bigl( k[x_1, \ldots , x_N][[t]]/I \bigr)$. 
Suppose that each irreducible component $X_i$ of $X$ has $\dim X_i \ge n+1$. 
A subset $A \subset X_{\infty}$ is called \textit{thin} 
if $A \subset Z_{\infty}$ holds for some closed subscheme $Z$ of $X$ with $\dim Z \le n$. 
\end{defi}

\begin{prop}[{cf.\ \cite[Th\'{e}or\`{e}me 6.3.5]{Seb04}}]\label{prop:negligible}
Let $n$ be a non-negative integer, and 
let $X$ be an affine scheme of the form $X = \operatorname{Spec} \bigl( k[x_1, \ldots , x_N][[t]]/I \bigr)$. 
Suppose that each irreducible component $X_i$ of $X$ has $\dim X_i \ge n+1$. 
Let $C$ be a cylinder in $X_{\infty}$. 
Let $\{ C_{\lambda} \} _{\lambda \in \Lambda}$ be a set of countably many disjoint subcylinders $C_{\lambda} \subset C$. 
If $C \setminus (\bigsqcup _{\lambda \in \Lambda} C_{\lambda}) \subset X_{\infty}$ is a thin set, 
then it follows that 
\[
\operatorname{codim}(C) = \min _{\lambda \in \Lambda} \operatorname{codim}(C_{\lambda}). 
\]
\end{prop}
\begin{proof}
This follows from \cite[Ch.6.\ Lemma 3.4.1]{CLNS} and \cite[Ch.6.\ Example 3.5.2]{CLNS}. 
\end{proof}

\begin{lem}[{cf.\ \cite[Ch.5.\ Proposition 2.2.6]{CLNS}}]\label{lem:EM4.4_k[[t]]}
Let $X$ be an affine scheme of the form $X = \operatorname{Spec} \bigl( k[x_1, \ldots , x_N][[t]]/I \bigr)$. 
Let $p$ and $m$ be non-negative integers with $2p+1 \ge m \ge p$. 
Let $\gamma \in X_p (k)$ be a jet with $\pi _{m,p}^{-1} (\gamma) \not = \emptyset$. 
Then we have 
\[
\pi _{m,p}^{-1} (\gamma) \simeq \operatorname{Hom}_{k[t]/(t^{p+1})} 
\bigl(\gamma^* \widehat{\Omega}_{X/k[[t]]}, (t^{p+1})/(t^{m+1}) \bigr). 
\]
\end{lem}

\begin{proof}
We set $A := k[x_1, \ldots , x_N][[t]]/I$. 
For the same reason as in the proof of Lemma \ref{lem:EM4.4_R}, 
we have 
\[
\pi _{m,p}^{-1} (\gamma) \simeq \operatorname{Der}_{k[[t]]} \left( A, (t^{p+1})/(t^{m+1}) \right). 
\]
Furthermore, by the universal property of $\widehat{\Omega}_{A/k[[t]]}$ (cf.\ Remark \ref{rmk:hat}(\ref{item:univ_hat})), we have 
\[
\operatorname{Der}_{k[[t]]} \left( A, (t^{p+1})/(t^{m+1}) \right)
\simeq \operatorname{Hom}_{A} \bigl( \widehat{\Omega}_{A/k[[t]]}, (t^{p+1})/(t^{m+1}) \bigr), 
\]
which proves the assertion. 
\end{proof}

\subsection{Codimension formulae}

In this subsection, we discuss a $k[t]$-morphism $f: X \to Y$ of affine $k[t]$-schemes in the following two cases. 
\begin{enumerate}
\item[(a)] $X$ and $Y$ are affine schemes of the forms 
$X = \operatorname{Spec} \bigl( k[x_1, \ldots , x_N][[t]] /I \bigr)$ and 
$Y = \operatorname{Spec} \bigl( k[x_1, \ldots , x_M][[t]] /J \bigr)$. 

\item[(b)] $X$ and $Y$ are affine schemes of the forms 
$X = \operatorname{Spec} \bigl( k[x_1, \ldots , x_M][[t]] /I \bigr)$ and
$Y = \operatorname{Spec} \bigl( k[t][[x_1, \ldots , x_L]] /J \bigr)$. 
Furthermore, $f$ satisfies $(x_1, \ldots , x_L) \mathcal{O}_X \subset (t)$. 
\end{enumerate}

\begin{lem}\label{lem:omega'hat}
In case (a) above, the canonical map $f^* \widehat{\Omega} _{Y / k[[t]]} \to \widehat{\Omega} _{X / k[[t]]}$ is induced. 
In case (b), the canonical map $f^* \Omega'  _{Y / k[t]} \to \widehat{\Omega} _{X / k[[t]]}$ is induced. 
\end{lem}
\begin{proof}
Let $A = \mathcal{O}_X$ and $B = \mathcal{O}_Y$ be the corresponding rings, 
and $g : B \to A$ the corresponding $k[t]$-ring homomorphism. 

First, we deal with case (a). 
Since $g: B \to A$ is a $k[t]$-ring homomorphism, 
$\widehat{\Omega} _{A/k[[t]]}$ is a complete $B$-module with respect to the $(t)$-adic topology. 
Therefore, by the universal property of $\widehat{\Omega} _{B/k[[t]]}$ (cf.\ Remark \ref{rmk:hat}(\ref{item:univ_hat})), 
the derivation $\widehat{d}_{A/k[[t]]} \circ g : B \to \widehat{\Omega} _{A/k[[t]]}$ factors through 
$\widehat{\Omega} _{B/k[[t]]}$. We complete the proof in case (a). 

Next, we deal with case (b). 
By the universal property of $\Omega '_{B/k[t]}$, it is sufficient 
to show that the composition $B \xrightarrow{g} A \xrightarrow{\widehat{d}_{A/k[[t]]}} \widehat{\Omega} _{A/k[[t]]}$ 
is a special $B$-derivation. 
Note that $\widehat{\Omega} _{A/k[[t]]}$ is a complete $A$-module with respect to the $(t)$-adic topology. 
Since $g \bigl( (x_1, \ldots, x_L) \bigr) \subset (t)$ holds by assumption, 
$\widehat{\Omega} _{A/k[[t]]}$ is a separated $B$-module with respect to the $(x_1, \ldots , x_L)$-adic topology. 
Therefore $\widehat{d}_{A/k[[t]]} \circ g$ is a special $B$-derivation by Lemma \ref{lem:D=D'}. 
\end{proof}

We define the order of the Jacobian for a morphism. 

\begin{defi}\label{defi:jac_k[[t]]}
\begin{enumerate}
\item
Let $f: X \to Y$ be a morphism of affine $k[t]$-schemes of the form (a) above. 
Then $f$ induces a homomorphism  $f^* \widehat{\Omega} _{Y / k[[t]]} \to \widehat{\Omega} _{X / k[[t]]}$ by Lemma \ref{lem:omega'hat}.
Let $\gamma \in X _{\infty}$ be a $k$-arc and let $\gamma ' := f_{\infty} (\gamma)$. 
Let $S$ be the torsion part of $\gamma ^* \widehat{\Omega} _{X / k[[t]]}$. 
Then we define the \textit{order} $\operatorname{ord}_{\gamma} (\operatorname{jac}_f)$ 
\textit{of the Jacobian} of $f$ at $\gamma$ as the length of the $k[[t]]$-module
\[
\operatorname{Coker} \bigl( \gamma ^{\prime *} \widehat{\Omega} _{Y / k[[t]]} \to \gamma ^* \widehat{\Omega} _{X / k[[t]]} /S \bigr).
\]
In particular, if $\operatorname{ord}_{\gamma} (\operatorname{jac}_f) < \infty$, then we have
\[
\operatorname{Coker} \bigl( \gamma ^{\prime *} \widehat{\Omega} _{Y / k[[t]]} \to \gamma ^* \widehat{\Omega} _{X / k[[t]]} /S \bigr) 
\simeq 
\bigoplus _i k[t]/(t^{e_i})
\]
as $k[[t]]$-modules with some positive integers $e_i$ satisfying $\sum_i e_i = \operatorname{ord}_{\gamma} (\operatorname{jac}_f)$. 

\item
Let $f: X \to Y$ be a morphism of affine $k[t]$-schemes of the form (b) above. 
Then $f$ induces a homomorphism  $f^* \Omega' _{Y / k[t]} \to \widehat{\Omega} _{X / k[[t]]}$ by Lemma \ref{lem:omega'hat}. 
Let $\gamma \in X _{\infty}$ be a $k$-arc and let $\gamma ' := f_{\infty} (\gamma)$. 
Let $S$ be the torsion part of  $\gamma ^* \widehat{\Omega} _{X / k[[t]]}$. 
Then we also define the \textit{order} $\operatorname{ord}_{\gamma} (\operatorname{jac}_f)$ 
\textit{of the Jacobian} of $f$ at $\gamma$ as the length of the $k[[t]]$-module
\[
\operatorname{Coker} \bigl( \gamma ^{\prime *} \Omega' _{Y / k[t]} \to \gamma ^* \widehat{\Omega} _{X / k[[t]]} /S \bigr).
\]

\item By abuse of notation, we define 
\[
\operatorname{Cont}^e(\operatorname{jac}_f) := 
\{ \gamma \in X_{\infty} \mid \operatorname{ord}_{\gamma} (\operatorname{jac}_f) = e \}
\]
for $e \ge 0$. 
\end{enumerate}
\end{defi}

\begin{lem}\label{lem:additive_k[[t]]}
Let $n$ be a non-negative integer. 
Let $f: X \to Y$ be a morphism of affine $k[t]$-schemes of the form (a) above. 
Let $g: Y \to Z$ be a morphism of affine $k[t]$-schemes of the form (b).
Suppose that each irreducible component $W_i$ of $X$, $Y$ and $Z$ has $\dim W_i \ge n + 1$. 
Let $\gamma \in X_{\infty}$ be a $k$-arc and let $\gamma ' := f_{\infty}(\gamma)$. 
Suppose that 
\[
\operatorname{ord}_{\gamma} \bigl( \operatorname{Fitt}^n \bigl( \widehat{\Omega} _{X/k[[t]]} \bigr) \bigr) < \infty, \quad 
\operatorname{ord}_{\gamma'} \bigl( \operatorname{Fitt}^n \bigl( \widehat{\Omega} _{Y/k[[t]]} \bigr) \bigr) < \infty. 
\]
Then we have 
\[
\operatorname{ord}_{\gamma} (\operatorname{jac}_{g \circ f}) = 
\operatorname{ord}_{\gamma} (\operatorname{jac}_{f}) + 
\operatorname{ord}_{\gamma'} (\operatorname{jac}_{g}). 
\]
\end{lem}
\begin{proof}
The same proof as in \cite[Lemma 2.10]{NS} works due to Lemma \ref{lem:order_k[[t]]}. 
\end{proof}

\begin{rmk}
The same statement holds if $g$ is of the form (a). 
\end{rmk}

\begin{rmk}\label{rmk:NS2.6_k[[t]]}
We note in this remark that all propositions (Proposition 2.29, Lemmas 2.31, 2.32 and Proposition 2.33) in Subsection 2.6 in \cite{NS} 
are also true for a $k[t]$-morphism $f$ of the form (a) and (b) by making the following modifications: 
\begin{itemize}
\item Replacing the conditions $(\star)_n$ and $(\star \star)_n$ in \cite{NS} on $X$ with the following condition:
\begin{itemize}
\item ``Each irreducible component $X_i$ of $X$ has $\dim X_i \ge n+1$."
\end{itemize}

\item Replacing $\Omega_{\text{--}/k[t]}$ with $\widehat{\Omega} _{\text{--}/k[[t]]}$ and $\Omega' _{\text{--}/k[t]}$, and 
replacing $\operatorname{Jac} _{\text{--}/k[t]}$ with $\operatorname{Fitt}^n (\widehat{\Omega} _{\text{--}/k[[t]]})$ and 
$\operatorname{Fitt}^n (\Omega' _{\text{--}/k[t]})$. 
\end{itemize}

We note that 
\begin{itemize}
\item Proposition 2.29(2), Lemmas 2.31, 2.32, and Proposition 2.33 in \cite{NS}
\end{itemize}
are formal consequences of Proposition 2.29(1), Lemma 2.13(1) and Proposition 2.17 in \cite{NS}. 
The formal power series ring versions of Proposition 2.29(1), Lemma 2.13(1) and Proposition 2.17 in \cite{NS} are proved 
in Lemmas \ref{lem:EM4.4_R} and \ref{lem:EM4.4_k[[t]]},
Lemmas \ref{lem:order_R} and \ref{lem:order_k[[t]]}, and Propositions \ref{prop:EM4.1_Rt} and \ref{prop:EM4.1_k[[t]]} in this paper. 

Furthermore, Lemma 2.34 in \cite{NS} is also true in the formal power series ring setting by replacing 
$A = \operatorname{Spec} k[t][x_1, \ldots, x_N]$ with $\operatorname{Spec} k[x_1, \ldots, x_N][[t]]$. 
Indeed, the same proof of Lemma 2.34 in \cite{NS} works in this setting. 
\end{rmk}

\begin{prop}\label{prop:EM6.2_k[[t]]}
Let $n$ be a non-negative integer. 
Let $f: X \to Y$ be a morphism of affine $k[t]$-schemes of the form (b) above. 
Suppose that each irreducible component $W_i$ of $X$ and $Y$ has $\dim W_i \ge n + 1$. 
Let $e,e',e''\in \mathbb Z_{\ge 0}$. 
Let $A \subset X_{\infty}$ be a cylinder and let $B = f_{\infty}(A)$. 
Assume that 
\[
A  \subset \operatorname{Cont}^{e''} \bigl( \operatorname{Fitt}^n \bigl( \widehat{\Omega} _{X/k[[t]]} \bigr) \bigr) 
	\cap \operatorname{Cont}^{e}(\operatorname{jac}_f), \quad
B  \subset \operatorname{Cont}^{e'} \bigl( \operatorname{Fitt}^n \bigl( \Omega ' _{Y/k[t]} \bigr) \bigr). 
\]
Then, $B$ is a cylinder of $Y_{\infty}$ contained in $\operatorname{Cont}^{\ge 1}(\mathfrak{o}_{Y})$, 
where $\mathfrak{o}_{Y} \subset \mathcal{O}_Y$ is the ideal sheaf generated by $x_1, \ldots , x_L \in \mathcal{O}_Y$. 
Moreover, if $f_{\infty}|_{A}$ is injective, then it follows that
\[
\operatorname{codim}(A) + e = \operatorname{codim}(B).
\]
\end{prop}
\begin{proof}
By Remark \ref{rmk:NS2.6_k[[t]]}, the same proofs as in Subsection 2.6 in \cite{NS} work by making the following modifications:
\begin{itemize}
\item 
Replacing $\Omega_{X/k[t]}$ with $\widehat{\Omega} _{X/k[[t]]}$, and
$\Omega_{Y/k[t]}$ with $\Omega ' _{Y/k[t]}$. 
\item 
Replacing $\operatorname{Jac} _{X/k[t]}$ with $\operatorname{Fitt}^n \bigl( \widehat{\Omega} _{X/k[[t]]} \bigr)$, and 
$\operatorname{Jac} _{Y/k[t]}$ with $\operatorname{Fitt}^n \bigl( \Omega ' _{Y/k[t]} \bigr)$. 
\end{itemize}
\end{proof}

\begin{prop}\label{prop:DL2_k[[t]]}
Suppose that a finite group $G$ acts on the ring $k[x_1, \ldots , x_N]$ over $k$. 
Let $I \subset  k[x_1, \ldots , x_N]^G[[t]]$ be an ideal, 
and let $I' \subset k[x_1, \ldots , x_N][[t]]$ be the ideal generated by $I$. 
We denote
\[
X := \operatorname{Spec} \bigl( k[x_1, \ldots , x_N][[t]]/I' \bigr), \quad 
Y := \operatorname{Spec} \bigl( k[x_1, \ldots , x_N]^G[[t]]/I \bigr), 
\]
and denote $f: X \to X/G = Y$ the quotient morphism. 
Suppose that each irreducible component $W_i$ of $X$ and $Y$ has $\dim W_i \ge n + 1$. 
Let $A \subset X_{\infty}$ be a $G$-invariant cylinder and let $B = f_{\infty} (A)$. 
Let $e,e',e''\in \mathbb Z_{\ge 0}$. 
Assume that 
\[
A  \subset \operatorname{Cont}^{e''} \bigl( \operatorname{Fitt}^n \bigl( \widehat{\Omega} _{X/k[[t]]} \bigr) \bigr) 
	\cap \operatorname{Cont}^{e}(\operatorname{jac}_f), \quad
B  \subset \operatorname{Cont}^{e'} \bigl( \operatorname{Fitt}^n \bigl( \widehat{\Omega} _{Y/k[[t]]} \bigr) \bigr). 
\]
Then $B$ is a cylinder of $Y_{\infty}$ with 
\[
\operatorname{codim}(A)+e=\operatorname{codim}(B).
\]
\end{prop}
\begin{proof}
Note that $f$ is a morphism of the form (a). 
By Remark \ref{rmk:NS2.6_k[[t]]}, the same proof of Proposition 2.35 in \cite{NS} works. 
\end{proof}

\begin{rmk}\label{rmk:codim_formulae}
\begin{enumerate}
\item 
Proposition \ref{prop:EM6.2_k[[t]]} is true also for $k[t]$-morphisms of the form (a). 
Indeed, this is known for morphisms $f: X \to Y$ of formal $k[[t]]$-schemes (not necessarily affine). 
When $X$ is smooth over $k[[t]]$, this is proved in \cite[Lemme 7.1.3]{Seb04} (cf.\ \cite[Ch.5.\ Theorem 3.2.2]{CLNS}). 
The general case is proved in \cite[Lemmas 10.19, 10.20]{Yas}. 
We also note that Yasuda proves it in the more general setting, for formal Deligne-Mumford stacks of arbitrary characteristics.

\item 
Proposition \ref{prop:EM6.2_k[[t]]} is true also for a $k[t]$-morphism $f:X \to Y$ of 
$k[t][[x_1, \ldots , x_N]]$-schemes of finite type (not necessarily affine) by making the following modification: 
\begin{itemize}
\item 
Replacing the condition ``each irreducible component $W_i$ of $X$ and $Y$ has $\dim W_i \ge n + 1$" in Proposition \ref{prop:EM6.2_k[[t]]} with 
``each irreducible component $W_i$ of $X$ and $Y$ has $\dim ' W_i \ge n + 1$". 
\end{itemize}
\end{enumerate}
\end{rmk}

\section{Denef and Loeser's theory for quotient singularities}\label{section:DL}
In this section, we review the theory of arc spaces of quotient varieties established by Denef and Loeser \cite{DL02}
(cf.\ \cite{Yas16}, \cite[Section 3]{NS}).
We explain their theory in the formal power series ring setting. 

Let $d$ be a positive integer and $\xi \in k$ a primitive $d$-th root of unity. 
Let $G \subset \operatorname{GL}_N(k)$ be a finite subgroup with order $d$ that linearly acts on 
$\overline{A} := \widehat{\mathbb{A}}^N _k := \operatorname{Spec} k[[x_1,\ldots,x_N]]$. 
Let $\overline{X} \subset \overline{A}$ be a $G$-invariant subscheme. 
We denote by 
\[
A := \overline{A}/G, \qquad X := \overline{X}/G
\]
the quotient schemes. 
Let $Z \subset A$ be the minimum closed subset such that $\overline{A} \to A$ is \'{e}tale outside $Z$. 

Fix $\gamma \in G$. 
Since $G$ is a finite group, $\gamma$ can be diagonalized with some new basis $x_1^{(\gamma)}, \ldots, x_N^{(\gamma)}$. 
Let $\operatorname{diag}(\xi ^{e_1}, \ldots , \xi ^{e_N})$ be the diagonal matrix 
with $0 < e_i \le d$. 
Then we define a $k[t]$-ring homomorphism
\[
\lambda^* _{\gamma}: k[t][[x_1, \ldots, x_N]]^G \to k[x_1, \ldots, x_N]^{C_{\gamma}}[[t]]; 
\quad x_i^{(\gamma)} \mapsto t^{e_i/d}x_i^{(\gamma)}, 
\]
where $C_{\gamma}$ is the centralizer of $\gamma$ in $G$. 
Let $i: k[x_1, \ldots, x_N]^{C_{\gamma}}[[t]] \to k[x_1, \ldots, x_N][[t]]$ be the inclusion map, and 
let $\overline{\lambda}^*_{\gamma} = i \circ \lambda^* _{\gamma}$ be the composite map. 

Let $I_X \subset k[[x_1,\ldots,x_N]]^G$ be the defining ideal of $X$ in $A$. 
Let $I'_X \subset k[t][[x_1, \ldots, x_N]]^{G}$ be the ideal generated by $I_X$. 
Let 
\[
\widetilde{I}_X^{(\gamma)} \subset k[x_1, \ldots, x_N]^{C_{\gamma}}[[t]], \quad 
\overline{I}_X^{(\gamma)} \subset k[x_1, \ldots, x_N][[t]]
\]
be the ideals generated by $\lambda^* _{\gamma}(I'_X)$ and $\overline{\lambda}^* _{\gamma}(I'_X)$, 
respectively. 
Then we have the following diagram: 
\[
  \xymatrix{
k[t][[x_1, \ldots, x_N]]^{G} \ar[r] ^{\lambda _{\gamma} ^*} _{x_i^{(\gamma)} \mapsto t^{e_i/d} x_i^{(\gamma)}} \ar@/^20pt/[rr]^{\overline{\lambda} _{\gamma} ^*} \ar@{->>}[d] & k[x_1, \ldots, x_N]^{C_{\gamma}}[[t]] \ar@{->>}[d] \ar[r] & k[x_1, \ldots, x_N][[t]] \ar@{->>}[d] \\
k[t][[x_1, \ldots, x_N]]^{G}/I'_X \ar[r] ^-{\mu _{\gamma} ^{*}} \ar@/_20pt/[rr]_{\overline{\mu} _{\gamma} ^*} & k[x_1, \ldots, x_N]^{C_{\gamma}}[[t]] / \widetilde{I}_X ^{(\gamma)} \ar[r] & k[x_1, \ldots, x_N][[t]] / \overline{I}_X ^{(\gamma)}
  }
\]
We define $k[t]$-schemes $\overline{A}^{(\gamma)}$, $\widetilde{X}^{(\gamma)}$ and $\overline{X}^{(\gamma)}$ as follows: 
\begin{align*}
\overline{A}^{(\gamma)} &:= \operatorname{Spec} k[x_1, \ldots, x_N][[t]],  \\
\widetilde{X} ^{(\gamma)} &:= 
	\operatorname{Spec} \left( k[x_1, \ldots, x_N]^{C_{\gamma}}[[t]] / \widetilde{I}_X ^{(\gamma)} \right) , \\
\overline{X} ^{(\gamma)} &:= 
	\operatorname{Spec} \left( k[x_1, \ldots, x_N][[t]] / \overline{I}_X ^{(\gamma)} \right). 
\end{align*}
Let $\overline{A}^{(\gamma)} _{\infty}$, $\widetilde{X}_{\infty}^{(\gamma)}$ and $\overline{X}_{\infty}^{(\gamma)}$ be their arc spaces as $k[t]$-schemes 
defined in Section \ref{section:ktjet}. 
Then we have the following diagram of arc spaces: 
\[
\xymatrix{
A_{\infty} & \bigl( \overline{A}^{(\gamma)} / C_{\gamma} \bigr)_{\infty} \ar[l]^-{\lambda _{\gamma}} & \overline{A}^{(\gamma)} _{\infty} \ar[l] \ar@/_20pt/[ll]_{\overline{\lambda} _{\gamma}} \\
X_{\infty} \ar@{^{(}-{>}}[u] & \widetilde{X}_{\infty} ^{(\gamma)} \ar[l]^{\mu _{\gamma}} \ar@{^{(}-{>}}[u] & \overline{X}_{\infty} ^{(\gamma)} \ar[l] \ar@{^{(}-{>}}[u] \ar@/^20pt/[ll]^{\overline{\mu} _{\gamma}}
}
\]
We denote by $\mu_{\gamma}$ and $\overline{\mu}_{\gamma}$ 
the restrictions of $\lambda _{\gamma}$ and $\overline{\lambda}_{\gamma}$ to 
$\widetilde{X}_{\infty} ^{(\gamma)}$ and $\overline{X}_{\infty} ^{(\gamma)}$, respectively. 

\begin{rmk}\label{rmk:pb}
\begin{enumerate}
\item 
Here, we have used the fact that the arc spaces of 
\[
\operatorname{Spec} k[t][[x_1, \ldots, x_N]]^{G}, \quad 
\operatorname{Spec} \bigl( k[t][[x_1, \ldots, x_N]]^{G} /I'_X \bigr)
\]
as $k[t]$-schemes (defined in Section \ref{section:ktjet}) are isomorphic to the arc spaces of 
$A$ and $X$ as $k$-schemes (defined in Section \ref{section:jet}). 

\item 
Furthermore, the vertical arrows are closed immersions. Under these identifications, we have 
\[
\lambda _{\gamma} ^{-1} (X_{\infty}) = \widetilde{X}_{\infty} ^{(\gamma)}, \qquad 
\overline{\lambda} _{\gamma} ^{-1} (X_{\infty}) = \overline{X}_{\infty} ^{(\gamma)}. 
\]
\end{enumerate}
\end{rmk}

\begin{prop}[{\cite[Section 2]{DL02}, cf.\ \cite[Subsections 3.1, 3.2]{NS}}]\label{prop:DL_A}
The ring homomorphism $\lambda^* _{\gamma}$ induces the maps 
$\lambda_{\gamma}:(\overline{A}^{(\gamma)} /C_{\gamma})_{\infty} \to A_{\infty}$ and 
$\overline{\lambda}_{\gamma}: \overline{A}^{(\gamma)} _{\infty} \to A_{\infty}$, 
and the following hold. 
\begin{enumerate}
\item There is a natural inclusion $\overline{A}^{(\gamma)} _{\infty}/C_{\gamma} \hookrightarrow \bigl( \overline{A}^{(\gamma)} /C_{\gamma} \bigr)_{\infty}$. 

\item The composite map $\overline{A}^{(\gamma)} _{\infty}/C_{\gamma} \hookrightarrow \bigl( \overline{A}^{(\gamma)} /C_{\gamma} \bigr)_{\infty} 
\xrightarrow{\lambda _{\gamma}} A_{\infty}$ is injective outside $Z_{\infty}$. 

\item $\bigsqcup _{\langle \gamma \rangle \in \operatorname{Conj}(G)} 
\bigl( \overline{\lambda}_{\gamma} \bigl( \overline{A}^{(\gamma)} _{\infty} \bigr) \setminus Z_{\infty} \bigr)
= A_{\infty} \setminus Z_{\infty}$ holds, where $\operatorname{Conj}(G)$ denotes the set of the conjugacy classes of $G$. 
\end{enumerate}
\end{prop}
\begin{proof}
In \cite{DL02} and \cite{NS}, the assertions are proved for the polynomial ring $k[t][x_1, \ldots , x_N]$, 
and their proofs work in the formal power series ring setting. 
\end{proof}

\noindent
By Remark \ref{rmk:pb}(2), we can deduce the same statement for $X$. 

\begin{prop}[{cf.\ \cite[Subsection 3.3]{NS}}]\label{prop:DL_X}
The ring homomorphism $\lambda^* _{\gamma}$ induces the maps 
$\mu_{\gamma}:\widetilde{X}_{\infty} ^{(\gamma)} \to X_{\infty}$ and 
$\overline{\mu}_{\gamma}: \overline{X}^{(\gamma)}_{\infty} \to X_{\infty}$, 
and the following hold. 
\begin{enumerate}
\item There is a natural inclusion $\overline{X}_{\infty} ^{(\gamma)} / C_{\gamma} \hookrightarrow \widetilde{X}_{\infty} ^{(\gamma)}$. 

\item The composite map $\overline{X}_{\infty} ^{(\gamma)} / C_{\gamma} \hookrightarrow \widetilde{X}_{\infty} ^{(\gamma)} 
\xrightarrow{\mu _{\gamma}} X_{\infty}$ is injective outside $Z_{\infty}$. 

\item $\bigsqcup _{\langle \gamma \rangle \in \operatorname{Conj}(G)} 
\bigl( \overline{\mu}_{\gamma}(\overline{X}^{(\gamma)}_{\infty}) \setminus Z_{\infty} \bigr)
= X_{\infty} \setminus Z_{\infty}$ holds. 
\end{enumerate}
\end{prop}

\begin{rmk}\label{rmk:RF}
\begin{enumerate}
\item 
In \cite{NS}, $e_i$ is taken to satisfy $0 \le e_i \le d-1$. 
Note that the ring homomorphism $\lambda _{\gamma} ^*$ cannot be defined in this way of taking in our formal power series ring setting.

\item
It is also natural to define a $k[t]$-ring homomorphism
\[
\overline{\lambda}^{\prime *}_{\gamma}: k[t][[x_1, \ldots, x_N]]^G \to k[t][[x_1, \ldots, x_N]]; 
\quad x_i^{(\gamma)} \mapsto t^{e_i/d}x_i^{(\gamma)}, 
\]
and schemes 
\[
\overline{A}^{\prime (\gamma)} := \operatorname{Spec} k[t][[x_1, \ldots, x_N]], \quad 
\overline{X}^{\prime (\gamma)} := 
\operatorname{Spec} \left( k[t][[x_1, \ldots, x_N]] / \overline{I}_X ^{\prime (\gamma)} \right), 
\]
where $\overline{I}_X ^{\prime (\gamma)}$ is the ideal of $k[t][[x_1, \ldots, x_N]]$ generated by 
$\overline{\lambda}^{\prime *} _{\gamma}(I'_X)$. 
Then by the same argument as in this section, $\overline{\lambda}^{\prime *}_{\gamma}$ induces maps 
\[
\overline{\lambda}^{\prime}_{\gamma} : \overline{A}^{\prime (\gamma)} _{\infty} \to A_{\infty}, \qquad 
\overline{\mu}^{\prime}_{\gamma} : \overline{X}^{\prime (\gamma)}_{\infty} \to X_{\infty}. 
\]
However, as we can see in the discussion below, 
if $\overline{A}^{(\gamma)}$ and $\overline{X}^{(\gamma)}$ are replaced with 
$\overline{A}^{\prime (\gamma)}$ and $\overline{X}^{\prime (\gamma)}$, then 
Propositions \ref{prop:DL_A} and \ref{prop:DL_X} are no longer valid. 

First, we note that 
\begin{align*}
\overline{A}^{\prime (\gamma)} _m & \simeq \operatorname{Spec} \bigl( k \bigl[ \bigl[ x_1^{(0)}, \ldots , 
x_N^{(0)} \bigr] \bigr] \bigl[x_1^{(s)}, \ldots , x_N^{(s)} \ \big| \ 1 \le s \le m \bigr] \bigr), \\
\overline{A}^{(\gamma)} _m & \simeq \operatorname{Spec} \bigl( k \bigl[x_1^{(s)}, \ldots , x_N^{(s)} \ \big| \ 0 \le s \le m \bigr] \bigr), 
\end{align*}
and we have a natural morphism $\overline{A}^{\prime (\gamma)} _m \to \overline{A}^{(\gamma)} _m$ induced by the ring inclusion
\[
k \bigl[x_1^{(s)}, \ldots , x_N^{(s)} \ \big| \ 0 \le s \le m \bigr] \hookrightarrow
k \bigl[ \bigl[ x_1^{(0)}, \ldots , x_N^{(0)} \bigr] \bigr] \bigl[x_1^{(s)}, \ldots , x_N^{(s)} \ \big| \ 1 \le s \le m \bigr]. 
\] 
Since the morphisms $\overline{A}^{\prime (\gamma)} _m \to \overline{A}^{(\gamma)} _m$ are compatible with the truncation maps, 
they induce a map $\overline{A}^{\prime (\gamma)} _{\infty} \to \overline{A}^{(\gamma)} _{\infty}$. 
The map $\overline{X}^{\prime (\gamma)}_{\infty} \to \overline{X}^{(\gamma)}_{\infty}$ is also induced, 
and we have the following commutative diagrams: 
\[
\xymatrix{
A_{\infty} & \overline{A}^{(\gamma)} _{\infty} \ar[l]_{\overline{\lambda}_{\gamma}} & X_{\infty} & \overline{X}^{(\gamma)}_{\infty} \ar[l]_{\overline{\mu}_{\gamma}}  \\
&  \overline{A}^{\prime (\gamma)} _{\infty} \ar[ul]^{\overline{\lambda}^{\prime}_{\gamma}}\ar[u] & & \overline{X}^{\prime (\gamma)}_{\infty} \ar[ul]^{\overline{\mu}^{\prime}_{\gamma}} \ar[u]
}
\]
Furthermore, by contruction, $\overline{A}^{\prime (\gamma)} _{\infty} \to \overline{A}^{(\gamma)} _{\infty}$ induces isomorphisms
\begin{align*}
\overline{A}^{\prime (\gamma)} _{\infty} \cap \operatorname{Cont}^{\ge 1} \big( (x_1, \ldots , x_N) \big) & \simeq 
\overline{A}^{(\gamma)} _{\infty} \cap \operatorname{Cont}^{\ge 1} \big( (x_1, \ldots , x_N) \big), \\
\overline{X}^{\prime (\gamma)}_{\infty} \cap \operatorname{Cont}^{\ge 1} \big( (x_1, \ldots , x_N) \big) & \simeq 
\overline{X}^{(\gamma)}_{\infty} \cap \operatorname{Cont}^{\ge 1} \big( (x_1, \ldots , x_N) \big). 
\end{align*}
We also note that $\overline{A}^{\prime (\gamma)} _{\infty} \to \overline{A}^{(\gamma)} _{\infty}$ induces injective maps 
\[
\overline{A}^{\prime (\gamma)}_{\infty}(k) \hookrightarrow \overline{A}^{(\gamma)} _{\infty}(k), \quad
\overline{X}^{\prime (\gamma)}_{\infty}(k) \hookrightarrow \overline{X}^{(\gamma)}_{\infty}(k)
\]
on $k$-points. However, these two maps are not surjective in general (see (3) below).

\item Suppose that $N = 2$ and $G = \langle \gamma \rangle$, where $\gamma: k[[x_1, x_2]] \to k[[x_1, x_2]]$ is 
the involution defined by $\gamma (x_i) = - x_i$ for $i \in \{ 1, 2 \}$. 
Then we have $d = 2$, $e_1 = e_2 = 1$, and $k[[x_1, x_2]] ^G = k[[x_1 ^2 , x_1 x_2, x_2 ^2]]$. 
We denote by $\alpha \in A_{\infty}$ the $k$-arc corresponding to the $k[t]$-ring homomorphism 
\[
\alpha ^* : k[t][[x_1 ^2, x_1x_2, x_2 ^2]] \to k[[t]]; \quad x_1 ^2 \mapsto t, \ x_1 x_2 \mapsto t, \ x_2 ^2 \mapsto t. 
\]
Then, $\alpha$ is contained in the image of $\overline{\lambda}_{\gamma} : \overline{A}^{(\gamma)} _{\infty} \to A_{\infty}$. 
Indeed, if $\beta \in \overline{A}^{(\gamma)} _{\infty}$ is the $k$-arc defined by 
\[
\beta ^* : k[x_1, x_2][[t]] \to k[[t]]; \quad x_1 \mapsto 1, \ x_2 \mapsto 1, 
\]
then we have $\alpha ^* = \beta ^* \circ \overline{\lambda}^* _{\gamma}$ and hence $\alpha = \overline{\lambda}_{\gamma}(\beta)$. 

On the other hand, $\alpha$ is not contained in the image of $\overline{\lambda}'_{\gamma} : \overline{A}^{\prime (\gamma)}_{\infty} \to A_{\infty}$ because 
there is no $k[t]$-ring homomorphism $\beta ^* : k[t][[x_1, x_2]] \to k[[t]]$ satisfying $\beta ^* (x_1) = \beta ^* (x_2) = 1$. 
\end{enumerate}
\end{rmk}

\section{Arc spaces of hyperquotient singularities}\label{section:mld_R}
In this section, we prove in Theorem \ref{thm:mld_hyperquot_R} that \cite[Theorem 4.8]{NS} is still valid 
in the formal power series ring setting. 

\subsection{Minimal log discrepancies of hyperquotient singularities}

Let $d$ be a positive integer and let $\xi \in k$ be a primitive $d$-th root of unity. 
Let $G \subset \operatorname{GL}_N( k)$ be a finite group with order $d$ that linearly acts on 
$\overline{A} := \widehat{\mathbb{A}}^N _k = \operatorname{Spec} k[[x_1,\ldots,x_N]]$. 
We denote by 
\[
A := \overline{A}/G
\]
the quotient scheme. Let $Z \subset A$ be the minimum closed subset such that $\overline{A} \to A$ is \'{e}tale outside $Z$. 
We assume that $\operatorname{codim} Z \ge 2$, and hence the quotient morphism $\overline{A} \to A$ is \'{e}tale in codimension one. 
We note that $A$ is $\mathbb{Q}$-Gorenstein (cf.\ Remark \ref{rmk:compareK}). 
We fix a positive integer $r$ such that $\omega_{A/k} ^{\prime [r]}$ is invertible.

We fix $\gamma \in G$. Let $C_\gamma$ be the centralizer of $\gamma$ in $G$. 
Since $G$ is a finite group, $\gamma$ can be diagonalized with a suitable basis $x_1,\ldots,x_N$. 
Let $\operatorname{diag}(\xi ^{e_1}, \ldots , \xi ^{e_N})$ be the diagonal matrix with $0 < e_i \le d$. 
We define a $k[t]$-ring homomorphism
\[
\lambda_{\gamma} ^*: k[t][[x_1, \ldots, x_N]]^{G}  \to  k[x_1, \ldots, x_N]^{C_{\gamma}}[[t]];\quad 
x_i \mapsto t^{\frac{e_i}{d}}x_i, 
\]
and define $\overline{\lambda}_{\gamma} ^*: k[t][[x_1, \ldots, x_N]]^{G}  \to  k[x_1, \ldots, x_N][[t]]$ 
as the composition of $\lambda_{\gamma} ^*$ and the inclusion 
$i: k[x_1, \ldots, x_N]^{C_{\gamma}}[[t]] \to k[x_1, \ldots, x_N][[t]]$. 

Let $f_1, \ldots, f_c \in k[[x_1, \ldots, x_N]]^{G}$ be a regular sequence. 
We set 
\begin{align*}
B &:= \operatorname{Spec} \bigl( k[[x_1, \ldots, x_N]]^{G}/(f_1, \ldots , f_c) \bigr), \\ 
\overline{B} &:= \operatorname{Spec} \bigl( k[[x_1, \ldots, x_N]]/(f_1, \ldots , f_c) \bigr). 
\end{align*}
We denote by $n := N-c$ their dimensions. 

Suppose that $B$ is normal. 
Then it follows that $\overline{B} \to B$ is also \'{e}tale in codimension one, 
and $\overline{B}$ is also normal. 
Indeed, since $\operatorname{codim}_A Z \ge 2$, we have $A_{\rm sing} = Z$ by the purity of the branch locus (cf.\ \cite{Nag59}). 
Therefore we have $B \cap Z \subset B_{\rm sing}$ since $f_1, \ldots, f_{c}$ is a regular sequence (cf.\ \cite[tag 00NU]{Sta}). 
Then it follows that $\operatorname{codim}_B (Z \cap B) \ge 
\operatorname{codim}_B (B_{\rm sing}) \ge 2$ by the normality of $B$. 
Therefore, $\overline{B} \to B$ is also \'{e}tale in codimension one, and hence $\overline{B}$ is also normal 
by Serre's criterion of normality. 

Note that $\omega _{B/k} ^{\prime [r]}$ is invertible. 
Indeed, we have the adjunction formula 
\[
\omega' _{B/k} \simeq \operatorname{det} ^{-1} (I/I^{2}) \otimes _{\mathcal{O}_A} \omega' _{A/k}, 
\]
where $I := (f_1, \ldots , f_c) \subset k[[x_1, \ldots, x_N]]^G$, since the sequence
\[
0 \to I/I^{2} \to \Omega' _{A/k} \otimes _{\mathcal{O}_A} \mathcal{O}_B \to \Omega' _{B/k} \to 0 
\]
is exact at a regular point of $B$ by Proposition \ref{prop:spOmega_exact}. 

We define ideals $I'$, $\widetilde{I}^{(\gamma)}$ and $\overline{I}^{(\gamma)}$ by 
\begin{align*}
I' &:= (f_1, \ldots, f_c) \subset k[t][[x_1, \ldots, x_N]]^{G}, \\
\widetilde{I}^{(\gamma)} &:= \bigl( \lambda^* _{\gamma} (f_1),\ldots,\lambda^* _{\gamma} (f_c) \bigr) \subset k[x_1, \ldots, x_N]^{C_{\gamma}}[[t]], \\
\overline{I}^{(\gamma)} &:= \bigl( \overline{\lambda}^* _{\gamma} (f_1),\ldots,\overline{\lambda}^* _{\gamma} (f_c) \bigr) \subset k[x_1, \ldots, x_N][[t]]. 
\end{align*}
Then we have the following diagram. 
\[
  \xymatrix{
 k[t][[x_1, \ldots, x_N]]^{G} \ar[r]_-{\lambda _{\gamma} ^*} \ar@/^20pt/[rr]^{\overline{\lambda} _{\gamma} ^*} \ar@{->>}[d] &  k[x_1, \ldots, x_N]^{C_{\gamma}}[[t]] \ar@{->>}[d] \ar[r]_-i &  k[x_1, \ldots, x_N][[t]] \ar@{->>}[d] \\
 k[t][[x_1, \ldots, x_N]]^{G}/I' \ar[r]^-{\mu _{\gamma} ^*} \ar@/_20pt/[rr]_{\overline{\mu} _{\gamma}^* }&  k[x_1, \ldots, x_N]^{C_{\gamma}}[[t]] /\widetilde{I}^{(\gamma)} \ar[r]^-j &  k[x_1, \ldots, x_N][[t]] /\overline{I}^{(\gamma)}
  }
\]
We define $k[t]$-schemes $A'$, $\widetilde{A}^{(\gamma)}$, $\overline{A}^{(\gamma)}$, 
$B'$, $\widetilde{B}^{(\gamma)}$ and $\overline{B}^{(\gamma)}$ as follows: 
\begin{alignat*}{2}
A' &:= \operatorname{Spec} k[t][[x_1, \ldots , x_N]]^G, &
	\quad B' &:= \operatorname{Spec}  \bigl( k[t][[x_1,\ldots,x_N]]^G/I' \bigr),\\
\widetilde{A}^{(\gamma)} &:= \operatorname{Spec} k[x_1, \ldots , x_N]^{C_{\gamma}}[[t]], &
	\quad \widetilde{B}^{(\gamma)} &:= \operatorname{Spec} \bigl( k[x_1,\ldots,x_N]^{C_{\gamma}}[[t]]/\widetilde{I}^{(\gamma)} \bigr), \\
\overline{A}^{(\gamma)} &:= \operatorname{Spec} k[x_1, \ldots , x_N][[t]], \quad &
	\overline{B}^{(\gamma)} &:= \operatorname{Spec} \bigl( k[x_1,\ldots,x_N][[t]]/\overline{I}^{(\gamma)} \bigr). 
\end{alignat*}
Then we have the following morphisms 
between the corresponding $k[t]$-schemes. 
\[
  \xymatrix{
A' & \widetilde{A}^{(\gamma)} \ar[l]^{\lambda _{\gamma}} & \overline{A}^{(\gamma)} \ar[l]^-q   \ar@/_18pt/[ll]_{\overline{\lambda} _{\gamma}} \\
B' \ar@{^{(}->}[u]^{\sigma} & \widetilde{B}^{(\gamma)} \ar@{^{(}->}[u] \ar[l]_{\mu _{\gamma}} &  \overline{B}^{(\gamma)} \ar[l]_-p \ar@/^18pt/[ll]^{\overline{\mu} _{\gamma}} \ar@{^{(}->}[u]_{\tau}
  }
\]

\begin{rmk}
\begin{enumerate}
\item 
Note that 
\begin{itemize}
\item $A'$ and $B'$ are affine schemes of the form $\operatorname{Spec} \bigl( k[t][[x_1, \ldots , x_M]]/J \bigr)$, and 
\item $\widetilde{A}^{(\gamma)}$, $\overline{A}^{(\gamma)}$, $\widetilde{B}^{(\gamma)}$ and $\overline{B}^{(\gamma)}$ 
are of the form $\operatorname{Spec} \bigl( k[x_1, \ldots , x_M][[t]]/J \bigr)$. 
\end{itemize}
We will use the notion of the sheaf $\Omega '_{X/k[t]}$ of special differentials for 
$X = A', B'$ defined in Section \ref{section:Omega'}, and use the notion of the sheaf $\widehat{\Omega}_{X/k[[t]]}$ for 
$X = \widetilde{A}^{(\gamma)}, \overline{A}^{(\gamma)}, \widetilde{B}^{(\gamma)}, \overline{B}^{(\gamma)}$ 
defined in Subsection \ref{subsection:formal}. 

Since $I'$, $\widetilde{I}^{(\gamma)}$ and $\overline{I}^{(\gamma)}$ are generated by $c$ elements, 
each irreducible component $W_i$ of $B'$, $\widetilde{B}^{(\gamma)}$ and $\overline{B}^{(\gamma)}$ has $\dim W_i \ge n+1$. 
Therefore, we can apply lemmas and propositions in Subsections \ref{subsection:ktjet} and \ref{subsection:formal} to their arc spaces. 

\item 
Lemma \ref{lem:isotrivial2}(3)(4) below show that $\overline{B}^{(\gamma)}$ has only one irreducible component that is flat over $k[[t]]$. 
Furthermore, the component $V$ has $\dim V = n+1$. 
\end{enumerate}
\end{rmk}

\begin{lem}\label{lem:ff}
Let $f: C_1 \to C_2$ be a  flat ring homomorphism of Noetherian rings. 
Then the following hold. 
\begin{enumerate}
\item 
Let $P$ be a prime ideal of $C_1$ such that $PC_2 \not = C_2$. 
Then $\operatorname{ht} P = \operatorname{ht} (PC_2)$ holds. 

\item 
Suppose that $f$ is faithfully flat. 
If $I \subsetneq C_1$ is a proper ideal of $C_1$, then $\operatorname{ht} I = \operatorname{ht} (IC_2)$ holds. 
\end{enumerate}
\end{lem}
\begin{proof}
We shall prove (1). 
Let $Q$ be a minimal prime of $PC_2$. 
Then by the going-down theorem, we have $P = Q \cap C_1$. 
Therefore, we have $\operatorname{ht} Q = \operatorname{ht} P$ by \cite[Theorem 15.1]{Mat89}, which proves (1). 

We shall prove (2). Note that $IC_2 \not = C_2$ holds by the faithfully flatness. 
First, the inequality $\operatorname{ht} I \le \operatorname{ht} (IC_2)$ follows from the going-down theorem. 
We shall prove the opposite inequality. 
Take a minimal prime $P$ of $I$ such that $\operatorname{ht} I = \operatorname{ht} P$. 
Then by (1), it follows that 
\[
\operatorname{ht} I = \operatorname{ht} P = \operatorname{ht} (PC_2) \ge \operatorname{ht} (IC_2), 
\]
which completes the proof. 
\end{proof}

\begin{lem}\label{lem:isotrivial2}
We denote $F_i := \overline{\lambda}^* _{\gamma} (f_i)$ for each $1 \le i \le c$. 
Consider the following diagram of rings. 
\[
\xymatrix{
S_1 := k[x_1, \ldots , x_N][[t]] \ar@{^{(}->}[d]_{h_1} \ar@{->>}[r] & C_1 := k[x_1, \ldots , x_N][[t]]/(F_1, \ldots , F_c) \ar[d] _{g_1} \\ 
S_2 := k[x_1, \ldots , x_N]((t)) \ar@{^{(}->}[d]_{h_2} \ar@{->>}[r] & C_2 := k[x_1, \ldots , x_N]((t))/(F_1, \ldots , F_c) \ar[d]_{g_2}  \\
S_3 := k[x_1, \ldots , x_N]((t^{1/d})) \ar@{->>}[r] & C_3 := k[x_1, \ldots , x_N]((t^{1/d}))/(F_1, \ldots , F_c)  \\
S_4 := k[[x_1, \ldots , x_N]] \ar@{^{(}->}[u]^{h_3}_{x_i \mapsto t^{e_i/d}x_i}  \ar@{->>}[r] & C_4 := k[[x_1, \ldots , x_N]]/(f_1, \ldots , f_c) \ar[u]^{g_3}_{x_i \mapsto t^{e_i/d}x_i} \\ 
}
\]
We denote $I_i := (F_1, \ldots, F_c) \subset S_i$ for $i \in \{1,2,3 \}$, and 
$I_4 := (f_1, \ldots, f_c) \subset S_4$. 
Then the following hold. 
\begin{enumerate}
\item $h_1$, $h_2$ and $h_3$ are regular, and hence so are $g_1$, $g_2$ and $g_3$. 
\item $h_2$ and $h_3$ are faithfully flat, and hence so are $g_2$ and $g_3$. 
\item $C_2$ and $C_3$ are normal domains. In particular, $I_2$ and $I_3$ are prime ideals. 
\item $\operatorname{ht}(I_2) = \operatorname{ht}(I_3) = c$. 
\end{enumerate}
\end{lem}

\begin{proof}
We shall prove (1) and (2). 
Since $h_1$ is the localization by $t \in S_1$, it is regular. 
Note that the inclusion map
\[
h'_2: S'_2 := k[x_1, \ldots, x_N][t, t^{-1}] \to S'_3 := k[x_1, \ldots, x_N][t^{1/d}, t^{-1/d}]
\] 
is \'{e}tale and faithfully flat. 
Since $h_2$ is the base change ${\rm -} \otimes _{S'_2} S_2$ of $h'_2$, 
it follows that $h_2$ is \'{e}tale (in particular, regular) and faithfully flat. 

Let $P$ be a prime ideal of $S_4$ and let $Q$ be a prime ideal of $S_3$ such that $Q \cap S_4 = P$. 
To see that $h_3$ is regular, it is sufficient to show the following two conditions: 
\begin{itemize}
\item $\operatorname{ht} (P(S_3)_Q) = \operatorname{ht} P$. 
\item $S_3/PS_3$ is regular at $Q$. 
\end{itemize}
Note here that the first condition is equivalent to the flatness by \cite[Theorem 23.1]{Mat89} 
since $S_4$ is regular and $S_3$ is Cohen-Macaulay. 
We also note that the inequality $\operatorname{ht} (P(S_3)_Q) \le \operatorname{ht} P$ 
always holds by \cite[Theorem 15.1]{Mat89}. 

Let $r := \operatorname{ht} P$. 
Since $S_4/P$ is regular at $P$, by the Jacobian criterion of regularity, 
there exist $D_1, \ldots , D_r \in \operatorname{Der}_k(S_4)$ and $a_1, \ldots, a_r \in P$ such that 
$c := \det (D_i(a_j))_{ij} \not \in P$. 
Since $\operatorname{Der}_k(S_4)$ is generated by $\partial/\partial x_i$'s, 
we may assume that $D_i = \partial/\partial x_i$ holds for $1 \le i \le r$ by changing the indices of $x_i$. 
We note that $\partial/\partial x_i$ defines a derivation on $S_3$, and we denote it by $D_i' \in \operatorname{Der}_k(S_3)$. 
Then we have 
\[
D'_i (h_3(a_j)) = t^{\frac{e_i}{d}} h_3 (D_i (a_j)), 
\]
and hence we have
\[
\det \bigl( D' _i (h_3(a_j)) \bigr)_{ij} = t^{\frac{e'}{d}} h_3(c)
\]
where $e' := \sum _{i=1}^r e_i$. 
Here, we have $t^{\frac{e'}{d}} h_3(c) \not \in Q$ since $P = Q \cap S_4$ and $c \not \in P$. 
Therefore, it follows that $\operatorname{ht} (P(S_3)_Q) \ge r$ by \cite[Theorem 30.4(1)]{Mat89}. 
Hence, we have $\operatorname{ht} (P(S_3)_Q) = r$ and 
$S_3/PS_3$ is regular at $Q$ by \cite[Theorem 30.4(2)]{Mat89}. 
We have proved the regularity of $h_3$. 

To see the faithfully flatness of $h_3$, it is sufficient to show $\mathfrak{m}S_3 \not = S_3$ 
for the maximal ideal $\mathfrak{m} = (x_1, \ldots, x_N)$ of $S_4$. 
It is clear because $\mathfrak{m}S_3 \subset (x_1, \ldots, x_N) \not = S_3$.

We shall prove (3). 
Note that $C_4$ is normal by our assumption. 
Since the normality is preserved under faithfully flat regular ring homomorphisms (cf.\ \cite[Theorem 32.2]{Mat89}), 
$C_2$ and $C_3$ are normal. Therefore, it is sufficient to show that $C_2$ and $C_3$ are domains. 
In what follows, we shall only prove that $C_2$ is a domain since the same proof works for $C_3$. 
Suppose the contrary that $I_2$ has minimal primes $P_1$ and $P_2$ with $P_1 \not = P_2$. 
Then by the normality of $C_2$, we have $P_1 + P_2 = S_2$ (cf.\ \cite[Proposition 2.20]{Eis95}). 
Set $Q_1 := P_1 \cap S_1$ and $Q_2 := P_2 \cap S_1$. 
Since $P_1 + P_2 = S_2$, it follows that $t^{s} \in Q_1 + Q_2$ for some $s \ge 0$. 
To get a contradiction, we shall prove 
\begin{itemize}
\item $Q_1, Q_2 \subset (x_1, \ldots , x_N, t^{s+1})$. 
\end{itemize}
Suppose the contrary that 
\begin{itemize}
\item there exist $g \in k[[t]]$ and $h \in (x_1, \ldots , x_N)$ such that $g+h \in Q_1$ and $g \not \in (t^{s+1})$. 
\end{itemize}
Let $0 \le a \le s$ be the minimum $a$ such that $g \not \in (t^{a+1})$. 
Then $g - b t^a \in (t^{a+1})$ holds for some $b \in k^{\times}$. 
We may assume $b = 1$ by replacing $g$ and $h$ with $b^{-1}g$ and $b^{-1}h$. 
For $c \in k^{\times}$, we denote by $T_c: S_1 \to S_1$ the ring isomorphism
\[
T_c: k[x_1, \ldots, x_N][[t]] \to k[x_1, \ldots, x_N][[t]]; \quad t \mapsto c^{-d}t,\ x_i \mapsto c^{e_i}x_i. 
\]
Since $I_1$ is $T_c$-invariant, so is its minimal prime $Q_1$. 
Since the ideal $(t^{s+1})$ of $S_1$ is also $T_c$-invariant, 
$T_c$ induces the ring isomorphism $T_c' : S_1/(t^{s+1}) \to S_1/(t^{s+1})$. 
Hence, $S_1/(t^{s+1})$ has a graded ring structure satisfying $\deg t = -d$ and $\deg x_i = e_i$. 
Then $(Q_1 + (t^{s+1})) / (t^{s+1})$ is a homogeneous ideal. 
Therefore, the term $f_{- da}$ of $f:=g+h$ with degree $-da$ is contained in $Q_1 + (t^{s+1})$. 
Since $f_{-da} - t^a \in (x_1, \ldots , x_N)$, we have $f_{-da} - t^a \in (t^{a+1})$ by looking at the degrees of its terms. 
Therefore $f_{-da} = t^a (1 + f')$ holds for some $f' \in (t)$. 
Since $1 + f' \in S_1^{\times}$, we have $t^a \in Q_1 + (t^{s+1})$, and hence $t^a \in Q_1$. 
Therefore, we have $P_1 = S_2$ and we get a contradiction. 

Note that $\operatorname{ht}(I_4) = c$ by our assumption. 
Then, (4) follows from (2) and Lemma \ref{lem:ff}. 
\end{proof}

\begin{lem}\label{lem:isotrivial3}
Let $S_i$ and $h_i$ be as in Lemma \ref{lem:isotrivial2}. 
Let $\mathfrak{c} \subset S_4^G$ be an ideal of $S_4^G$. 
Let $\mathfrak{c}_1$ be the ideal of $S_1$ generated by the image of $\mathfrak{c}$ 
by the ring homomorphism $S_4^G \to S_1; \ x_i \mapsto t^{e_i/d} x_i$. 
We denote $\mathfrak{c}_2 := \mathfrak{c}_1 S_2$. 
Then we have $\operatorname{ht}(\mathfrak{c}_2) = \operatorname{ht} \mathfrak{c}$. 
\end{lem}
\begin{proof}
We define $\mathfrak{c}_3 := \mathfrak{c}_2 S_3$ and $\mathfrak{c}_4 := \mathfrak{c} S_4$.
Then we have $\mathfrak{c}_3 = \mathfrak{c}_4 S_3$. 
Since $S_4 ^G \hookrightarrow S_4$ is an integral extension, 
we have $\operatorname{ht} \mathfrak{c} = \operatorname{ht}(\mathfrak{c} S_4) = \operatorname{ht}(\mathfrak{c}_4)$. 
Furthermore, we have
\[
\operatorname{ht}(\mathfrak{c}_3) = \operatorname{ht}(\mathfrak{c}_4 S_3) = \operatorname{ht}(\mathfrak{c}_4), \quad 
\operatorname{ht}(\mathfrak{c}_3) = \operatorname{ht}(\mathfrak{c}_2 S_3) = \operatorname{ht}(\mathfrak{c}_2)
\]
by Lemma \ref{lem:isotrivial2}(2) and Lemma \ref{lem:ff}. 
We complete the proof. 
\end{proof}

\begin{defi}
\begin{enumerate}
\item 
We shall define sheaves $\omega' _{A'/k[t]}$ on $A'$ and $\omega' _{B'/k[t]}$ on $B'$ using 
the special canonical sheaves $\omega' _{A/k}$ and $\omega' _{B/k}$ defined in Section \ref{section:Omega'}. 
Let $A' \to A$ and $B' \to B$ be the morphisms induced by the natural ring inclusions
\begin{align*}
k[[x_1, \ldots, x_N]]^G   &\hookrightarrow k[t][[x_1, \ldots, x_N]]^G, \\
k[[x_1, \ldots, x_N]]^G/I &\hookrightarrow k[t][[x_1, \ldots, x_N]]^G/I'. 
\end{align*}
Then, we define 
\[
\omega' _{A'/k[t]} := \omega '_{A/k} \otimes _{\mathcal{O}_A} \mathcal{O}_{A'}, \quad 
\omega' _{B'/k[t]} := \omega '_{B/k} \otimes _{\mathcal{O}_B} \mathcal{O}_{B'}. 
\]
Note that they satisfy 
\[
\omega' _{B'/k[t]} \simeq \operatorname{det} ^{-1} (I'/I^{\prime 2}) \otimes _{\mathcal{O}_{B'}} \sigma^* \omega' _{A'/k[t]}. 
\]
Here, $\operatorname{det} ^{-1} (I'/I^{\prime 2}) := \bigl( \bigwedge ^{c} (I'/I^{\prime 2}) \bigr)^*$ is an invertible sheaf on $B'$.

\item 
The canonical map 
\[
\Omega _{B'/k[t]} ^{\prime n} \to \omega' _{B'/k[t]}
\] 
is induced by $\Omega _{B/k} ^{\prime n} \to \omega' _{B/k}$ and 
the isomorphism $\Omega' _{B'/k[t]} \simeq \Omega' _{B/k} \otimes _{\mathcal{O}_B} \mathcal{O}_{B'}$. 
The canonical map $\Omega _{A'/k[t]} ^{\prime n} \to \omega' _{A'/k[t]}$ is also defined. 
We define an ideal sheaf $\mathfrak{n}_{r,B'} \subset \mathcal{O}_{B'}$ by 
\[
\operatorname{Im} \Bigl( \bigl( \Omega _{B'/k[t]} ^{\prime n} \bigr)^{\otimes r} \to \omega^{\prime [r]} _{B'/k[t]} \Bigr) 
	= \mathfrak{n}_{r,B'} \otimes _{\mathcal{O}_{B'}} \omega^{\prime [r]} _{B'/k[t]}. 
\]
Then it satisfies $\mathfrak{n}_{r,B'} = \mathfrak{n}_{r,B} \mathcal{O}_{B'}$. 

\item 
We define ideal sheaves
$\operatorname{Jac}'_{B'/k[t]}$, 
$\operatorname{Jac}'_{\widetilde{B}^{(\gamma)}/k[[t]]}$ and 
$\operatorname{Jac}'_{\overline{B}^{(\gamma)}/k[[t]]}$ by
\begin{align*}
\operatorname{Jac}'_{B'/k[t]} &:= \operatorname{Fitt}^n \bigl( \Omega' _{B'/k[t]} \bigr) \subset \mathcal{O}_{B'}, \\
\operatorname{Jac}'_{\widetilde{B}^{(\gamma)}/k[[t]]} &:= \operatorname{Fitt}^n \bigl( \widehat{\Omega} _{\widetilde{B}^{(\gamma)}/k[[t]]} \bigr) \subset \mathcal{O}_{\widetilde{B}^{(\gamma)}}, \\
\operatorname{Jac}'_{\overline{B}^{(\gamma)}/k[[t]]} &:= \operatorname{Fitt}^n \bigl( \widehat{\Omega} _{\overline{B}^{(\gamma)}/k[[t]]} \bigr) \subset \mathcal{O}_{\overline{B}^{(\gamma)}}. 
\end{align*}
Here, we note that $\widetilde{B}^{(\gamma)}$ and $\overline{B}^{(\gamma)}$ are not necessarily equidimensional. 

\item 
We define an invertible sheaf $L_{\overline{B}^{(\gamma)}}$ on $\overline{B}^{(\gamma)}$ by
\[
L_{\overline{B}^{(\gamma)}} := \overline{\mu} _{\gamma}^* 
\bigl( \operatorname{det} ^{-1} (I'/I^{\prime 2}) \bigr) \otimes _{\mathcal{O}_{\overline{B}^{(\gamma)}}} \tau ^* \widehat{\Omega}^{N} _{\overline{A}^{(\gamma)}/k[[t]]}. 
\]
Then there exist canonical homomorphisms
\[
\widehat{\Omega}^{n} _{\overline{B}^{(\gamma)}/k[[t]]} \to L _{\overline{B}^{(\gamma)}}, \qquad 
\overline{\mu} _{\gamma}^* \omega^{\prime [r]} _{B'/k[t]} \to L^{[r]} _{\overline{B}^{(\gamma)}}
\]
such that the following diagram commutes (cf.\ \cite[Lemma 4.5(2)]{NS}). 
\[
  \xymatrix{
 \overline{\mu} _{\gamma}^* ( \Omega^{\prime n} _{B'/k[t]}) ^{\otimes r} \ar[r] \ar[d] & \bigl( \widehat{\Omega}^{n} _{\overline{B}^{(\gamma)}/k[[t]]} \bigr)^{\otimes r} \ar[d] \\
 \overline{\mu} _{\gamma}^* \omega^{\prime [r]} _{B'/k[t]}  \ar[r] & L^{[r]} _{\overline{B}^{(\gamma)}}
  }
\]
Furthermore, by the same argument as in \cite[Lemma 4.5(1)]{NS}, we have 
\[
\operatorname{Im}\bigl( \widehat{\Omega}^{n} _{\overline{B}^{(\gamma)}/k[[t]]} \to L _{\overline{B}^{(\gamma)}} \bigr)
= \operatorname{Jac}'_{\overline{B}^{(\gamma)}/k[[t]]} 
\otimes _{\mathcal{O}_{\overline{B}^{(\gamma)}}} L_{\overline{B}^{(\gamma)}}. 
\]

\item 
We define ideal sheaves $\mathfrak{n}'_{1,p}$ and $\mathfrak{n}'_{1,\overline{\mu} _{\gamma}}$ 
on $\overline{B}^{(\gamma)}$ by
\begin{align*}
\operatorname{Im} \bigl( p^* \widehat{\Omega}^n _{\widetilde{B}^{(\gamma)}/k[[t]]} \to L_{\overline{B}^{(\gamma)}} \bigr) 
	&= \mathfrak{n}'_{1,p} \otimes_{\mathcal{O}_{\overline{B}^{(\gamma)}}} L_{\overline{B}^{(\gamma)}}, \\
\operatorname{Im} \bigl( \overline{\mu} _{\gamma}^* \Omega^{\prime n} _{B'/k[t]} \to L_{\overline{B}^{(\gamma)}} \bigr) 
	&= \mathfrak{n}'_{1,\overline{\mu} _{\gamma}} \otimes _{\mathcal{O}_{\overline{B}^{(\gamma)}}} L_{\overline{B}^{(\gamma)}}. 
\end{align*}

\item 
We define $\operatorname{age}'(\gamma) := \sum _{i=1}^N \frac{e_i}{d}$. 
Note that we took $e_i$ to satisfy $0 < e_i \le d$. 
The age of $\gamma$ is usually defined by 
$\operatorname{age}(\gamma) = \operatorname{age}'(\gamma) - \# \{ 1 \le i \le N \mid e_i = d \}$. 
\end{enumerate}
\end{defi}

\begin{lem}\label{lem:L2_R}
Let $\alpha \in \overline{B}^{(\gamma)} _{\infty}$ be a $k$-arc with 
$\operatorname{ord}_{\alpha} \bigl( \operatorname{Jac}'_{\overline{B}^{(\gamma)}/k[[t]]} \bigr) < \infty$. 
Then the following hold. 
\begin{enumerate}
\item $\operatorname{ord}_{\alpha} (\operatorname{jac}_p) 
	+ \operatorname{ord}_{\alpha} \bigl( \operatorname{Jac}'_{\overline{B}^{(\gamma)}/k[[t]]} \bigr)
		= \operatorname{ord}_{\alpha} (\mathfrak{n}'_{1,p})$. 
\item $\operatorname{ord}_{\alpha} (\operatorname{jac}_{\overline{\mu} _{\gamma}}) 
	+ \operatorname{ord}_{\alpha} \bigl( \operatorname{Jac}'_{\overline{B}^{(\gamma)}/k[[t]]} \bigr)
		= \operatorname{ord}_{\alpha} (\mathfrak{n}'_{1,\overline{\mu} _{\gamma}})$. 
\end{enumerate}
\end{lem}
\begin{proof}
The same proof as in \cite[Lemma 4.6]{NS} works due to Lemma \ref{lem:order_k[[t]]}. 
\end{proof}

\begin{lem}\label{lem:age2_R}
Let $\alpha \in \overline{B}^{(\gamma)}_{\infty}$ be a $k$-arc. 
Set $\alpha ' := \overline{\mu} _{\gamma \infty} (\alpha)$. 
Suppose that $\alpha ' \not \in Z_{\infty}$. 
Then it follows that 
\[
\operatorname{ord}_{\alpha} (\mathfrak{n}'_{1,\overline{\mu}_{\gamma}}) = \frac{1}{r} \operatorname{ord}_{\alpha '} (\mathfrak{n}_{r,B'})
+ 
\operatorname{age}'(\gamma).
\]
\end{lem}
\begin{proof}
The same proof as in \cite[Lemma 4.7]{NS} works. 
\end{proof}

\begin{lem}\label{lem:thinset}
Let $I_Z \subset \mathcal{O}_A$ be the ideal sheaf defining $Z \subset A$. 
Let $J$ be one of the following ideal sheaves on $\overline{B}^{(\gamma)}$: 
\begin{align*}
I_Z \mathcal{O}_{\overline{B}^{(\gamma)}}, \quad
\operatorname{Jac}'_{\overline{B}^{(\gamma)}/k[[t]]}, \quad
\operatorname{Jac}'_{\widetilde{B}^{(\gamma)}/k[[t]]} \mathcal{O}_{\overline{B}^{(\gamma)}}, \\
\operatorname{Jac}'_{B'/k[t]} \mathcal{O}_{\overline{B}^{(\gamma)}}, \quad 
\mathfrak{n}'_{1,p}, \quad
\mathfrak{n}'_{1,\overline{\mu}_{\gamma}}, \quad
\mathfrak{n}_{r,B'} \mathcal{O}_{\overline{B}^{(\gamma)}}. 
\end{align*}
Let $W \subset \overline{B}^{(\gamma)}$ be the closed subscheme defined by $J$. 
Then $W_{\infty}$ is a thin subset of $\overline{B}^{(\gamma)}_{\infty}$.
\end{lem}
\begin{proof}
We set 
\begin{align*}
J_1 := I_Z \mathcal{O}_{\overline{B}^{(\gamma)}}, \quad
J_2 := \operatorname{Jac}'_{\overline{B}^{(\gamma)}/k[[t]]}, \quad
J_3 := \operatorname{Jac}'_{\widetilde{B}^{(\gamma)}/k[[t]]} \mathcal{O}_{\overline{B}^{(\gamma)}}, \\
J_4 := \operatorname{Jac}'_{B'/k[t]} \mathcal{O}_{\overline{B}^{(\gamma)}}, \quad 
J_5 := \mathfrak{n}'_{1,p}, \quad
J_6 := \mathfrak{n}'_{1,\overline{\mu}_{\gamma}}, \quad
J_7 := \mathfrak{n}_{r,B'} \mathcal{O}_{\overline{B}^{(\gamma)}}, 
\end{align*}
and we denote by $W_i \subset \overline{B}^{(\gamma)}$ the closed subscheme defined by $J_i$. 

Since $B \cap Z \subset B_{\rm sing}$, we have $(W_1)_{\rm red} \subset (W_4)_{\rm red}$. 
Since the map $\eta_r \colon (\Omega_{B/k} ^{\prime n})^{\otimes r} \to \omega _{B/k} ^{\prime [r]}$ 
in Definition \ref{defi:omega'}(5) is an isomorphism on the regular locus $B_{\rm reg}$, 
we have $(W_7)_{\rm red} \subset (W_4)_{\rm red}$. 
By Lemma \ref{lem:age2_R}, we have 
\[
(W_6)_{\infty} \cup (W_1)_{\infty} = (W_7)_{\infty} \cup (W_1)_{\infty}. 
\]
By Lemmas \ref{lem:L2_R} and \ref{lem:additive_k[[t]]}, we have 
\[
(W_5)_{\infty} \subset (W_2)_{\infty} \cup (W_3)_{\infty} \cup (W_6)_{\infty}. 
\]
Therefore, it is sufficient to show the assertion for
\[
J_2 = \operatorname{Jac}'_{\overline{B}^{(\gamma)}/k[[t]]}, \quad
J_3 = \operatorname{Jac}'_{\widetilde{B}^{(\gamma)}/k[[t]]} \mathcal{O}_{\overline{B}^{(\gamma)}}, \quad 
J_4 = \operatorname{Jac}'_{B'/k[t]} \mathcal{O}_{\overline{B}^{(\gamma)}}. 
\]

We set 
\begin{alignat*}{2}
S_1 &:= k[x_1, \ldots , x_N][[t]],& \quad S_2 :=& k[x_1, \ldots , x_N]((t)), \\
T_1 &:= k[x_1, \ldots , x_N]^{C_{\gamma}}[[t]],& \quad T_2 :=& k[x_1, \ldots , x_N]^{C_{\gamma}}((t)). 
\end{alignat*}

The ideal $J_2 = \operatorname{Jac}'_{\overline{B}^{(\gamma)}/k[[t]]}$ corresponds to the ideal 
\[
\mathfrak{r} := \mathcal{J}_{c} \Bigl( \overline{I}^{(\gamma)}; \operatorname{Der}_{k[[t]]}(S_1) \Bigr) + \overline{I}^{(\gamma)} \subset S_1
\]
of $S_1$. 
To show that $(W_2)_{\infty}$ is a thin set, it is sufficient to show $\operatorname{ht}(\mathfrak{r}S_2) \ge c+1$. 
Since $\overline{I}^{(\gamma)}S_2$ is a prime ideal of height $c$ by Lemma \ref{lem:isotrivial2}(3)(4), 
it is sufficient to show that $\mathfrak{r}S_2 \not \subset \overline{I}^{(\gamma)}S_2$. 
Since $S_2$ satisfies the weak Jacobian condition (WJ)$_{k((t))}$ over $k((t))$ by \cite[Theorem 46.3]{Nag62}, we have 
\[
\mathcal{J}_{c} \Bigl( \overline{I}^{(\gamma)}; \operatorname{Der}_{k[[t]]}(S_1) \Bigr) S_2 + \overline{I}^{(\gamma)}S_2
= \mathcal{J}_{c} \Bigl( \overline{I}^{(\gamma)}S_2 ; \operatorname{Der}_{k((t))}(S_2) \Bigr) + \overline{I}^{(\gamma)}S_2
\not \subset \overline{I}^{(\gamma)}S_2. 
\]
Here, the first equality follows from the fact that both $\operatorname{Der}_{k[[t]]}(S_1)$ and $\operatorname{Der}_{k((t))}(S_2)$ 
are generated by the derivations $\partial/\partial x_i$'s. 
We complete the proof of the assertion for $J_2$. 

Let $\mathfrak{r} \subset T_1$ be the ideal of $T_1$ corresponding to 
$\operatorname{Jac}'_{\widetilde{B}^{(\gamma)}/k[[t]]} \subset T_1/\widetilde{I}^{(\gamma)}$. 
Then it is sufficient to show $\operatorname{ht}(\mathfrak{r}S_2) \ge c+1$. 
Note that $\widetilde{I}^{(\gamma)} T_2 = \overline{I}^{(\gamma)} S_2 \cap T_2$ holds. 
Therefore, $\widetilde{I}^{(\gamma)} T_2$ is a prime ideal of height $c$. 
By the same argument as above, we have 
$\mathfrak{r} T_2 \not \subset \widetilde{I}^{(\gamma)} T_2$, 
and hence $\operatorname{ht}(\mathfrak{r}T_2) \ge c+1$. 
Since $T_2 \subset S_2$ is an integral extension, 
we have $\operatorname{ht}(\mathfrak{r}S_2) = \operatorname{ht}(\mathfrak{r}T_2) \ge c+1$, 
which completes the proof of the assertion for $J_3$. 

Let $\mathfrak{r} \subset k[[x_1, \ldots , x_N]]^G$ be the ideal corresponding to $\operatorname{Jac}'_{B/k} \subset k[[x_1, \ldots , x_N]]^G/I$. 
Since $J_4 = \mathfrak{r}\mathcal{O}_{\overline{B}^{(\gamma)}}$, it is sufficient to show $\operatorname{ht}(\mathfrak{r}S_2) \ge c+1$. 
We have $\operatorname{ht} \mathfrak{r} \ge c+2$ by the normality of $B$. 
Therefore, we have $\operatorname{ht}(\mathfrak{r}S_2) = \operatorname{ht} \mathfrak{r} \ge c+2$ by Lemma \ref{lem:isotrivial3}. 
We complete the proof of the assertion for $J_4$. 
\end{proof}

\begin{thm}\label{thm:mld_hyperquot_R}
Let $x = 0 \in B$ be the origin and let $\mathfrak{m}_x \subset \mathcal{O}_B$ be the corresponding maximal ideal. 
Let $\mathfrak{a} \subset  \mathcal{O}_B$ be a non-zero ideal sheaf and $\delta$ a positive real number. 
Then 
\begin{align*}
\operatorname{mld}_x(B,\mathfrak{a}^\delta) 
&= \inf _{w, b_1 \in \mathbb{Z}_{\ge 0}, \gamma \in G}
\bigl \{ 
\operatorname{codim} (C_{w, \gamma, b_1}) + \operatorname{age}'(\gamma)  - b_1-\delta w
\bigr \} \\
&= \inf _{w, b_1 \in \mathbb{Z}_{\ge 0}, \gamma \in G}
\bigl \{ 
\operatorname{codim} (C' _{w, \gamma, b_1}) + \operatorname{age}'(\gamma)  - b_1-\delta w
\bigr \}
\end{align*}
holds for 
\begin{align*}
C_{w, \gamma, b_1} 
& := \operatorname{Cont}^w  \bigl( \mathfrak{a} \mathcal{O}_{\overline{B}^{(\gamma)}} \bigr) \cap 
\operatorname{Cont}^{\ge 1} \bigl( \mathfrak{m}_x \mathcal{O}_{\overline{B}^{(\gamma)}} \bigr) \cap 
\operatorname{Cont}^{b_1} \bigl(\operatorname{Jac}'_{\overline{B}^{(\gamma)}/k[[t]]} \bigr), \\
C'_{w, \gamma, b_1}
& := \operatorname{Cont}^{\ge w}  \bigl(\mathfrak{a} \mathcal{O}_{\overline{B}^{(\gamma)}} \bigr) \cap 
\operatorname{Cont}^{\ge 1} \bigl(\mathfrak{m}_x \mathcal{O}_{\overline{B}^{(\gamma)}} \bigr) \cap 
\operatorname{Cont}^{b_1} \bigl(\operatorname{Jac}'_{\overline{B}^{(\gamma)}/k[[t]]} \bigr). 
\end{align*} 
\end{thm}
\begin{proof}
The formula for $C_{w, \gamma , b_1}$ is the formal power series ring version of \cite[Theorem 4.8]{NS}. 
The same proof as in \cite[Theorem 4.8]{NS} works. 
First, \cite[Theorem 7.4]{EM09} plays an important role in the proof of \cite[Theorem 4.8]{NS} and 
it can be substituted by Theorem \ref{thm:mld_R}. 
Furthermore, Propositions 2.25, 2.33, 2.35, 3.4 and 3.8, and Lemmas 2.10, 4.6 and 4.7 in \cite{NS}, 
which are also the key ingredients of the proof of \cite[Theorem 4.8]{NS}, are substituted by 
Propositions \ref{prop:negligible}, \ref{prop:EM6.2_k[[t]]}, \ref{prop:DL2_k[[t]]} and \ref{prop:DL_X} and 
Lemmas \ref{lem:additive_k[[t]]}, \ref{lem:L2_R} and \ref{lem:age2_R} in this paper. 

We also note that Proposition \ref{prop:negligible} is applied to $Z_{\infty}$ and $W_{\infty}$, 
where $W$ is the closed subscheme of $\overline{B}^{(\gamma)}$ corresponding to one of the following ideals: 
\begin{align*}
\operatorname{Jac}'_{\overline{B}^{(\gamma)}/k[[t]]}, \ \ 
\operatorname{Jac}'_{\widetilde{B}^{(\gamma)}/k[[t]]} \mathcal{O}_{\overline{B}^{(\gamma)}}, \ \ 
\operatorname{Jac}'_{B'/k[t]} \mathcal{O}_{\overline{B}^{(\gamma)}}, \ \ 
\mathfrak{n}'_{1,p}, \ \ 
\mathfrak{n}'_{1,\overline{\mu}_{\gamma}}, \ \ 
\mathfrak{n}_{r,B'} \mathcal{O}_{\overline{B}^{(\gamma)}}. 
\end{align*}
By Lemma \ref{lem:thinset}, they are actually thin sets. 

The formula for $C' _{w, \gamma , b_1}$ is the formal power series ring version of \cite[Corollary 4.9]{NS}, and 
the same proof works. 
\end{proof}

\begin{rmk}\label{rmk:multi_index}
Theorem \ref{thm:mld_hyperquot_R} can be easily extended to $\mathbb{R}$-ideals $\mathfrak{a} = \prod _{i=1} ^r \mathfrak{a}_i ^{\delta_i}$, 
where $\mathfrak a_1, \ldots, \mathfrak a_r$ are non-zero ideal sheaves on $B$ and $\delta_1, \ldots ,\delta_r$ are positive real numbers. 
In this setting, we have 
\begin{align*}
&\operatorname{mld}_x \bigl( B,\prod _{i=1} ^r \mathfrak{a}_i ^{\delta_i} \bigr) \\
&= 
\inf _{w_1, \ldots, w_r, b_1 \in \mathbb{Z}_{\ge 0}, \gamma \in G}
\left \{ 
\operatorname{codim} (C_{w_1, \ldots , w_r, \gamma, b_1}) + \operatorname{age}'(\gamma)  - b_1 - \sum _{i=1} ^r \delta_i w_i
\right \} \\
&=
\inf _{w_1, \ldots, w_r, b_1 \in \mathbb{Z}_{\ge 0}, \gamma \in G}
\left \{ 
\operatorname{codim} (C'_{w_1, \ldots , w_r, \gamma, b_1}) + \operatorname{age}'(\gamma)  - b_1 - \sum _{i=1} ^r \delta_i w_i
\right \}
\end{align*}
for 
\begin{align*}
C_{w_1, \ldots , w_r, \gamma, b_1} & := 
\Bigl( \bigcap _{i=1} ^{r} \operatorname{Cont}^{w_i} \bigl(\mathfrak{a}_i \mathcal{O}_{\overline{B}^{(\gamma)}} \bigr) \Bigr) \cap 
\operatorname{Cont}^{\ge 1} \bigl( \mathfrak{m}_x \mathcal{O}_{\overline{B}^{(\gamma)}} \bigr) \cap \operatorname{Cont}^{b_1} \bigl( \operatorname{Jac}'_{\overline{B}^{(\gamma)}/k[[t]]} \bigr), \\
C' _{w_1, \ldots , w_r, \gamma, b_1} & := 
\Bigl( \bigcap _{i=1} ^{r} \operatorname{Cont}^{\ge w_i} \bigl(\mathfrak{a}_i \mathcal{O}_{\overline{B}^{(\gamma)}} \bigr) \Bigr) \cap 
\operatorname{Cont}^{\ge 1} \bigl( \mathfrak{m}_x \mathcal{O}_{\overline{B}^{(\gamma)}} \bigr) \cap \operatorname{Cont}^{b_1} \bigl( \operatorname{Jac}'_{\overline{B}^{(\gamma)}/k[[t]]} \bigr). 
\end{align*}
\end{rmk}

\subsection{Properties on $\overline{B}^{\prime (\gamma)}$}

In the remainder of this section, 
we define a scheme $\overline{B}^{\prime (\gamma)}$ and investigate its properties, which will be used in Section \ref{section:PIA}. 

We denote by $\overline{\lambda}_{\gamma} ^{\prime *}$ the $k[t]$-ring homomorphism
\[
\overline{\lambda}_{\gamma} ^{\prime *}: k[t][[x_1, \ldots, x_N]]^{G}  \to  k[t][[x_1, \ldots, x_N]];\quad 
x_i \mapsto t^{\frac{e_i}{d}}x_i. 
\]
We set 
\begin{align*}
\overline{I}^{\prime (\gamma)} &:= \bigl( \overline{\lambda}^{\prime *} _{\gamma} (f_1),\ldots,\overline{\lambda}^{\prime *} _{\gamma} (f_c) \bigr) \subset k[t][[x_1, \ldots, x_N]],  \\
\overline{B}^{\prime (\gamma)} &:= \operatorname{Spec} \bigl( k[t][[x_1,\ldots,x_N]]/\overline{I}^{\prime (\gamma)} \bigr). 
\end{align*}
Then $\overline{B}^{\prime (\gamma)}$ is a scheme of finite type over $R = k[t][[x_1,\ldots,x_N]]$. 
Let $\Omega '_{\overline{B}^{\prime (\gamma)}/k[t]}$ be the sheaf of special differentials defined in Definition \ref{defi:omega'}(1) with respect to $R$ and $R_0 = k[t]$. 
We set
\[
\operatorname{Jac}'_{\overline{B}^{\prime (\gamma)}/k[t]} := \operatorname{Fitt}^n \bigl( \Omega '_{\overline{B}^{\prime (\gamma)}/k[t]} \bigr). 
\]

First, we study the dimensions of the irreducible components of $\overline{B}^{\prime (\gamma)}$. 
\begin{lem}\label{lem:isotrivial}
We denote $F_i := \overline{\lambda}^{\prime *} _{\gamma} (f_i)$ for each $1 \le i \le c$. 
Consider the following diagram of rings. 
\[
\xymatrix{
S_1 := k[t][[x_1, \ldots , x_N]] \ar@{^{(}->}[d]_{h_1} \ar@{->>}[r] & C_1 := k[t][[x_1, \ldots , x_N]]/(F_1, \ldots , F_c) \ar[d] _{g_1} \\ 
S_2 := k[t,t^{-1}][[x_1, \ldots , x_N]] \ar@{^{(}->}[d]_{h_2} \ar@{->>}[r] & C_2 := k[t,t^{-1}][[x_1, \ldots , x_N]]/(F_1, \ldots , F_c) \ar[d]_{g_2} \\ 
S_3 := k[t^{1/d},t^{-1/d}][[x_1, \ldots , x_N]] \ar@{->>}[r] & C_3 := k[t^{1/d},t^{-1/d}][[x_1, \ldots , x_N]]/(F_1, \ldots , F_c) \\
S_4 := k[t^{1/d},t^{-1/d}][[x_1, \ldots , x_N]] \ar[u]^{\simeq}_{x_i \mapsto t^{e_i/d}x_i} \ar@{->>}[r] & C_4 := k[t^{1/d},t^{-1/d}][[x_1, \ldots , x_N]]/(f_1, \ldots , f_c) \ar[u]^{\simeq}_{x_i \mapsto t^{e_i/d}x_i} \\
S_5 := k[[x_1, \ldots , x_N]] \ar@{^{(}->}[u]^{h_3} \ar@{->>}[r] & C_5 := k[[x_1, \ldots , x_N]]/(f_1, \ldots , f_c) \ar[u]^{g_3} \\ 
}
\]
We denote $I_i := (F_1, \ldots, F_c) \subset S_i$ for $i \in \{ 1,2,3 \}$, and 
$I_i := (f_1, \ldots, f_c) \subset S_i$ for $i \in \{ 4, 5 \}$. 
Then the following hold. 
\begin{enumerate}
\item $h_1$, $h_2$ and $h_3$ are regular, and hence so are $g_1$, $g_2$ and $g_3$. 
\item $h_2$ and $h_3$ are faithfully flat, and hence so are $g_2$ and $g_3$. 
\item $C_2$, $C_3$ and $C_4$ are normal domains. In particular, $I_2$, $I_3$ and $I_4$ are prime ideals. 
\item $\operatorname{ht}(I_2) = \operatorname{ht}(I_3) = \operatorname{ht}(I_4) = c$. 
\item There exists only one minimal prime $P$ of $I_1$ of the form (2) in Lemma \ref{lem:action}. 
Furthermore, it satisfies $\operatorname{ht} P = c$ and $P = I_2 \cap S_1$. 
\end{enumerate}
\end{lem}
\begin{proof}
We shall prove (1). We shall only see the regularity of $h_2$ since the other two can be proved more easily. 
Since the ring inclusion $k[t,t^{-1}] \to k[t^{1/d},t^{-1/d}]$ is \'{e}tale, so is its base change
\[
k[t,t^{-1}][[x_1, \ldots , x_N]] \longrightarrow 
k[t,t^{-1}][[x_1, \ldots , x_N]][t^{1/d},t^{-1/d}]. 
\]
Furthermore, the ring inclusion
\[
k[t,t^{-1}][[x_1, \ldots , x_N]][t^{1/d},t^{-1/d}]
\longrightarrow 
k[t^{1/d},t^{-1/d}][[x_1, \ldots , x_N]]
\]
is regular since it can be seen as the completion at the prime ideal $(x_1, \ldots , x_N)$ and the ring 
on the left side is an excellent ring, in particular a G-ring (cf.\ \cite[Theorem 79]{Mat80}). 
Therefore, their composition $h_2$ turns out to be regular. 

We shall prove (2) for $h_2$. 
Any maximal ideal $M$ of $S_2$ is of the form $M = (t - a, x_1, \ldots , x_N)$, where $a \in k^{\times}$. 
Therefore, we have $M S_3 \not = S_3$ and hence $h_2$ is faithfully flat. 
The same proof works for $h_3$. 

We shall prove (3). 
Note that $C_5$ is normal by our assumption. 
Therefore, the normality of $C_2$, $C_3$ and $C_4$ follows from 
(1), (2) and the fact that the normality is preserved under faithfully flat regular ring homomorphisms 
(cf.\ \cite[Theorem 32.2]{Mat89}). 
In what follows, we prove that $C_2$, $C_3$ and $C_4$ are domains. 
Since $h_2$ is faithfully flat, $g_2$ is injective (cf.\ \cite[Theorem 7.5]{Mat89}). 
Therefore, it is sufficient to show that $C_4$ is a domain. 
Let $P_1, \ldots, P_{\ell}$ be the minimal primes of $I_4$. 
Suppose the contrary that $\ell \ge 2$. 
Since $C_4$ is normal, we have $P_1 + P_2 = S_4$ (cf.\ \cite[Proposition 2.20]{Eis95}). 
Take maximal ideals $M_1$ and $M_2$ of $S_4$ such that $P_i \subset M_i$ for each $i \in \{ 1, 2 \}$. 
We may write $M_i = (t^{1/d} - a_i, x_1, \ldots , x_N)$ with $a_i \in k^{\times}$. 
For each $c \in k^{\times}$, we denote by $T_c$ the ring isomorphism 
\[
T_c : S_4 \to S_4; \qquad t^{1/d} \mapsto c t^{1/d}, \quad x_i \mapsto x_i.
\]
Then, $I_4$ is $T_c$-invariant for any $c \in k^{\times}$. 
Therefore, its minimal primes $P_1$ and $P_2$ are also $T_c$-invariant for any $c \in k^{\times}$. 
Therefore $P_1 \subset M_2$ holds, and hence $P_1 + P_2 \subset M_2$, a contradiction.

Note that $\operatorname{ht}(I_5) = c$ by our assumption. 
Therefore, (4) follows from (2) and Lemma \ref{lem:ff}. 

We shall prove (5). 
Let $P_1, \ldots, P_{\ell}$ be the minimal primes of $I_1$. 
Then $P_1 ^{a_1} \cap \cdots \cap P_{\ell}^{a_{\ell}} \subset I_1$ holds for some $a_1, \ldots , a_{\ell} \ge 1$. 
Since $h_1$ is flat, we have 
\[
I_2 =  I_1 S_2 \supset (P_1 ^{a_1} \cap \cdots \cap P_{\ell}^{a_{\ell}}) S_2 = P_1 ^{a_1} S_2 \cap \cdots \cap P_{\ell}^{a_{\ell}} S_2
\]
by \cite[Theorem 7.4]{Mat89}. 
If $P_i$ is of the form (1) or (3) in Lemma \ref{lem:action}, then we have $P_i S_2 = S_2$. 
Since $I_2 \not = S_2$, some $P_i$ has to be of the form (2) in Lemma \ref{lem:action}. 

Suppose that $P$ is a minimal prime of $I_1$ of the form (2) in Lemma \ref{lem:action}. 
Then $P + (t-a) \not = S_1$ holds for any $a \in k^{\times}$, and hence
we have $P S_2 \not = S_2$. 
Since we have 
\[
c = \operatorname{ht} I_2 \le \operatorname{ht} (PS_2) = \operatorname{ht} P \le c
\]
by Lemma \ref{lem:ff} and Krull's height theorem, 
it follows that $\operatorname{ht} P = c$ and $I_2 = PS_2$. 
Since we have
\[
c = \operatorname{ht} P \le \operatorname{ht} (I_2 \cap S_1) \le \operatorname{ht} I_2 = c
\]
by the going-down theorem, we have $P = I_2 \cap S_1$, which also shows the uniqueness of $P$. 
\end{proof}

\begin{rmk}
Let $\overline{B}^{\prime (\gamma)} = V_1 \cup \cdots \cup V_{\ell}$ be the irreducible decomposition. 
By Lemma \ref{lem:isotrivial}(5), $\overline{B}^{\prime (\gamma)}$ has 
the unique irreducible component $V$ of the form (2) in Lemma \ref{lem:action} and 
it satisfies $\dim V = \dim ' V = n+1$. 
Furthermore, any irreducible component $V'$ other than $V$ satisfies
\[
V'_{\infty} \cap \operatorname{Cont}^{\ge 1} \bigl( \mathfrak{o}_{\overline{B}^{\prime (\gamma)}} \bigr) = \emptyset
\] by Remark \ref{rmk:kpt}. 
Here, $\mathfrak{o}_{\overline{B}^{\prime (\gamma)}} \subset \mathcal{O}_{\overline{B}^{\prime (\gamma)}}$ 
denotes the ideal sheaf generated by $x_1, \ldots , x_N$. 
\end{rmk}

Next, we see the relationship between $\overline{B}^{(\gamma)}_{\infty}$ and $\overline{B}^{\prime (\gamma)}_{\infty}$. 

\begin{lem}\label{lem:RF}
Let $\mathfrak{o}_{\overline{B}^{(\gamma)}} \subset \mathcal{O}_{\overline{B}^{(\gamma)}}$ and 
$\mathfrak{o}_{\overline{B}^{\prime (\gamma)}} \subset \mathcal{O}_{\overline{B}^{\prime (\gamma)}}$
be the ideal sheaves generated by $x_1, \ldots , x_N$. 
Then the following hold. 
\begin{enumerate}
\item For $m \ge 0$, there exist canonical morphisms $\overline{B}^{\prime (\gamma)}_m \to \overline{B}^{(\gamma)}_m$ 
which commute with the truncation morphisms. 
In particular, they induce $\overline{B}^{\prime (\gamma)}_{\infty} \to \overline{B}^{(\gamma)}_{\infty}$. 

\item The map $\overline{B}^{\prime (\gamma)}_{\infty} \to \overline{B}^{(\gamma)}_{\infty}$ induces an isomorphism
\[
\overline{B}^{\prime (\gamma)}_{\infty} \cap 
\operatorname{Cont}^{\ge 1} \bigl( \mathfrak{o}_{\overline{B}^{\prime (\gamma)}} \bigr)  
\xrightarrow{\simeq} \overline{B}^{(\gamma)}_{\infty} \cap 
\operatorname{Cont}^{\ge 1} \bigl( \mathfrak{o}_{\overline{B}^{(\gamma)}} \bigr). 
\]

\item For a $k$-arc $\gamma \in \overline{B}^{\prime (\gamma)}_{\infty} \cap 
\operatorname{Cont}^{\ge 1} \bigl( \mathfrak{o}_{\overline{B}^{\prime (\gamma)}} \bigr)$
and the corresponding arc $\gamma ' \in \overline{B}^{(\gamma)}_{\infty} \cap 
\operatorname{Cont}^{\ge 1} \bigl( \mathfrak{o}_{\overline{B}^{(\gamma)}} \bigr)$, it follows that
\[
\operatorname{ord}_{\gamma} \Bigl( \operatorname{Jac}'_{\overline{B}^{\prime (\gamma)}/k[t]} \Bigr)
= \operatorname{ord}_{\gamma '} \Bigl( \operatorname{Jac}'_{\overline{B}^{(\gamma)}/k[[t]]} \Bigr). 
\]

\end{enumerate}
\end{lem}
\begin{proof}
(1) and (2) follow from the discussion in Remark \ref{rmk:RF}(2). 
(3) follows from Remarks \ref{rmk:jacobianmatrix} and \ref{rmk:hat}(\ref{item:Fitt_hat}). 
\end{proof}

\begin{lem}\label{lem:thin}
Let $W \subset \overline{B}^{\prime (\gamma)}$ be the closed subscheme defined by 
$\operatorname{Jac}'_{\overline{B}^{\prime (\gamma)}/k[t]}$. 
Then $W_{\infty} \cap \operatorname{Cont}^{\ge 1} \bigl( \mathfrak{o}_{\overline{B}^{\prime (\gamma)}} \bigr)$ 
is a thin subset of $\overline{B}^{\prime (\gamma)}_{\infty}$. 
\end{lem}
\begin{proof}
Let $T_c : k[t][[x_1, \ldots, x_N]] \to k[t][[x_1, \ldots, x_N]]$ be the ring isomorphism defined in Lemma \ref{lem:action}.

Set $J := \operatorname{Jac}'_{\overline{B}^{\prime (\gamma)}/k[t]}$. 
Let $\overline{J} \subset k[t][[x_1, \ldots, x_N]]$ be the corresponding ideal. 
Since $\overline{I}^{\prime (\gamma)}$ is $T_c$-invariant, 
$\overline{J}$ is also $T_c$-invariant for each $c \in k^{\times}$. 
Therefore by Lemma \ref{lem:action}, each minimal prime $P$ of $\overline{J}$ satisfies one of the conditions in Lemma \ref{lem:action}. 
By Remark \ref{rmk:kpt}, it is sufficient to show that 
$\operatorname{ht} P \ge c+1 = N - n + 1$ for $P$ satisfying (2) in Lemma \ref{lem:action}. 
Since $P + (t-1) \not = k[t][[x_1, \ldots, x_N]]$ in this case, it is sufficies to show that 
\[
\operatorname{ht}(P + (t-1)) \ge N - n + 2. 
\]
Under the identification $k[t][[x_1, \ldots, x_N]]/(t-1) \simeq k[[x_1, \ldots, x_N]]$, 
the ideal $\bigl( \overline{J} + (t-1) \bigr)/(t-1)$ corresponds to $\operatorname{Jac}'_{\overline{B}/k}$. 
Since $\overline{B}$ is normal and hence regular in codimension one, 
it follows that 
\[
\operatorname{ht} \bigl( \overline{J} + (t-1) \bigr) \ge N - n + 3
\] 
by the Jacobian criterion of regularity. 
Therefore, we get the desired inequality for $P$. 
\end{proof}

\section{PIA formula for quotient singularities of non-linear action}\label{section:PIA}
In this section, we generalize Theorem 5.1 in \cite{NS} to non-linear group actions (Theorem \ref{thm:PIA_nonlin}). 
First, we clarify the definition of quotient singularities in this paper. 

\begin{defi}\label{defi:qs}
Let $X$ be a variety over $k$ and $x \in X$ a closed point. 
We say that $X$ has a \textit{quotient singularity} at $x$ 
if there exist a quasi-projective variety $\overline{M}$ over $k$, a finite subgroup $G \subset \operatorname{Aut} (\overline{M})$, 
and a smooth closed point $\overline{y} \in \overline{M}$ such that 
$\mathcal{O}_{X,x} \simeq \mathcal{O}_{M,y}$ holds, where $M:= \overline{M}/G$ is the quotient variety and $y \in M$ is the image of $\overline{y}$. 
\end{defi}

\begin{thm}\label{thm:PIA_nonlin}
Suppose that a variety $X$ has a quotient singularity at a closed point $x \in X$. 
Let $Y$ be a subvariety of $X$ of codimension $c$ that is locally defined by $c$ equations 
$h_1, \ldots, h_c \in \mathfrak{m}_{X,x}$ at $x$. 
Suppose that $Y$ is klt at $x$. 
Let $\mathfrak{a} \subset \mathcal{O}_X$ be an ideal sheaf and 
let $\delta$ be a positive real number. 
Suppose that $\mathfrak{b} := \mathfrak{a} \mathcal{O}_{Y} \not = 0$. 
Then it follows that
\[
\operatorname{mld}_x \bigl(X,(h_1 \cdots h_c) \mathfrak{a} ^{\delta} \bigr) 
= \operatorname{mld}_x (Y,\mathfrak{b}^{\delta}). 
\]
\end{thm}

\begin{proof}
Since $X$ has a quotient singularity at $x$, 
there exist a variety $\overline{M}$ with a smooth closed point $\overline{y} \in \overline{M}$ and 
a finite subgroup $G' \subset \operatorname{Aut} (\overline{M})$ such that 
$\mathcal{O}_{X,x} \simeq \mathcal{O}_{M',y'}$ holds, where $M' := \overline{M} / G'$ and $y' \in M'$ is the image of $\overline{y}$. 

We denote $G := \{ g \in G' \mid g(\overline{y}) = \overline{y} \}$ the stabilizer group of $\overline{y}$, 
$M := \overline{M}/G$ the quotient variety, and $y \in M$ the image of $\overline{y}$. 
Then we note that $M = \overline{M}/G \to M' = \overline{M}/G'$ is \'{e}tale at $y$ (cf.\ \cite[3.17]{Kol13}). 
Furthermore, $G$ acts on $\mathfrak{m}_{\overline{M}, \overline{y}} ^i$ for each $i \ge 0$, and hence the projection 
$s: \mathfrak{m}_{\overline{M}, \overline{y}} \to \mathfrak{m}_{\overline{M}, \overline{y}}/ \mathfrak{m}_{\overline{M}, \overline{y}} ^2$ 
becomes a $G$-equivariant $k$-linear map. 
Let $u$ be any $k$-linear section of $s$. 
Then the map $u' := \frac{1}{|G|} \sum _{g \in G} g^{-1} \circ u \circ g$ 
gives a $G$-equivariant $k$-linear section 
$u' : \mathfrak{m}_{\overline{M}, \overline{y}}/ \mathfrak{m}_{\overline{M}, \overline{y}} ^2 \to \mathfrak{m}_{\overline{M}, \overline{y}}$ of $s$. 
Let $N := \dim X$. Then $u'$ induces a ring homomorphism 
\[
k[x_1, \ldots, x_N] \to \mathcal{O}_{\overline{M}, \overline{y}}
\]
which is \'{e}tale. 
Furthermore, $G$ acts linearly on $k[x_1, \ldots, x_N]$ and the ring homomorphism above becomes $G$-equivariant. 
Since the ring homomorphism above is \'{e}tale, 
we get an isomorphism 
\[
k[[x_1, \ldots, x_N]] \overset{\simeq}{\longrightarrow} \widehat{\mathcal{O}}_{\overline{M}, \overline{y}}. 
\]
Note that $k[[x_1, \ldots, x_N]]^{G _{\rm pr}} \simeq k[[x_1, \ldots, x_N]]$ holds for 
$G _{\rm pr} \subset G$, where $G _{\rm pr}$ is the subgroup generated by the pseudo-reflections (cf.\ \cite[3.18]{Kol13}). 
Hence $\overline{M}/G_{\rm pr}$ is smooth at the image of $\overline{y}$. 
Therefore, by replacing $G$ with $G/G_{\rm pr}$ and $\overline{M}$ with $\overline{M}/G_{\rm pr}$, 
we may assume that $G$ does not contain a pseudo-reflection. 
Then we have the following diagram of rings. 
\[
\xymatrix{
\mathcal{O}_{X,x} \ar@{}[d]|{\text{\large \rotatebox{-90}{$\simeq$}}} && \\ 
\mathcal{O}_{M',y'} \ar[d]^{\text{\'{e}tale}} && \\
\mathcal{O}_{M,y} \ar@{}[d]|{\text{\large \rotatebox{-90}{$=$}}} \ar@{}[r]|{\text{\large $\subset$}} & \widehat{\mathcal{O}}_{M,y} \ar@{}[d]|{\text{\large \rotatebox{-90}{$=$}}} & \\
\mathcal{O}_{\overline{M}, \overline{y}} ^{G} \ar@{}[d]|{\text{\large \rotatebox{-90}{$\subset$}}} \ar@{}[r]|{\text{\large $\subset$}} & \widehat{\mathcal{O}}_{\overline{M}, \overline{y}} ^{G} \ar@{}[d]|{\text{\large \rotatebox{-90}{$\subset$}}} \ar@{}[r]|{\text{\large $\simeq$ \hspace{7mm}}} &  k[[x_1, \ldots , x_N]] ^G \\
\mathcal{O}_{\overline{M}, \overline{y}} \ar@{}[r]|{\text{\large $\subset$}} & \widehat{\mathcal{O}}_{\overline{M}, \overline{y}} \ar@{}[r]|{\text{\large $\simeq$ \hspace{7mm}}} & k[[x_1, \ldots , x_N]] \ar@{}[u]|{\text{\large \rotatebox{-90}{$\subset$}}}
}
\]

We denote by $(\overline{N}, \overline{y}) \subset (\overline{M}, \overline{y})$ the germ defined by the images of 
$h_1, \ldots , h_c \in \mathfrak{m}_{X,x}$ in $\mathcal{O}_{\overline{M}, \overline{y}}$. 
Let $f_1, \ldots , f_c \in k[[x_1, \ldots, x_N]]^G$ be the images of $h_1, \ldots , h_c \in \mathfrak{m}_{X,x}$. 
Then we have an isomorphism
\[
\widehat{\mathcal{O}}_{\overline{N}, \overline{y}} \simeq k[[x_1, \ldots, x_N]]/(f_1, \ldots , f_c). 
\]
We set
\[
A := \operatorname{Spec} k[[x_1, \ldots, x_N]]^G, \quad 
B := \operatorname{Spec} \bigl( k[[x_1, \ldots, x_N]]^G/(f_1, \ldots , f_c) \bigr), 
\]
and $x' \in A$ the origin. 
Let $\mathfrak{a} ^\prime \subset \mathcal{O}_A$ and $\mathfrak{b} ^\prime \subset \mathcal{O}_B$ 
be the ideal sheaves corresponding to $\mathfrak{a} \subset \mathcal{O}_X$ and $\mathfrak{b} \subset \mathcal{O}_Y$.
Since 
\[
\operatorname{mld}_x \bigl(X,(h_1 \cdots h_c) \mathfrak{a} ^{\delta} \bigr) 
= \operatorname{mld}_{x'} \bigl(A,(f_1 \cdots f_c) \mathfrak{a} ^{\prime \delta} \bigr), \quad 
\operatorname{mld}_x (Y,\mathfrak{b}^{\delta})
= \operatorname{mld}_{x'} (B,\mathfrak{b}^{\prime \delta})
\]
hold (cf.\ Remark \ref{rmk:compareK}), 
it is sufficient to show that 
\[
\operatorname{mld}_{x'} \bigl(A,(f_1 \cdots f_c) \mathfrak{a} ^{\prime \delta} \bigr) = \operatorname{mld}_{x'} (B,\mathfrak{b}^{\prime \delta}). 
\tag{$\diamondsuit$}
\]

Note that $G$, $A$ and $B$ satisfy all assumptions in Section \ref{section:mld_R}. 
Therefore we take over all notation. 
Since $Y$ is klt at $x$, so is $B$. 
Since $\overline{B} \to B$ is \'{e}tale in codimension one, 
it follows that the germ $(\overline{N}, \overline{y})$ is also klt. 
This fact will be used in the proof of Claim \ref{claim:not_thin_F}.

($\diamondsuit$) is proved for $k$-varieties in \cite[Theorem 5.1]{NS}. 
The key ingredients of the proof of \cite[Theorem 5.1]{NS} are Corollary 4.9, Lemma 2.34 and Claim 5.2 in \cite{NS}. 
First, Corollary 4.9 in \cite{NS} can be substituted by Theorem \ref{thm:mld_hyperquot_R} in the formal power series ring setting. 
Second, Lemma 2.34 in \cite{NS} is still true in our setting by replacing 
$k[t][x_1, \ldots , x_N]$ with $k[x_1, \ldots , x_N][[t]]$ (cf.\ Remark \ref{rmk:NS2.6_k[[t]]}). 
On the other hand, the proof of Claim 5.2 in \cite{NS} does not work directly because 
they use the result on the rational connectedness proved by Hacon and McKernan \cite{HM07}, 
which is not clear for the formal power series ring setting. 
We also note the lack of \cite[Lemma 2.27]{NS} in our setting. 
It will be substituted by Proposition \ref{prop:resol}. 

In what follows, we shall only prove Claim 5.2 in \cite{NS} in the formal power series ring setting.

\begin{claim}[{cf.\ \cite[Claim 5.2]{NS}}]\label{claim:not_thin_F}
Let $C \subset \overline{A}^{(\gamma)} _{\infty}$ be a cylinder that is the intersection of finitely many cylinders of the form 
$\operatorname{Cont}^{\ge \ell} \bigl( \mathfrak{c} \mathcal{O}_{\overline{A}^{(\gamma)}} \bigr)$, 
where $\mathfrak{c} \subset \mathcal{O}_A$ is an ideal sheaf on $A$ and $\ell$ is a non-negative integer. 
Let $C'$ be an irreducible component of $C$. 
Then $C' \cap \overline{B}^{(\gamma)}_\infty$ contains a $k$-arc $\delta$ such that
$\operatorname{ord}_{\delta} \bigl( \operatorname{Jac}'_{\overline{B}^{(\gamma)}/k[[t]]} \bigr) < \infty$. 
\end{claim}
\begin{proof}[Proof of Claim \ref{claim:not_thin_F}]
First, we introduce a $k$-action on the arc space $\overline{A}^{(\gamma)} _{\infty}$ as follows. 
For a $k$-arc $\alpha \in \overline{A}^{(\gamma)} _{\infty}$, we denote $g_{i}^{\alpha} := \alpha ^* (x_i) \in k[[t]]$, 
where $\alpha ^* : k[x_1, \ldots , x_N][[t]] \to k[[t]]$ is the corresponding $k[t]$-ring homomorphism. 
Let $\alpha \in \overline{A}^{(\gamma)} _{\infty}$ and $a \in k$. Then we define $a \cdot \alpha \in \overline{A}^{(\gamma)} _{\infty}$ by 
\[
g_{i}^{a \cdot \alpha} (t) := a^{e_i} g_i ^{\alpha} (a^d t). 
\]
Then for $f \in k[[x_1, \ldots , x_N]]^G$, we have $v(t) = u(a^d t)$ for
\[
u(t) := \alpha^* \bigl( \overline{\lambda}^*_{\gamma}(f) \bigr), \quad 
v(t) := (a \cdot \alpha)^* \bigl( \overline{\lambda}^*_{\gamma}(f) \bigr) \in k[[t]]. 
\]
Therefore, we have
\[
\operatorname{ord}_{\alpha} \bigl( \overline{\lambda}^*_{\gamma}(f) \bigr) = \operatorname{ord}_{a \cdot \alpha} \bigl( \overline{\lambda}^*_{\gamma}(f) \bigr)
\]
if $a \in k^{\times}$. 
Hence, any cylinder of the form $\operatorname{Cont}^{\ge \ell} \bigl( \mathfrak{c} \mathcal{O}_{\overline{A}^{(\gamma)}} \bigr)$ with 
an ideal $\mathfrak{c} \subset \mathcal{O}_A = k[[x_1, \ldots , x_N]]^G$ is invariant under the $k$-action. 
Therefore, $C$ in the statement and its irreducible component $C'$ are also invariant under the $k$-action. 

We denote by $\beta \in \overline{A}^{(\gamma)} _{\infty}$ the trivial arc determined by $g_i^{\beta} = 0$ for each $i$. 
Note that $\beta = 0 \cdot \alpha$ holds for any $k$-arc $\alpha \in \overline{A}^{(\gamma)} _{\infty}$. 
Therefore, we have $\beta \in C'$ and hence 
\[
\beta \in C' \cap \operatorname{Cont}^{\ge 1} \bigl( \mathfrak{o}_{\overline{B}^{(\gamma)}} \bigr) \not = \emptyset. 
\]
By Lemma \ref{lem:RF}(1)(2), there exists a cylinder $D \subset \overline{B}^{\prime (\gamma)}_{\infty}$ 
that is isomorphic to $C' \cap \operatorname{Cont}^{\ge 1} \bigl( \mathfrak{o}_{\overline{B}^{(\gamma)}} \bigr)$ under 
the map in Lemma \ref{lem:RF}(2). 
Then by Lemma \ref{lem:RF}(3), 
it is sufficient to show that $D$ contains a $k$-arc $\delta$ such that 
$\operatorname{ord}_{\delta} \bigl( \operatorname{Jac}'_{\overline{B}^{\prime (\gamma)}/k[t]} \bigr) < \infty$. 
Therefore, Claim \ref{claim:not_thin_F} follows from Claim \ref{claim:not_thin_R} below and Lemma \ref{lem:thin}. 
\end{proof}

\begin{claim}\label{claim:not_thin_R}
Let $D \subset \overline{B}^{\prime (\gamma)}_{\infty}$ be a cylinder contained in 
$\operatorname{Cont}^{\ge 1} \bigl( \mathfrak{o}_{\overline{B}^{\prime (\gamma)}} \bigr)$. 
If $D$ contains the trivial arc $\beta$, 
then $D$ is not a thin subset of $\overline{B}^{\prime (\gamma)}_{\infty}$. 
\end{claim}
\begin{proof}[Proof of Claim \ref{claim:not_thin_R}]
Let $T \subset \overline{B}^{\prime (\gamma)}$ be the closed subscheme defined 
by $\mathfrak{o}_{\overline{B}^{\prime (\gamma)}}$. 
First, we prove the following claim. 
\begin{itemize}
\item[($\spadesuit$)] $\overline{B}^{\prime (\gamma)}$ has a desingularization 
$r: W \to \overline{B}^{\prime (\gamma)}$ with the following conditions. 
\begin{enumerate}
\item For each $a \in k ^{\times}$, the closed subscheme $W_a \subset W$ defined by the ideal 
$(t-a)\mathcal{O}_{W} \subset \mathcal{O}_{W}$ 
is an integral regular scheme with $\dim' W_a = n$. 

\item $r|_{T'}: T' \to T$ has a section, where $T' := r^{-1}(T)$. 
\end{enumerate}
\end{itemize}

We set 
\[
F_i := \overline{\lambda}^{\prime *} _{\gamma} (f_i) = f_i \bigl( t^{\frac{e_1}{d}} x_1, \ldots , t^{\frac{e_N}{d}} x_N \bigr). 
\]
Then, we have the following natural morphisms. 
\[
\xymatrix{
V_1 := \overline{B}^{\prime (\gamma)} = \operatorname{Spec} \bigl( k[t][[x_1, \ldots , x_N]]/(F_1, \ldots , F_c) \bigr) \\
V_2 := \operatorname{Spec} \bigl( k[t,t^{-1}][[x_1, \ldots , x_N]]/(F_1, \ldots , F_c) \bigr) \ar[u] \\ 
V_3 := \operatorname{Spec} \bigl( k[t^{1/d},t^{-1/d}][[x_1, \ldots , x_N]]/(f_1, \ldots , f_c) \bigr) \ar[d] \ar[u] _{[x_i \mapsto t^{-e_i/d}x_i]} \\
V_4 := \operatorname{Spec} \bigl( \mathcal{O}_{\overline{N},\overline{y}}[t^{1/d},t^{-1/d}] \bigr) \ar[d] \\ 
V_5 := \operatorname{Spec} \bigl( \mathcal{O}_{\overline{N},\overline{y}} \bigr)
}
\]
Note that these four morphisms are regular morphisms (cf.\ Lemma \ref{lem:isotrivial}(1)). 
Hence by the functorial desingularization by Temkin \cite[Theorem 1.2.1]{Tem12}, 
there exist desingularizations $r_i: W_i \to V_i$ with the following Cartesian diagram. 
\[
\xymatrix{
W_1 \ar[r]^-{r_1} & V_1 := \operatorname{Spec} \bigl( k[t][[x_1, \ldots , x_N]]/(F_1, \ldots , F_c) \bigr) \\
W_2 \ar[r]^-{r_2}\ar[u] & V_2 := \operatorname{Spec} \bigl( k[t,t^{-1}][[x_1, \ldots , x_N]]/(F_1, \ldots , F_c) \bigr) \ar[u] \\ 
W_3 \ar[r]^-{r_3}\ar[d]\ar[u] & V_3 := \operatorname{Spec} \bigl( k[t^{1/d},t^{-1/d}][[x_1, \ldots , x_N]]/(f_1, \ldots , f_c) \bigr) \ar[d]\ar[u] \\
W_4 \simeq W_5 \times (\mathbb{A}^1_k \setminus \{ 0 \}) \ar[r]^-{r_4}\ar[d] & V_4 := \operatorname{Spec} \bigl( \mathcal{O}_{\overline{N},\overline{y}}[t^{1/d},t^{-1/d}] \bigr) = V_5 \times (\mathbb{A}^1_k \setminus \{ 0 \}) \ar[d]  \\
W_5 \ar[r]^-{r_5} & V_5 := \operatorname{Spec} \bigl( \mathcal{O}_{\overline{N},\overline{y}} \bigr)
}
\]
We shall prove that $r_1$ satisfies the conditions (1) and (2) in ($\spadesuit$). 

We shall prove (1). 
For  each $i \in \{ 1,2 \}$ and $a \in k^{\times}$, we denote by 
$(W_i)_a \subset W_i$ and $(V_i)_a \subset V_i$ the closed subschemes 
defined by $(t - a)\mathcal{O}_{W_i} \subset \mathcal{O}_{W_i}$ and 
$(t - a)\mathcal{O}_{V_i} \subset \mathcal{O}_{V_i}$, respectively. 
Similarly, for each $i \in \{ 3,4 \}$ and $a \in k^{\times}$, we denote by $(W_i)_a \subset W_i$ and $(V_i)_a \subset V_i$ the closed subschemes 
defined by $(t^{1/d} - a)\mathcal{O}_{W_i} \subset \mathcal{O}_{W_i}$ and 
$(t^{1/d} - a)\mathcal{O}_{V_i} \subset \mathcal{O}_{V_i}$, respectively. 
Then, the Cartesian diagram above induces the following Cartesian diagram for each $a \in k^{\times}$. 
\[
\xymatrix{
(W_1)_a \ar[r] & (V_1)_a \\
(W_2)_a \ar[r] \ar[u] &  (V_2)_a \ar[u] \\
\bigsqcup_{b^d = a} (W_3)_b \ar[r] \ar[d] \ar[u] & \bigsqcup_{b^d = a} (V_3)_b \ar[d] \ar[u] \\
\bigsqcup_{b^d = a} (W_4)_b  \ar[r] & \bigsqcup_{b^d = a} (V_4)_b
}
\]
Here, by construction, we have $(V_1)_a \simeq (V_2)_a \simeq (V_3)_b$ for any $a, b \in k^{\times}$ with $b^d = a$. 
Therefore, we also have $(W_1)_a \simeq (W_2)_a \simeq (W_3)_b$. 
Since $W_4 \simeq W_5 \times (\mathbb{A}^1_k \setminus \{ 0 \})$, 
$(W_4)_b$ turns out to be regular for any $b \in k^{\times}$. 
Since $(W_3)_b \to (W_4)_b$ is a regular morpshism, $(W_3)_b$ is also regular (cf.\ \cite[Theorem 32.2]{Mat89}). 
Note that the morphism $(V_3)_b \to (V_4)_b \simeq V_5$ is isomorphic to the completion map
\[
\overline{B} = \operatorname{Spec} \bigl( \widehat{\mathcal{O}}_{\overline{N}, \overline{y}} \bigr) \to 
\operatorname{Spec} \bigl( \mathcal{O}_{\overline{N}, \overline{y}} \bigr) = V_5. 
\]
Since $W_5 \to V_5$ is a birational map, 
so is $(W_3)_b \to (V_3)_b$. 
Since $(V_3)_b \simeq \overline{B}$ is integral, so is $(W_3)_b$. 
Furthermore, we have $\dim' (W_3)_b = \dim' (V_3)_b = \dim (V_3)_b = n$. 
Therefore, for any $a \in k^{\times}$, $(W_1)_a$ is an integral regular scheme with $\dim' (W_1)_a = n$.

We shall prove (2). 
Let $T_i \subset V_i$ be the closed subschemes defined by the ideals $(x_1, \ldots, x_N) \mathcal{O}_{V_i}$ for $i \in \{ 1,2,3 \}$ and 
by the ideal $\mathfrak{m}_{\overline{N},\overline{y}} \mathcal{O}_{V_4}$ for $i=4$. 
Let $T_i' := r_i ^{-1}(T_i)$. 
Then, we have the following Cartesian diagram.
\[
\xymatrix{
T'_1 \ar[r] & T_1 \simeq \mathbb{A}^1_k   \\
T'_2 \ar[r] \ar[u] & T_2 \ar@{^{(}->}[u] \simeq \mathbb{A}^1_k \setminus \{ 0 \} \\
T'_3 \ar[r] \ar[d] \ar[u] & T_3 \ar[d]^{\simeq} \ar@{->>}[u] \simeq \mathbb{A}^1_k \setminus \{ 0 \} \\
T'_4 \ar[r] & T_4 \simeq \mathbb{A}^1_k \setminus \{ 0 \}
}
\]
Since the above diagram forms a Cartesian diagram, any closed fiber of $T'_2 \to T_2$ is isomorphic to some fiber of $T'_4 \to T_4$. 
Since the germ $(\overline{N}, \overline{y})$ is klt, so is $V_4$. 
Therefore, $T'_4 \to T_4$ has rationally connected fibers by \cite[Corollary 1.7(1)]{HM07}, 
and so does $T'_2 \to T_2$. Therefore $T'_2 \to T_2$ has a section by \cite{GHS03}. 
Hence by the properness of $r_1$, the morphism $T'_1 \to T_1$ also has a section. 
We have proved the claim ($\spadesuit$). 

Note that $T_{\infty} = \{ \beta \}$ consists of only one arc $\beta$. 
By claim ($\spadesuit$), there exists $\beta' \in T'_{\infty} \subset W_{\infty}$ such that $r_{\infty} (\beta') = \beta$. 
Suppose the contrary that $D$ is a thin subset of $\overline{B}^{\prime (\gamma)}_{\infty}$. 
Since $\beta ' \in r^{-1}_{\infty} (D)$ satisfies $\operatorname{ord}_{\beta '} (\mathfrak{o}_{W}) = \infty$, 
to get a contradiction by Proposition \ref{prop:resol}, 
it is sufficient to show that $r^{-1}_{\infty} (D)$ is also a thin subset of $W_{\infty}$.

By Lemma \ref{lem:isotrivial}(5), there exists the unique irreducible component $U$
of $\overline{B}^{\prime (\gamma)}$ of the form (2) in Lemma \ref{lem:action}. 
We also note that $U'_{\infty} \cap D = \emptyset$ holds for any other irreducible component $U'$ of $\overline{B}^{\prime (\gamma)}$ since 
$D \subset \operatorname{Cont}^{\ge 1} \bigl( \mathfrak{o}_{\overline{B}^{\prime (\gamma)}} \bigr)$ (cf.\ Remark \ref{rmk:kpt}). 
Therefore, since $D$ is a thin set, there exists a closed subscheme $F \subset U$ such that $D \subset F_{\infty}$ and $\dim F \le n$. 
Let $W'$ be the connected component of $W$ that dominates $U$. We note that $\dim W' = \dim U = n+1$ 
by Lemma \ref{lem:isotrivial}(5). 
We set $F' := r^{-1} (F) \cap W'$. 
Then we have $r^{-1}_{\infty} (D) \subset F'_{\infty}$ and $\dim F' \le n$, 
and hence $r^{-1}_{\infty} (D)$ is also a thin set. 
We complete the proof of Claim \ref{claim:not_thin_R}. 
\end{proof}
We complete the proof of Theorem \ref{thm:PIA_nonlin}. 
\end{proof}

\begin{rmk}
Theorem \ref{thm:PIA_nonlin} can be generalized to $\mathbb{R}$-ideals due to Remark \ref{rmk:multi_index}. 
We have
\[
\operatorname{mld}_x \bigl(X,(h_1 \cdots h_c) \mathfrak{a}  \bigr) 
= \operatorname{mld}_x (Y,\mathfrak{b}).
\]
for an $\mathbb{R}$-ideal $\mathfrak{a}$ on $X$ and $\mathfrak{b} := \mathfrak{a} \mathcal{O}_Y$. 
\end{rmk}

\section{Proof of the main theorems}\label{section:maintheorem}

As a corollary of Theorem \ref{thm:PIA_nonlin}, 
we prove the PIA conjecture for quotients of locally complete intersection singularities. 
\begin{cor}\label{cor:PIA_general}
Suppose that a variety $X$ has a quotient singularity at a closed point $x \in X$. 
Let $Y$ be a subvariety of $X$ of codimension $c$ that is locally defined by $c$ equations at $x$. 
Suppose that $Y$ is klt at $x$. 
Let $\mathfrak{a}$ be an $\mathbb{R}$-ideal sheaf on $Y$. 
Let $D$ be a prime divisor on $Y$ through $x$ that is klt and Cartier at $x$. 
Suppose that $D$ is not contained in the cosupport of the $\mathbb{R}$-ideal sheaf $\mathfrak{a}$. 
Then it follows that
\[
\operatorname{mld}_x \bigl( Y, \mathfrak{a} \mathcal{O}_Y (-D) \bigr) = 
\operatorname{mld}_x (D, \mathfrak{a} \mathcal{O}_D). 
\]
\end{cor}
\begin{proof}
Take an $\mathbb{R}$-ideal sheaf $\mathfrak{b}$ on $X$ such that $\mathfrak{a} = \mathfrak{b} \mathcal{O}_Y$, 
and take local equations $h_1, \ldots , h_c \in \mathcal{O}_{X,x}$ of $Y$ in $X$ at $x$. 
Furthermore, take $g \in \mathcal{O}_{X,x}$ such that its image $\overline{g} \in \mathcal{O}_{Y,x}$ defines $D$. 
Then we have
\[
\operatorname{mld}_x \bigl( Y, \mathfrak{a} \mathcal{O}_Y (-D) \bigr) = 
\operatorname{mld}_x \bigl( X, (h_1 \cdots h_{c} \cdot g)\mathfrak{b} \bigr) =
\operatorname{mld}_x ( D, \mathfrak{a} \mathcal{O}_D )
\]
by applying Theorem \ref{thm:PIA_nonlin} twice. 
\end{proof}

\begin{thm}\label{thm:LSC_general}
Let $X$ be a variety with only quotient singularities. 
Let $Y$ be a klt subvariety of $X$ of codimension $c$ that is locally defined by $c$ equations in $X$. 
Let $\mathfrak{a}$ be an $\mathbb{R}$-ideal sheaf on $Y$. 
Then the function 
\[
|Y| \to \mathbb{R}_{\ge 0} \cup \{ - \infty \}; \quad x \mapsto \operatorname{mld}_x(Y,\mathfrak{a})
\]
is lower semi-continuous, where we denote by $|Y|$ the set of all closed points of $Y$ with the Zariski topology. 
\end{thm}
\begin{proof}
We take over the notation in the proof of Corollary \ref{cor:PIA_general}. 
Then we have 
\[
\operatorname{mld}_x(Y,\mathfrak{a}) = \operatorname{mld}_x \bigl( X, (h_1 \cdots h_{c})\mathfrak{b} \bigr) 
\]
by Theorem \ref{thm:PIA_nonlin}. 
Then the assertion follows from the fact proved in \cite[Corollary 1.3]{Nak16}
that the lower semi-continuity holds for the variety $X$ with only quotient singularities. 
\end{proof}


\end{document}